\documentclass[11pt]{article}
\usepackage{makecell}
\usepackage{arxiv}
\usepackage{xcolor}
\usepackage[utf8]{inputenc} % allow utf-8 input
\usepackage{hyperref, subcaption}    % hyperlinks
\usepackage{diagbox}
\usepackage{url}            % simple URL typesetting
\usepackage{booktabs}
\usepackage{tabularx}
\usepackage{array}

\usepackage{amsfonts}       % blackboard math symbols
\usepackage{subcaption,booktabs}
\usepackage{nicefrac}       % compact symbols for 1/2, etc.
\usepackage{multirow}
\usepackage{amsmath, amsthm, amssymb, hyperref}		% Can be removed after putting your text content

\usepackage{tabularx}

\usepackage{graphicx}
\usepackage{doi}

\usepackage{siunitx}

\newtheorem{definition}{Definition}[section]
\newtheorem{theorem}{Theorem}[section]
\newtheorem{lemma}{Lemma}[section]
\newtheorem{corollary}{Corollary}[section]
\newtheorem{remark}{Remark}[section]
\newcommand{\review}[1]{\textcolor{black}{#1}}
\newcommand{\revmark}[1]{\textcolor{black}{#1}}
\newcommand{\correction}[1]{\textcolor{black}{#1}}

\title{Artifacts of Numerical Integration in Learning Dynamical Systems}

\author{
  Bing‑Ze Lu \\
  National Chung Cheng University, Taiwan
  \AND
  Richard Tsai \\
  The University of Texas at Austin, USA
}

\date{}

\renewcommand{\headeright}{}
\renewcommand{\undertitle}{}
\renewcommand{\shorttitle}{Artifacts of Numerical Integration in Learning Dynamical Systems}

\hypersetup{
pdftitle={A template for the arxiv style},
pdfsubject={q-bio.NC, q-bio.QM},
pdfauthor={David S.~Hippocampus, Elias D.~Striatum},
pdfkeywords={First keyword, Second keyword, More},
}
\begin{document}

\maketitle
\begin{abstract}
In many applications, one needs to learn a dynamical system from its solutions sampled at a finite number of time points. The learning problem is often formulated as an optimization problem over a chosen function class. However, in the optimization procedure, \review{prediction data from generic dynamics requires a numerical integrator to assess the mismatch with the observed data.} This paper reveals potentially serious effects of a chosen numerical scheme on the learning outcome. Specifically, the analysis demonstrates that a damped oscillatory system may be incorrectly identified as having "anti-damping" and exhibiting a reversed oscillation direction, even though it adequately fits the given data points. 
\review{This paper shows that the stability region of the selected integrator will distort the nature of the learned dynamics.}
\revmark{Crucially, reducing the step size or raising the order of an explicit integrator does not, in general, remedy this artifact, because higher-order explicit methods have stability regions that extend further into the right half complex plane.} Furthermore, it is shown that the implicit midpoint method can preserve either conservative or dissipative properties from discrete data — offering a principled integrator choice even when the only prior knowledge is that the system is autonomous.
\end{abstract}

\maketitle
\newpage
\section{Introduction}

In numerous applications, such as robotics \cite{Djeumou22}\cite{Djeumou23}\cite{Neary23}\cite{Hersch08}, autonomous driving \cite{Levinson11}\cite{Lucas20}, and pharmacokinetics and medical applications \cite{Laurie23}\cite{Lu21},
it is necessary to extract some understanding of a complex evolutionary process from time-snapshots of the states. We will refer to this kind of task as dynamical system learning. 

Consider autonomous dynamical systems
\begin{eqnarray}\label{eq:true}
    \dfrac{dy}{dt} = f(y)\in \mathbb{R}^d,
\end{eqnarray}
where $f$ is a  {Lipschitz function}. 
Our main goal is to study the potentially non-negligible effects in identifying $f$ from observations of $y(t)$ when a numerical integrator. 

Denote $\phi_t[f]y_0$ as the flow of \eqref{eq:true} that propagates the initial condition $y_0$ to time $t$. Thus, the solution of the differential equation at time $t$, starting with initial condition $y_0$, is denoted as $y(t; y_0)\equiv \phi_t[f]y_0$.

We assume that observations of the trajectories are sampled with a uniform frequency $1/H$, forming the dataset $\mathcal{D}_N = \left\{ (t_n, y(t_n)) \right\}_{n=1}^N$, where $t_n\equiv nH$. 

We formulate the learning problem as the following least squares fitting problem:
\begin{eqnarray}\label{eq:LS-flow}
    \min_{\hat {f}\in \mathcal{F}} \dfrac{1}{NM}\sum_{y_0\in U_M}\sum_{n = 1}^N |\phi_{t_n}[\hat {f}] y_0 - y(t_n\review{; y_0)}|^2,
\end{eqnarray}
where $\mathcal{F}$ is a functional class in which we look for $f$, and $U_M$ contains $M$ samples of initial conditons.

In practice, one typically approximates $ \phi_t$ by using a numerical integrator with a step size $h$. We denote the propagator of the integrator by $S_h[\hat {f}]$. 
 This practice leads to the least squares fitting problem that we shall focus on in this paper:
\begin{eqnarray}\label{eq:LS-Scheme}
    \min_{\hat {f}\in \mathcal{F}} \dfrac{1}{NM}\sum_{y_0\in U_M}\sum_{n = 1}^N |S_h^{mn}[\hat {f}] y_0 - \review{y(t_n; y_0)} |^2,
\end{eqnarray}
where $mh=H.$ In this paper, we also consider linear multistep methods, for which the associated $S_h[\hat f]$ necessarily depends more than just a single initial state.

{Denote the global minimizer of \eqref{eq:LS-Scheme} by $g$. Strictly speaking, one has at hand a discrete-time dynamical system defined by the involved numerical scheme and $g$:
\begin{equation}\label{eq:discrete-time-dynamical-system}
    y_{n+1}:= S_h[g]y_n.
\end{equation}
We distinguish two driving motivations for solving \eqref{eq:LS-Scheme}:
\begin{enumerate}
    \item Use $S_h[g]$ to generate simulated data, extract
    information, and make inferences about the originally observed dynamical system \eqref{eq:true};
    \item Analyze and simulate  the continuous-time dynamical system
    \begin{eqnarray}\label{eq:approx}
    \dfrac{dy}{dt} =  {g}(y);  %S_h, \mathcal{D}_N), 
    \end{eqnarray}
    and make inferences about the originally observed dynamical system \eqref{eq:true}. 
    %(For brevity of notation, we shall drop $S_h$ and $\mathcal{D}_N$ from $g$.)
\end{enumerate}
}

We are interested in understanding how to address the following questions:
\begin{itemize}
\item[Q1] \emph{What if I use a smaller step size or higher order numerical integrator in computing new simulations?}

\item[Q2] \emph{What is the "generalization" to initial data that is out of the original sampled distribution?}
\end{itemize}

\revmark{We shall see that 
%the familiar responses to Q1 — 
reducing $h$ or raising the order of an explicit integrator 
 are not, by themselves, sufficient to avoid the artifact, and can, in some cases, worsen it. This observation motivates the prescriptive recommendation developed in Section~\ref{sec:nonlinear-learning} and in the Conclusion. Q2 is addressed in Section~\ref{sec:nonlinear-learning} through a local-linearization analysis.}

By solving the optimization problem \eqref{eq:LS-Scheme} for a fixed $f$ and $y_0$, one constructs an approximation of the flow $\phi_t[f]y_0$ that is inherently constrained by the underlying numerical scheme. Consequently, comparing the discrete approximation $S_h^n[g]y_0$ with the true flow $\phi_{t_n}[f]y_0$ at times $t_n$ outside the observation window serves to assess the extrapolation performance of this scheme-dependent constraint. \emph{It is therefore critical to quantify the potential discrepancy between $g$ and $f$ arising from the choice of integrator.}

In Figure~\ref{fig:pendulum}, we compare a fully resolved trajectory of a system learned from trajectory snapshots of a damped nonlinear pendulum. The function class $\mathcal{F}$ in \eqref{eq:LS-Scheme} consists of sufficiently expressive Multi-Layer Perceptrons, and a fourth-order Runge-Kutta scheme is used to compute solutions of the candidate dynamical system during learning. While the data are sampled well above the Nyquist rate and are accurately fitted by the optimized discrete dynamical systems induced by the numerical scheme, the fully resolved trajectory reveals that the learned dynamical system is expansive rather than dissipative.

Addressing Q2 is challenging for nonlinear systems: approximating the infinite-dimensional function $f$ from limited trajectory data without constraints is inherently ill-posed, yet essential for capturing limit cycles or attractors. Section~\ref{sec:numerical-examples} investigates this through two numerical studies of nonlinear oscillators.

\revmark{A finite trajectory fit constrains only the values of the learned vector field $g$ along the sampled orbit, leaving the pointwise Jacobian $Dg(z_n)$ underdetermined; whether the resulting loss of stability is visible in the discrete data depends on the choice of integrator. Since this reduces to the stability properties of the discretization applied to local linearizations, we concentrate in Sections~\ref{sec:one-step}--3 on the linear scalar problem and develop the nonlinear consequences in Section~\ref{sec:nonlinear-learning}.}

In most parts of this paper, we consider  \eqref{eq:true} to be a diagonalizable linear system: 
\begin{equation*}
    f(y) = Ay = P\Lambda P^{-1}y,
\end{equation*}
and assume that we know the $P$. 
It suffices to study the impact of a numerical integrator in learning the scalar equation
\begin{eqnarray}
    \dfrac{dz}{dt} = \lambda  z \in \mathbb{C},~~z(t_0) = z_0,
\end{eqnarray} where $\lambda$ is an eigenvalue of $f$.

\review{Henceforth, we focus on the following problem, defined for any $h>0$:
\begin{equation}\label{eq:scalar}
    \min_{\hat{\lambda} \in \mathbb{C}} \sum_{n=1}^{N} \left\| S_{h}^{mn}[\hat{\lambda} z] z_0 - z(t_n; z_0) \right\|^2. 
\end{equation} 
}
Thus, our analysis applies to 
Spectral methods for a broad class of time-dependent partial differential equations.

\revmark{Our contributions are as follows:
\begin{itemize}
    \item For one-step methods applied to the scalar model, we show that the learned parameter $\hat\lambda h$ must lie in the closure of the integrator's stability region. This is a finite-$h$ geometric characterization of the learned dynamics, distinct from asymptotic-in-$h$ error expansions.
    \item As a direct consequence, numerical integrators whose stability regions overlap the right half complex plane can produce a learned $\hat\lambda$ with positive real part; in particular, higher-order explicit methods, whose stability regions extend further into the right half-plane, may in fact worsen rather than reduce this artifact. This mechanism explains Figure~\ref{fig:pendulum} and speaks directly to Q1.
    \item For linear multistep methods, $\hat\lambda h$ need not lie in the stability region, but at least one root of the associated characteristic polynomial remains in the closed unit disc. Selecting this root resolves the non-uniqueness observed in~\cite{Terakawa24}. Linear multistep methods are also more sensitive to perturbations of the initial data, owing to the excitation of spurious roots of the characteristic polynomial.
    \item We quantify how the sample size $N$ and noise level $\sigma$ propagate into the learned parameter through a perturbation analysis, obtaining $\hat\lambda_N(\sigma)$, and show that aggregating multiple trajectories suppresses the effect of noise.
\end{itemize}
Taken together, these results lead to a prescriptive recommendation: when the only prior knowledge about the underlying dynamics is that it is autonomous, we advocate the implicit midpoint rule as a principled default, since its stability region coincides with the left half complex plane, it preserves conservative modes exactly, and it extends naturally to the dissipative regime.}
\begin{figure}
    \centering
        \includegraphics[width=\linewidth]{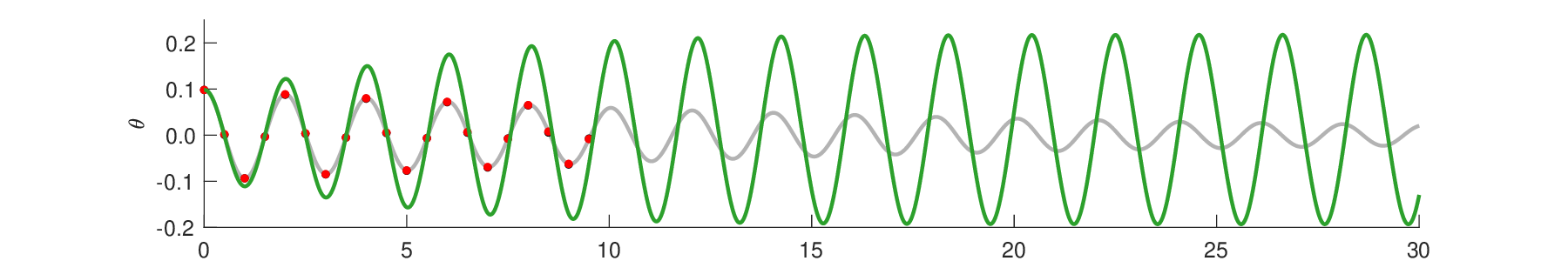}
    \caption{The gray curve shows the angle profile of the underlying damped pendulum dynamics. The red dots along the curve represent the sampled data, while the black dots depict the fitted data obtained using the fourth-order Runge–Kutta method with a time step size of $h$. The green line shows the high-resolution solution for the learned dynamics' trajectory. \revmark{RK4 is chosen here deliberately: its stability region extends into the right half complex plane, which is precisely what permits the learned dynamics to acquire a positive real part and appear expansive even though the sampled data are fitted well. This mechanism is analyzed in Section~\ref{sec:one-step}.}}
    \label{fig:pendulum}
\end{figure}

\subsection*{Related works}
\textbf{Relating numerical schemes to learning.}
 The Inverse Modified Differential Equations (IMDE) method, introduced in \cite{Zhu22}, which builds upon the modified differential equation framework \cite{Calvo94}, provides a robust tool for evaluating learned dynamics by expanding ${  {g_\theta}}$ as a power series in the step size $h$. This method characterizes a differential equation whose numerical solution matches the sampled data at each sampling time. Additionally, in \cite{Zhu24}, a comprehensive analysis using the IMDE method is presented to study linear multistep methods in dynamic learning problems. \review{A comprehensive study of $\|g_\theta - f\|$ can be found in \cite[Theorem 3.2]{Zhu22} for one-step methods and \cite[Theorem 4]{Zhu22}}.  \revmark{The IMDE framework characterizes the learned dynamics through a power series in~$h$. Our analysis is complementary: it relates the learned parameter $\hat\lambda h$ to the finite-$h$ geometry of the integrator's stability region on the complex plane, which makes qualitative artifacts, such as a change in the sign of $\operatorname{Re}\hat\lambda$, directly visible without recourse to an asymptotic expansion.}

In \cite{Terakawa24}, the authors connect learning outcomes to the stability region of Runge-Kutta methods and point out a non-uniqueness phenomenon in learning the scalar equation; \revmark{Section~3 resolves this non-uniqueness through the selection of an appropriate root of the characteristic polynomial.}

In \cite{Keller21, Du22}, the authors study the stability of the discrete dynamical systems defined by linear multistep methods fitting observations of a dynamical system's trajectory and right-hand side, and prove that the learned discrete process $S_h[g]$ defined via Adams-Moulton methods, or higher-order Adams-Bashforth methods, is unstable.

\textbf{Koopman theory and learning dynamical systems.}
Koopman theory~\cite{Koopman31} replaces the nonlinear flow of \eqref{eq:true} by the action of the linear Koopman operator $\mathcal{K}_t g(y):=g(\phi_t(y))$ on observables $g\colon\mathbb{R}^d\to\mathbb{C}$, and extensions to dissipative and driven systems appear in~\cite{mezic04,mezic05}. In practice, $\mathcal{K}_t$ is approximated by projection onto a finite basis of observables, for instance via Extended Dynamic Mode Decomposition (EDMD), whose spectrum is used to infer the continuous-time eigenvalues $\lambda_j$. \revmark{Since the inferred $\lambda_j$ are, ultimately, eigenvalues of a discretized operator, the integrator-induced artifacts identified in this paper propagate directly into Koopman-based learning pipelines.}

\revmark{\textbf{SINDy.}
The Sparse Identification of Nonlinear Dynamics algorithm \cite{Brunton16} fits $g$ by regressing a dictionary of candidate functions against the observed time derivatives $\dot y(t_n)$. In practice, $\dot y(t_n)$ must be estimated from trajectory samples, and numerical differentiation of noisy data introduces its own artifacts — related to, but distinct from, the integrator-stability artifacts analyzed in this paper. Weak-form and trajectory-matching variants of SINDy~\cite{Messenger21,Reinbold20} replace the pointwise derivative regression with a loss that involves numerical integration of the candidate vector field; for those variants, the integrator choice enters the optimization problem directly, and the analysis of the present paper becomes immediately applicable.}

\textbf{Neural ODEs.}
A range of deep-learning approaches construct approximations of $f$ from observations, using either one-step integrators~\cite{Chen18} or linear multistep methods~\cite{Xie19,Raissi18} in the forward solve; see~\cite{Mardt18,Vlachas18,Wehmeyer18,Yeung19,Bertalan19,Kolter19,Qin19,Hu22} for related approaches and surveys. Structure-preserving variants, such as Hamiltonian neural networks~\cite{Greydanus19} and symplectic-integrator-based refinements~\cite{Chen19}, assume prior knowledge of conservative structure, as do implicit-method-based approaches~\cite{Rico93,Oussar01}. The present paper is complementary: it quantifies how the choice of integrator shapes the learned dynamics even in the absence of such structural priors.

\revmark{A summary of the notation used throughout the paper is provided in Table\ref{tab:notation-dyn} ~\ref{tab:notation-data}~\ref{tab:notation-integrators}~\ref{tab:notation-stability} in the Appendix.}

\section{Linear systems learning with one-step methods}\label{sec:one-step}
This section explores the learned dynamics inferred by data using one-step methods. We will qualitatively analyze the behavior of the learned dynamics and connect the results to the theory of absolute stability.

\review{This section proceeds in three steps.
We first use Lemma~\ref{thm:main} to show that the learned dynamics, under the given sampling rate, must remain within the stability region of the selected numerical method.
We then leverage Lemma~\ref{thm:main} to analyze the structure of the learned eigenvalue $\hat{\lambda}$; specifically, $\operatorname{Re}(\hat{\lambda})$ governs the amplitude evolution while $\operatorname{Im}(\hat{\lambda})$ determines the direction of rotation.
Finally, we show that when the data are noisy, incorporating additional trajectories can help mitigate the effect of noise.}
% We first analyze the learning outcomes using noise-free observed data and then extend the discussion to noisy data.

One-step methods for scalar linear equations with constant step size $h$ evolve
the numerical solution by multiplying by the stability function $p(\xi)$,
where $\xi = \lambda h$. The numerical solution after $n$ steps is given by \(z_n = p(\xi)^n z_0\).
In Table~\ref{tab:characteristic-poly}, we show the stability functions
for a few schemes \review{~\cite[Ch.~7]{LeVeque07}} that we will analyze in this paper. \revmark{For the explicit Runge--Kutta family, $p$ is a polynomial; for the implicit Euler and implicit midpoint entries, $p$ is a rational function, and we retain the symbol $p$ in both cases for notational uniformity.}

\begin{table}[h!]
\caption{\revmark{Stability functions $p(\xi)$ for representative one-step schemes: a polynomial for the explicit Runge--Kutta family and a rational function for the implicit schemes shown here. The common symbol $p$ is retained throughout for consistency with later sections.}
%The Explicit Euler is the Explicit Runge-Kutta method with $k=1$.
}
\label{tab:characteristic-poly}
\medskip
\centering
\renewcommand{\arraystretch}{1.8}
\begin{tabularx}{\textwidth}{
>{\centering\arraybackslash}X
>{\centering\arraybackslash}X
>{\centering\arraybackslash}X
}
\toprule
\textbf{Explicit Runge--Kutta ($k=1,2,3,4$)} & \textbf{Implicit Euler} & \textbf{Implicit Midpoint} \\
\midrule
$p(\xi) = \displaystyle\sum_{j=0}^{k} \frac{\xi^j}{j!}$ &
$p(\xi) = \dfrac{1}{1 - \xi}$ &
$p(\xi) = \dfrac{1 + \tfrac{1}{2}\xi}{1 - \tfrac{1}{2}\xi}$ \\
\bottomrule
\end{tabularx}
\end{table}

In the following discussion, we use the notation $Z_n$ to represent the observed data. In particular, $Z_n =e^{\lambda (nh)} Z_0$ is the exact solution of \eqref{eq:scalar} with initial data $Z_0$ at time $t = nh$. Meanwhile, we use $z_n$ to represent the numerical solution at the $n$th iteration step with step size $h$, starting from $z_0 = Z_0$.

In this section, we analyze the global minimizer of least squares problems defined as follows:
\begin{eqnarray}\label{eq:learning-one-step}
    \min_{\xi\in \mathbb{C}} \dfrac{1}{N}\sum_{k = 1}^N |z_{mk}-Z_0e^{\lambda mh k}|^2~~
    \operatorname{\review{where}}~~z_{mk} = p(\xi)^{mk} z_0,
\end{eqnarray}
where $p(\xi)$ is the stability function of a one-step integrator as shown in Table~\ref{tab:characteristic-poly} and $mh$ is the sampling step size with $m$ being a positve integer. \review{In addition, the global minimizer $\hat{\xi} = \hat{\lambda} h$ then serves as the learned dynamic.}

The following Lemma characterizes the global minimizers.
\begin{figure}[!ht]
    \centering
    \begin{subfigure}{0.24\linewidth}
        \includegraphics[width=\linewidth]{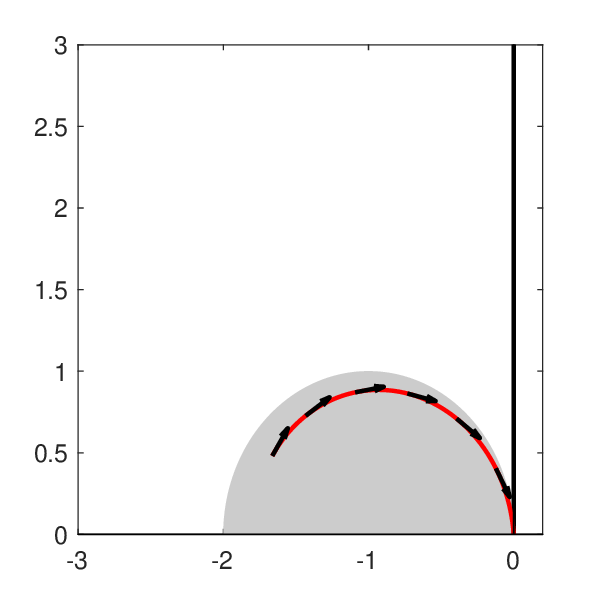}
    \caption{Explicit Euler}
    \label{fig:FE_obj}
    \end{subfigure}
    \begin{subfigure}{0.24\linewidth}
        \includegraphics[width=\linewidth]{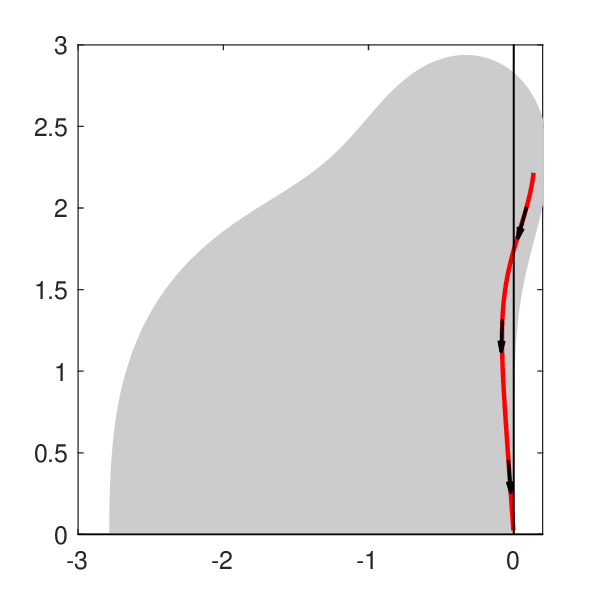}
    \caption{RK4}
    \label{fig:RK4_obj}
    \end{subfigure}
    \begin{subfigure}{0.24\linewidth}
        \includegraphics[width=\linewidth]{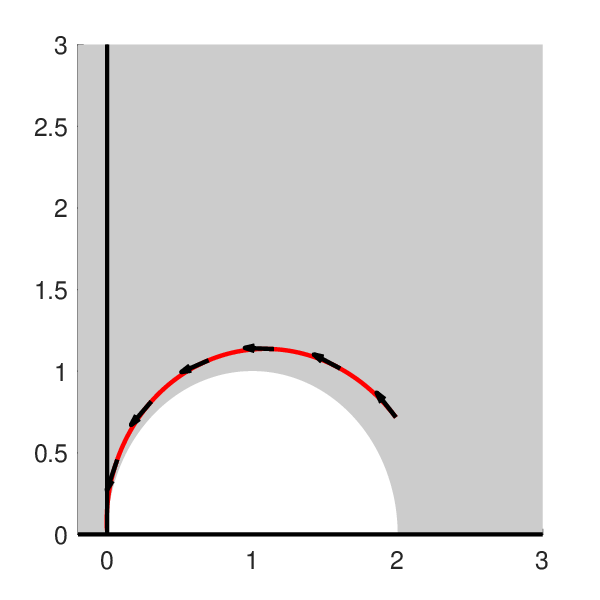}
    \caption{Implicit Euler}
    \label{fig:BE_obj}
    \end{subfigure}
    \begin{subfigure}{0.24\linewidth}
        \includegraphics[width=\linewidth]{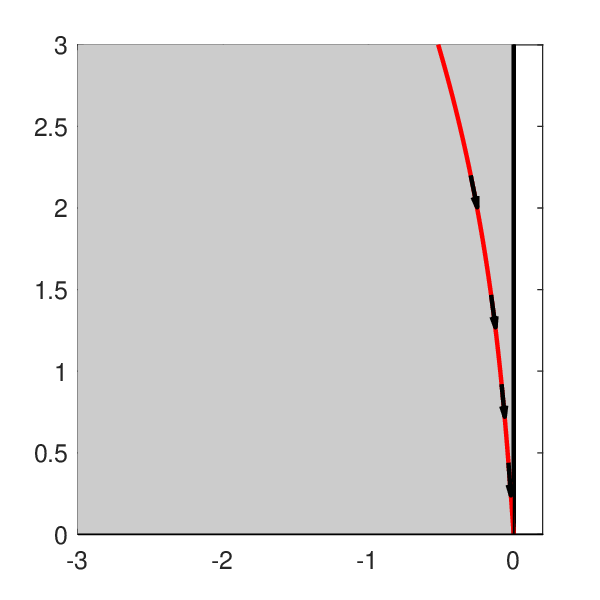}
    \caption{Impl. Midpoint}
    \label{fig:ITr_obj}
    \end{subfigure}
    \caption{Profiles of the learned quantity $\hat{\lambda}h$ using the Explicit Euler, Implicit Euler, RK4, and implicit midpoint rules as $h$ varying from 0 to 0.2. The true $\lambda = -1 + 4\pi i$ is set. Red curves show how $\hat{\lambda}h$ moves within each method’s stability region; arrows point toward the limit as $h \to 0$. }\label{fig:one-step-convergence-diagram}
\end{figure}

\begin{figure}[!ht]
    \centering
    \begin{subfigure}{0.24\linewidth}
        \includegraphics[width=\linewidth]{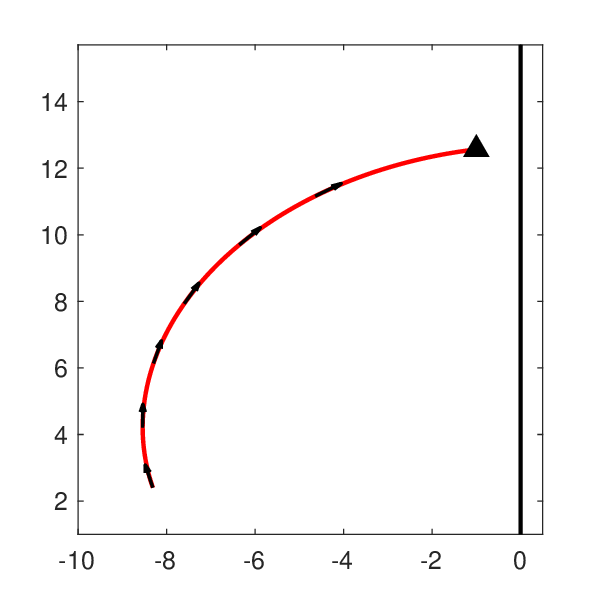}
    \caption{Explicit Euler}
    % \label{fig:FE_obj}
    \end{subfigure}
    \begin{subfigure}{0.24\linewidth}
        \includegraphics[width=\linewidth]{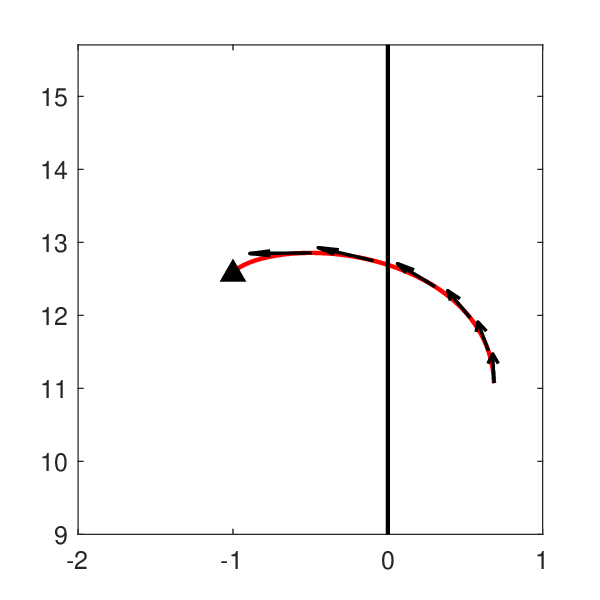}
    \caption{RK4}
    % \label{fig:RK4_obj}
    \end{subfigure}
    \begin{subfigure}{0.24\linewidth}
        \includegraphics[width=\linewidth]{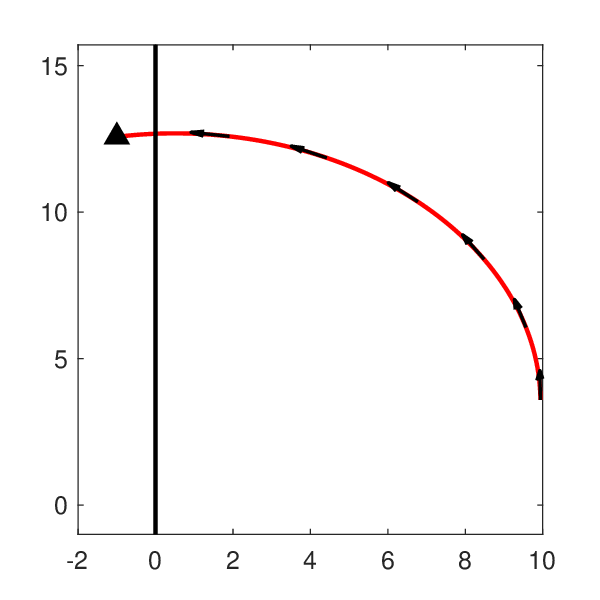}
    \caption{Implicit Euler}
    % % \label{fig:BE_obj}
    \end{subfigure}
    \begin{subfigure}{0.24\linewidth}
        \includegraphics[width=\linewidth]{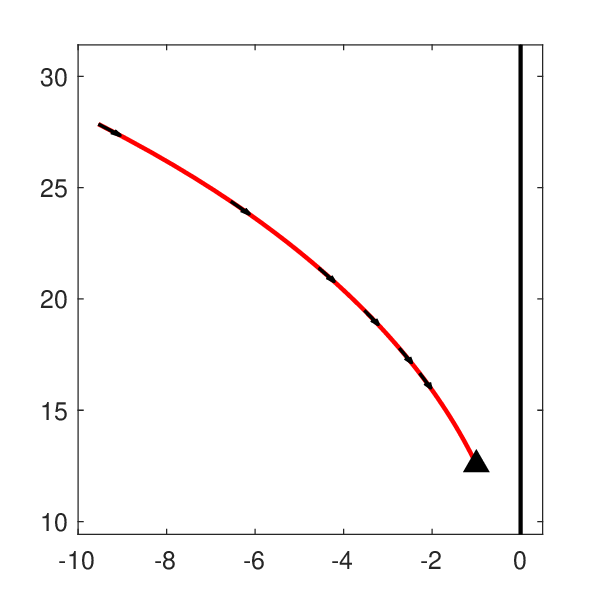}
    \caption{Impl. Midpoint}
    % \label{fig:ITr_obj}
    \end{subfigure}
    \caption{\review{Convergence of the learned dynamic $\hat{\lambda}(h) $ 
    to the ground truth $\lambda = -1 + 4\pi i$ (solid triangle) as the step 
    size $h$ decreases from $0.2$ to $0$ for four time-integration schemes: 
    (a) Explicit Euler, (b) RK4, (c) Implicit Euler, and (d) implicit
    midpoint. The red curve traces the trajectory of $\hat{\lambda}(h)$
    in the complex plane, and the arrows indicate the direction of 
    decreasing $h$.}}
\end{figure}

\begin{lemma}\label{thm:main}
Let $\hat\lambda$ be the minimizer of \eqref{eq:learning-one-step}.
For any positive $h$, there exists $\hat\lambda h\in\mathbb{C}$ satisfying 
    \begin{eqnarray}\label{eq:one-step-condition}
        p({\hat\lambda}h) = e^{\lambda h}. 
    \end{eqnarray}
    Furthermore, $\xi^* = \hat\lambda h$ is the global minimizer to \eqref{eq:learning-one-step} \review{with which the objective function reaches its global minimum 0.}  
\end{lemma}
\begin{proof}
    (i) Explicit Runge-Kutta methods, including Explicit Euler:\\
    By the Fundamental Theorem of Algebra, for any given $\lambda h\in\mathbb{C},$ 
    \begin{eqnarray*}
        p(\xi) - e^{\lambda h}=0
    \end{eqnarray*} has a solution on the complex plane. 
    
    (ii) Now consider the Implicit Euler and the Implicit Midpoint Rule:
    \begin{itemize}
        \item $(1-\xi)e^{\lambda h} = 1$ ~~~[Implicit Euler],
        \item $(1-\dfrac{1}{2}\xi)e^{\lambda h}-(1+\dfrac{1}{2}\xi) = 0$ ~~~[Implicit Midpoint Rule].
    \end{itemize}
    Since $e^{\lambda h} \neq 0$, it is clear that each equation has a unique solution. 
    Denote the solution as $\xi^*=\hat\lambda h$.
    Clearly, for $z_0 = Z_0$,
    \begin{eqnarray*}
        \dfrac{1}{N}\sum_{k = 1}^N |p(\xi)^{mk} z_0-Z_0e^{\lambda mh k}|^2 = 0
    \end{eqnarray*}
    is a global minimum. \\
\end{proof}

 Figure~\ref{fig:one-step-convergence-diagram} compares the paths of optimal $\hat\lambda h$ as $h\rightarrow 0^+$ for different one-step methods and their stability regions.
 In the following subsections, we will provide more explicit characterizations of the signs of $\operatorname{Re}(\hat\lambda)$ and  $\operatorname{Im}(\hat\lambda)$ relative to $\lambda h$.

\subsection{Consistency in recovering the sign of $\lambda$}
Based on Lemma~\ref{thm:main}, we may relate the global minimizer of the least squares, the step size used in the numerical integrator, to the integrator's stability region. 

The following theorem clarifies the errors in the amplitude and the rotation direction of the identified dynamical system using one-step methods.
\begin{theorem}\label{thm:euler_stab}
    Assume that $\operatorname{Re}(\lambda)\leq 0$ in \eqref{eq:scalar}.
    Then the global minimum of \eqref{eq:learning-one-step}, with $S_h^n$ being one-step methods, is attained at $\hat\lambda$ satisfying
    \begin{eqnarray*}
        |p(\hat\lambda h)| = |e^{\lambda h}| \leq 1.
    \end{eqnarray*}
    The equality holds when the $\hat\lambda h$ resides on the boundary of the selected numerical method's stability region.
\end{theorem}

In the following, we present our analysis of a few chosen one-step methods with respect to either conservative or dissipative dynamics. The results concerning the sign of $\operatorname{Re}(\hat\lambda)$ are summarized in
Table~\ref{comp-table-one-step}.
The readers can find the detailed proofs in the Appendix.
A natural prerequisite of our results is that the step size $h$ has to be sufficiently small so that at least two data points are observed for each period of the oscillation.   
\begin{definition}[Nyquist condition for $\lambda h$]\label{def:Nyquist-cond}
    Given $z^\prime=\lambda z$ for some $\lambda\in\mathbb{C}$. 
    We say that $\lambda h$ meets the Nyquist condition~\cite{nyquist1928certain} 
    when 
    \begin{equation}
        -\pi< \operatorname{Im}(\lambda h)<\pi.
    \end{equation}
\end{definition}
As shown in Theorem~\eqref{thm:euler_stab}, the stability region of a numerical method governs the qualitative behavior of the identified dynamics. The next corollary summarizes this relationship in terms of the stability region’s geometry.

\begin{corollary}[Unified characterization of learned dynamics via stability regions]\label{cor:unified_cons}
Assume that $\operatorname{Re}(\lambda) \leq 0$, $\lambda h$ satisfies the Nyquist condition~\eqref{def:Nyquist-cond} and $h>0$.
Let $S_h^n$ denote the evolution operator associated with a single-step numerical method applied in~\eqref{eq:learning-one-step}, and let $\hat{\lambda}$ be the corresponding global minimizer.  

Then the qualitative behavior of the learned dynamics is determined by the location of the method’s stability region in the complex plane:

\noindent When $\operatorname{Re}(\lambda)  = 0$, we have
\begin{enumerate}
    \item The Explicit Euler and RK2 methods will always learn a dynamic with $\operatorname{Re}(\hat{\lambda}) < 0 $.
    
    \item For sufficient small $h$, the RK3 method can learn a dynamic with $\operatorname{Re}(\hat{\lambda}) > 0 $.
    \item The Implicit Euler will always learn a dynamic with $\operatorname{Re}(\hat{\lambda}) > 0 $.
    \item The Implicit Midpoint Rule will always learn a dynamic with $\operatorname{Re}(\hat{\lambda}) = 0 $.
\end{enumerate}
\noindent When $\operatorname{Re}(\lambda)  < 0$, we have

\begin{enumerate}
    \item The Explicit Euler and RK2 methods will always learn a dynamic with $\operatorname{Re}(\hat{\lambda}) < 0 $.
    \item For sufficient small $h$, the RK3 method can learn a dynamic with $\operatorname{Re}(\hat{\lambda}) > 0 $.
    \item The Implicit Midpoint Rule will always learn a dynamic with $\operatorname{Re}(\hat{\lambda}) < 0 $.
\end{enumerate}
Note that for sufficiently small $h$, the Implicit Euler method may learn dynamics that are either expansive or dissipative.

In all cases, the imaginary part of $\hat{\lambda}$ preserves the sign of the ground-truth dynamics:
\[
    \operatorname{sign}\big( \operatorname{Im}(\hat{\lambda}) \big)
    = \operatorname{sign}\big( \operatorname{Im}(\lambda) \big).
\]
\end{corollary}
\begin{remark}[RK4 case]
\review{
The RK4 method is not included in Corollary~\ref{cor:unified_cons} because it requires a thorough analysis of the degree 4 polynomial:
\[
    1 + \xi + \frac{\xi^2}{2} + \frac{\xi^3}{6} + \frac{\xi^4}{24} = e^{\lambda h},
    \qquad \xi = \hat{\lambda} h,
\]
and is deferred to future work.
Nevertheless, the geometric picture remains transparent.
By zooming in on the RK4 stability region near the origin, one observes that
the stability boundary $|p(\xi)|=1$ crosses into the right half of the complex
plane along the imaginary axis in a neighborhood of $\xi = 0$.
This geometric fact implies that, when the data are generated by conservative
dynamics (i.e., $\operatorname{Re}(\lambda) = 0$), the RK4 scheme may learn an
expansive dynamics with $\operatorname{Re}(\hat{\lambda}) > 0$.}
\end{remark}
% Table~\ref{comp-table-one-step} provides a comparison of amplitude errors produced by different numerical methods in learning either conservative or dissipative dynamics.
\review{Table~\ref{comp-table-one-step} reveals a fundamental and perhaps counterintuitive conclusion: 
Classical stability of a numerical integrator is not a reliable guide for the integrator 
selection in the context of dynamical system learning. The decisive factor is the 
geometry of the integrator's stability region on the complex plane. Specifically, 
if the stability region is confined to the left half-plane, then 
$\operatorname{Re}(\hat{\lambda}) \leq 0$ is guaranteed regardless of whether the 
true dynamics are conservative or dissipative, and the learned system will correctly 
preserve the non-expansive structure of the true dynamics. Methods, such as the Implicit Euler, RK3, and RK4 methods,  whose stability regions extend into the right half-plane, violate this condition and may learn an expansive system even when the true dynamics are damped. Among the considered one-step methods, only the 
\emph{implicit midpoint rule} has its stability region confined to the left half-plane, 
and it is therefore the uniquely principled choice for dynamical system learning.}
\begin{table}[!ht]
    \centering
        \caption{The difference between dynamics $\hat{\lambda}$ learned from different numerical methods, while the true dynamics $\lambda$ are either conservative $\operatorname{Re}(\lambda)=0$ or dissipative $\operatorname{Re}(\lambda)<0$. The implicit midpoint rule effectively preserves the structural trends of the true dynamics across different methods. }\label{comp-table-one-step}\medskip
    \begin{tabular}{lllc}
                    \toprule
                    \rule{0pt}{20pt}$y^\prime=\lambda y$ & $\operatorname{Re}(\lambda)=0$ & $\operatorname{Re}(\lambda)<0$ \\
                    \midrule
                    \rule{0pt}{20pt}Explicit Euler & $\operatorname{Re}(\hat\lambda)<0$ & $\operatorname{Re}(\hat\lambda)<0$\\ %\hline
                    \rule{0pt}{20pt}RK2 & $\operatorname{Re}(\hat\lambda)<0$ & $\operatorname{Re}(\hat\lambda)<0$\\%\hline
                    \rule{0pt}{20pt}RK3 & $\operatorname{Re}(\hat\lambda)>0$ (for  $0<h\ll 1$) & Cond. Dissipative \\%\hline
                    % RK4 & Expansive (for sufficiently small $h$) & Conditionally Dissipative \\\hline
                    \rule{0pt}{20pt}Implicit Euler & $\operatorname{Re}(\hat\lambda) >0$ & Cond. Dissipative\\%\hline
                    \rule{0pt}{20pt}{Implicit Midpoint} & $\operatorname{Re}(\hat\lambda)=0$ & $\operatorname{Re}(\hat\lambda)<0$ \\
                    \bottomrule
                \end{tabular}

\end{table}

\begin{table}[h!]
\caption{Sign of $\operatorname{Im}(\hat\lambda h)$ for one-step methods, inferred from data $e^{\lambda h} = e^{a + i\theta}$.}
\label{tab:im-lambdah-one-step-trimmed}
\medskip
\centering
\renewcommand{\arraystretch}{2.5}
\begin{tabularx}{\textwidth}{l >{\raggedright\arraybackslash}X >{\raggedright\arraybackslash}X}
\toprule
\textbf{Method/Data type} & \textbf{Data~with~damping ~($a < 0$)} & \textbf{Purely oscillatory data ($a = 0$)} \\
\midrule
\vspace{5pt}
Explicit Euler &
\makecell[l]{
$\operatorname{Im}(\hat\lambda h) = e^{a} \sin\theta$\\
lagging phase
} &
\makecell[l]{
$\operatorname{Im}(\hat\lambda h) = \sin\theta$\\
small $|\theta|\Rightarrow$ small phase error
} \\

\hline
\vspace{5pt}
Implicit Euler &
\makecell[l]{
$\operatorname{Im}(\hat\lambda h) = e^{-a} \sin\theta$\\
leading phase
} &
\makecell[l]{
$\operatorname{Im}(\hat\lambda h) = \sin\theta$\\
small $|\theta|\Rightarrow$ small phase error
} \\
\hline
\vspace{5pt}
% RK2 &
% \makecell[l]{
% $\operatorname{Im}(\hat\lambda h) = \dfrac{e^a\sin\theta}{1+\operatorname{Re}(\hat\lambda h)}$\\
% } &
% \makecell[l]{
% $\operatorname{Im}(\hat\lambda h)  = \dfrac{\sin\theta}{1+\operatorname{Re}(\hat\lambda h)}$\\
% } \\
% \hline
% \vspace{5pt}
Impl. Midpoint &
\makecell[l]{ %\small
$\operatorname{Im}(\hat\lambda h) = \dfrac{4e^a \sin \theta}{|e^{\lambda h}+1|^2}$\\
lagging phase
} &
\makecell[l]{
$\operatorname{Im}(\hat\lambda h) =  \dfrac{4\sin\theta}{|e^{\lambda h}+1|^2}\approx \theta$\\
small $|\theta|\Rightarrow$ small phase error
} \\

\bottomrule
\end{tabularx}
\end{table}

\review{Table~\ref{tab:im-lambdah-one-step-trimmed} shows that all one-step methods considered correctly 
preserve the sign of $\operatorname{Im}(\hat{\lambda})$, meaning the rotation 
direction of the true dynamics is always recovered. However, small phase fluctuations 
are inevitably introduced, even for the implicit midpoint rule.}

\subsection{Phase errors}\label{sec:phase-error-one-step-methods}
Let the true solution over one sampling interval be
\(e^{\lambda h}=e^{a+i\theta}\) with prescribed step size \(h\).
We derive formulas for the imaginary part of the learned
eigenvalue \(\hat{\lambda}h\) and classify the resulting phase errors. 

Without loss of generality, $z_0 = Z_0 = 1$. Since the analytical solution of the learned dynamics takes the form $e^{\hat\lambda t}$,
\[
\operatorname{Im}(\hat\lambda h) = \phi(\theta, a) \implies 
e^{\hat\lambda t}  = e^{\operatorname{Re}(\hat\lambda t)} e^{i \phi(\theta, a) t/h}. 
\]
So $\phi(\theta, a)$ determines the frequency of oscillations in the learned dynamics, and the phase errors relative to the actual dynamics.

If $|\phi(\theta, a)|>|\theta |$, the learned dynamics has faster oscillations, and we say that the phases in the learned dynamics are "leading."
Otherwise, we say that the phases are "lagging."

%For \(|\theta|\ll1\), i.e. $|\operatorname{Im}(\lambda h)|\ll 1$.

Consider the Explicit Euler method. By Lemma~\ref{thm:main}, each discrete step with step length $h$ should be precisely $e^{\lambda h}:$
\[
1+\hat{\lambda}h =e^{(a+i\theta)}
\;\Longrightarrow\;
\operatorname{Im}(\hat{\lambda}h)=e^{a}\sin(\theta).
\]
Thus, \emph{dissipation/damping in the actual dynamics, i.e., $a<0$, leads to slower oscillations in the learned dynamics.}

In contrast, for the Implicit Euler method, we have
\[
\frac{1}{1-\hat{\lambda}h}=e^{a+i\theta}
\;\Longrightarrow\;
\operatorname{Im}(\hat{\lambda}h)=e^{-a}\sin\theta.
\]
Without dissipation, i.e., $a=0$, the Implicit Euler method will yield identical lagging phase errors (albeit different amplitudes). 
However, dissipation in the actual dynamics \emph{accelerates} the oscillations in the learned dynamics inferred by the Implicit Euler. See also Section~\ref{sec:convection-diffusion-ex} for an example involving a convection-diffusion equation. 

The amplitude error and the phase error in the systems inferred by different numerical schemes are summarized in Table~\ref{comp-table-one-step} and Table~\ref{tab:im-lambdah-one-step-trimmed}.
\subsection{Noisy Data}
\review{Note that when data are noise-free, Theorem~\ref{thm:main} shows that a single
observation suffices for learning; the learned dynamic $\hat{\lambda} h$ will fall
inside the stability region, namely $|p(\hat{\lambda} h)| \leq 1$. In practice,
however, measurements are often polluted by noise, and the learned dynamic
$\hat\lambda h$ may fail to reside in the stability region of the selected method.
A natural remedy is to collect more data, so that noise contributions average out.
 }

\revmark{%
In this subsection, we provide a perturbation analysis of the noisy least-squares problem in the scalar geometric model to analyze the structure of $\hat{\lambda} h$ as the number of sampled data points $N$ or trajectories $M$ in~\eqref{eq:LS-flow} increases.
Throughout, let
\[
a:=e^{\lambda h}\in\mathbb C,
\qquad
Z_0\in\mathbb C\setminus\{0\},
\]
and let $\xi_0,\xi_1,\dots,\xi_N$ be independent, identically distributed complex-valued random variables satisfying
$\mathbb E[\xi_n]=0$, $\mathbb E[|\xi_n|^2]=1$, and $|\xi_n|\le C$ almost surely for some constant $C>0$.
For $\sigma>0$, we write the observation noise in the form
$\eta_n:=\sigma \xi_n$,
so that $\mathbb E[\eta_n]=0$ and $\mathbb E[|\eta_n|^2]=\sigma^2$.%
}

\begin{theorem}[\revmark{Perturbation of a nearby minimizer for the noisy least-squares problem}]\label{thm:noisy-perturbation}
\revmark{For $\sigma\in\mathbb R$, define
\[
Z(t_n):=Z_0 a^n+\eta_n=Z_0 a^n+\sigma\xi_n,\qquad n=1,\dots,N,
\]
and
\[
L_\sigma(q;N):=\frac1{2N}\sum_{n=1}^N \left|\,Z(t_n)-(Z_0+\eta_0)q^n\,\right|^2,
\qquad q\in\mathbb C.
\]
Set
\[
S_N:=\sum_{n=1}^N n^2|a|^{2n-2},
\qquad
R_N:=\sum_{n=1}^N n\,|a|^{2n-2},
\]
and
\[
\Delta_N:=\sum_{n=1}^N n\,\overline a^{\,n-1}\xi_n-a\,\xi_0R_N.
\]}

\revmark{Then the following hold.}

\begin{enumerate}
\item[\rm (i)] \revmark{In the noiseless case $\sigma=0$, the function $L_0(\cdot;N)$ has the unique global minimizer
$q=a$.}

\item[\rm (ii)] \revmark{For $|\sigma|$ sufficiently small, there exists a unique local minimizer $\hat q_N(\sigma)$ in a neighborhood of $a$, with
$\hat q_N(0)=a$,
and
\begin{equation}\label{eq:noisy-perturbation}
\hat q_N(\sigma)
=
a+\sigma\,\frac{\Delta_N}{Z_0 S_N}+O(\sigma^2).
\end{equation}
In particular,
$\hat q_N(\sigma)-a=O(\sigma)$ as $\sigma\to 0$.}

\item[\rm (iii)] \revmark{Consequently,
$\mathbb E[\hat q_N(\sigma)-a]=O(\sigma^2)$,
and
\begin{equation}\label{eq:noisy-mse}
\mathbb E\bigl[|\hat q_N(\sigma)-a|^2\bigr]
=
\frac{\sigma^2}{|Z_0|^2}\,
\frac{S_N+|a|^2R_N^2}{S_N^2}
+O(\sigma^3).
\end{equation}}
\end{enumerate}

\revmark{In particular, the first-order perturbation of the nearby minimizer is distribution-free, while the leading-order mean-square error depends only on the noise variance~$\sigma^2$.}
\end{theorem}

\begin{proof}
\revmark{For part~{\rm(i)}, when $\sigma=0$,
$L_0(q;N)
=
{|Z_0|^2}(2N)^{-1}\sum_{n=1}^N |a^n-q^n|^2$.
This is nonnegative and vanishes at $q=a$. Since the $n=1$ term is
${|Z_0|^2}(2N)^{-1}|a-q|^2$,
any minimizer must satisfy $q=a$.}

\revmark{For part~{\rm(ii)}, since $L_\sigma$ depends smoothly on $(q,\overline q,\sigma)$, and since $a$ is a nondegenerate local minimizer of $L_0$, the implicit function theorem applied to the stationarity equation yields a smooth local minimizer branch
$\hat q_N(\sigma)=a+\delta_N(\sigma)$
with $\delta_N(\sigma)=O(\sigma)$ as $\sigma\to 0$.
The boundedness assumption $|\xi_n|\le C$ ensures that $|\Delta_N|$ is uniformly bounded, so the remainder is indeed $O(\sigma^2)$ for all realizations.
Expanding
$(a+\delta_N)^n=a^n+n a^{n-1}\delta_N+O(|\delta_N|^2)$,
one obtains
$Z(t_n)-(Z_0+\eta_0)(a+\delta_N)^n
=
\sigma(\xi_n-\xi_0 a^n)-Z_0 n a^{n-1}\delta_N+O(\sigma^2)$.
Hence, to first order in $\sigma$, the displacement $\delta_N$ is determined by minimizing
$\sum_{n=1}^N \left|\sigma(\xi_n-\xi_0 a^n)-Z_0 n a^{n-1}\delta\right|^2$
with respect to $\delta\in\mathbb C$. The corresponding normal equation gives
$\delta
=
\sigma\,{\Delta_N}/(Z_0 S_N)$,
establishing~\eqref{eq:noisy-perturbation}.}

\revmark{For part~{\rm(iii)}, since $\mathbb E[\xi_n]=0$, one has $\mathbb E[\Delta_N]=0$, and therefore
$\mathbb E[\hat q_N(\sigma)-a]=O(\sigma^2)$.
By independence and zero mean,
$\mathbb E[|\Delta_N|^2]
=
S_N+|a|^2R_N^2$,
which yields~\eqref{eq:noisy-mse}.}
\end{proof}

\begin{corollary}[\revmark{Multiple independent trajectories}]\label{cor:noisy-multiple-shots}
\revmark{For $m=1,\dots,M$, let
$Z_m(t_n):=Z_0 a^n+\sigma\xi_{n,m}$, $n=1,\dots,N$,
and define
\[
L_\sigma(q;N,M):=\frac1{2NM}\sum_{m=1}^M\sum_{n=1}^N
\left|\,Z_m(t_n)-(Z_0+\sigma\xi_{0,m})q^n\,\right|^2,
\qquad q\in\mathbb C.
\]
Assume that $\{\xi_{n,m}\}_{0\le n\le N,\, 1\le m\le M}$ are i.i.d.\ with
$\mathbb E[\xi_{n,m}]=0$, $\mathbb E[|\xi_{n,m}|^2]=1$, and $|\xi_{n,m}|\le C$ a.s.
Let $\hat q_{N,M}(\sigma)$ be the local minimizer near $a$. For each $m$, define
$\Delta_N^{(m)}
:=
\sum_{n=1}^N n\,\overline a^{\,n-1}\xi_{n,m}
-
a\,\xi_{0,m}R_N$.
Then
\begin{equation}\label{eq:noisy-multi-perturbation}
\hat q_{N,M}(\sigma)
=
a+\sigma\,\frac{1}{MZ_0 S_N}\sum_{m=1}^M \Delta_N^{(m)}+O(\sigma^2).
\end{equation}
Consequently,
$\mathbb E[\hat q_{N,M}(\sigma)-a]=O(\sigma^2)$,
and
\begin{equation}\label{eq:noisy-multi-mse}
\mathbb E\bigl[|\hat q_{N,M}(\sigma)-a|^2\bigr]
=
\frac{\sigma^2}{M|Z_0|^2}\,
\frac{S_N+|a|^2R_N^2}{S_N^2}
+O(\sigma^3).
\end{equation}
In particular, the leading-order mean-square fluctuation is reduced by a factor of $M$ relative to the single-trajectory case.}
\end{corollary}

\begin{proof}
\revmark{Writing $\hat q_{N,M}(\sigma)=a+\delta_{N,M}(\sigma)$ and expanding as in the proof of the theorem gives, for each $m,n$,
$Z_m(t_n)-(Z_0+\sigma\xi_{0,m})(a+\delta_{N,M})^n
=
\sigma(\xi_{n,m}-\xi_{0,m}a^n)-Z_0 n a^{n-1}\delta_{N,M}+O(\sigma^2)$.
The normal equation yields
$\delta
=
\sigma\,(MZ_0 S_N)^{-1}\sum_{m=1}^M \Delta_N^{(m)}$,
proving~\eqref{eq:noisy-multi-perturbation}.
By independence across $m$,
$\mathbb E\left|\sum_{m=1}^M \Delta_N^{(m)}\right|^2
=
M\,\mathbb E\bigl[|\Delta_N^{(1)}|^2\bigr]
=
M(S_N+|a|^2R_N^2)$,
which gives~\eqref{eq:noisy-multi-mse}.}
\end{proof}

\revmark{The preceding analysis is carried out in the amplification-factor variable $q$. The connection to the eigenvalue learning problem~\eqref{eq:learning-one-step} is provided by the following reduction.}

\begin{corollary}\label{cor:noisy-lambda}
\revmark{Let $p$ be a nonconstant characteristic polynomial from Table~\ref{tab:characteristic-poly}, and let $\hat\lambda$ satisfy
$p(\hat\lambda h)=a$ with $p'(\hat\lambda h)\neq 0$.
Let $\hat\lambda_N(\sigma)$ be the unique branch near $\hat\lambda$ such that
$p(\hat\lambda_N(\sigma) h)=\hat q_N(\sigma)$.
Then
\begin{equation}\label{eq:noisy-lambda-perturbation}
\hat\lambda_N(\sigma)-\hat\lambda
=
\frac{\hat q_N(\sigma)-a}{h\,p'(\hat\lambda h)}+O(\sigma^2)
=
\sigma\,\frac{\Delta_N}{Z_0\, h\,p'(\hat\lambda h)\,S_N}+O(\sigma^2).
\end{equation}
Moreover,
$\mathbb E[\hat\lambda_N(\sigma)-\hat\lambda]=O(\sigma^2)$,
and
\begin{equation}\label{eq:noisy-lambda-mse}
\mathbb E\bigl[|\hat\lambda_N(\sigma)-\hat\lambda|^2\bigr]
=
\frac{\sigma^2}{|h\,p'(\hat\lambda h)|^2|Z_0|^2}\,
\frac{S_N+|a|^2R_N^2}{S_N^2}
+O(\sigma^3).
\end{equation}}
\end{corollary}

\begin{proof}
\revmark{Since $p'(\hat\lambda h)\neq 0$, the implicit function theorem yields a unique local branch $\mu=\mu(\rho)$ solving $p(\mu h)=\rho$ for $\rho$ near $a$, with $\mu(a)=\hat\lambda$. Expanding at $\hat\lambda$,
$p(\hat\lambda h+h\delta)=a+h\,p'(\hat\lambda h)\delta+O(\delta^2)$,
and setting $p(h\hat\lambda_N(\sigma))=\hat q_N(\sigma)$ gives
the stated formula. The moment estimates follow from Theorem~\ref{thm:noisy-perturbation}.}
\end{proof}

\begin{remark}
\revmark{The simple-root assumption $p'(\hat\lambda h)\neq 0$ is essential. If $p'(\hat\lambda h)=0$, then the perturbation of nearby roots is generally not linear in $\hat q_N(\sigma)-a$.}
\end{remark}

\begin{remark}
\revmark{The boundedness assumption $|\xi_n|\le C$ guarantees that $|\Delta_N|$ is uniformly bounded over all realizations, so the implicit function theorem delivers a deterministic $O(\sigma^2)$ remainder. If the noise is instead sub-Gaussian (e.g., Gaussian), the same expansion holds with probability at least $1-\delta$ for any $\delta>0$, provided $\sigma$ is small enough depending on $\delta$ and $N$: standard concentration inequalities give $|\Delta_N|=O(\sqrt{N\log(1/\delta)})$ with the stated probability, and the implicit function theorem argument applies on this event. The moment formulas~\eqref{eq:noisy-mse} and~\eqref{eq:noisy-multi-mse} remain valid whenever $\mathbb{E}[|\xi_n|^2]=1$ and $\mathbb{E}[|\xi_n|^4]<\infty$.}
\end{remark}

The following numerical example illustrates \revmark{Corollary~\ref{cor:noisy-multiple-shots}} with the conservative scalar ODE $\frac{dx}{dt} = ix, \quad x(0) = x_0$,
whose exact solution traces the orbit. We recover $\hat{\lambda}$ by
minimizing~\eqref{eq:LS-flow} under the implicit midpoint discretization,
which is energy-preserving for purely imaginary eigenvalues, meaning
$|p(\hat{\lambda}h)| = 1$ holds exactly when $\hat{\lambda}$ is identified
correctly. Table~\ref{tab:noise-M} shows that for small $M$, insufficient
data causes a negative real part, breaking
the energy-preservation condition and introducing artificial dissipation or
growth into the predicted dynamics.

\begin{table}[h!]
\caption{Effect of $M$ on eigenvalue estimation for $\dot{x} = ix$, 
identified via the implicit midpoint rule ($h = 0.1$, $\sigma = 0.1$).}
\label{tab:noise-M}
\medskip
\centering
\renewcommand{\arraystretch}{2.5}
\begin{tabularx}{\textwidth}{l >{\raggedright\arraybackslash}X 
                               >{\raggedright\arraybackslash}X}
\toprule
\textbf{$M$} & \textbf{Learned eigenvalue $\hat{\lambda}$} 
             & \textbf{$|p(\hat{\lambda}h)|$} \\
\midrule
\vspace{5pt}
Reference (no noise) &
\makecell[l]{$\hat{\lambda} = \phantom{-}0.0000 + 1.0008i$} &
\makecell[l]{$|z| = 1.00000$} \\
\hline
\vspace{5pt}
$M = 1$ &
\makecell[l]{$\hat{\lambda} = -0.0009 + 1.0018i$} &
\makecell[l]{$|z| = 0.99990$} \\
\hline
\vspace{5pt}
$M = 5$ &
\makecell[l]{$\hat{\lambda} = \phantom{-}0.0001 + 1.0010i$} &
\makecell[l]{$|z| = 1.00001$} \\
\hline
\vspace{5pt}
$M = 100$ &
\makecell[l]{$\hat{\lambda} = \phantom{-}0.0000 + 1.0009i$} &
\makecell[l]{$|z| = 1.00000$} \\
\bottomrule
\end{tabularx}
\end{table}

Figure~\ref{fig:noise_M_comparison} further confirms this: with $M = 1$ the 
learned dynamics deviate from the true orbit, causing the long-term prediction 
to drift; with $M = 100$, where 5 trajectories are randomly selected from the 
pool for display, the prediction remains 
on the orbit.

\begin{figure}[h!]
    \centering
    \begin{subfigure}[t]{0.24\textwidth}
        \centering
        \includegraphics[width=\linewidth]{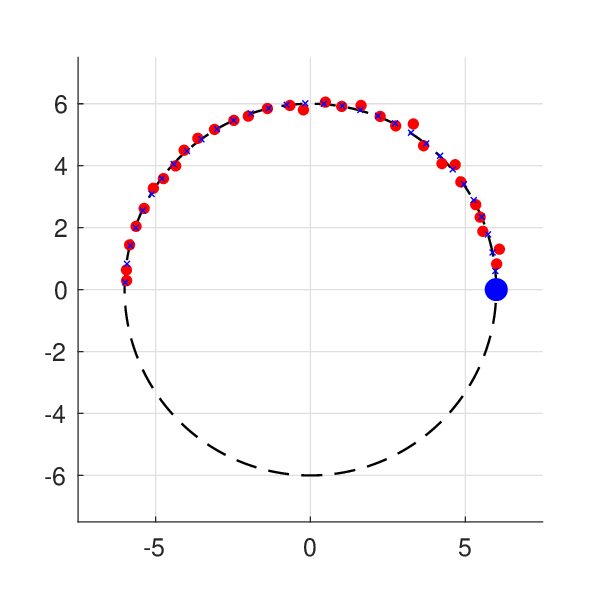}
        \caption{Observations, \\$M=1$.}
        \label{fig:M1_sample}
    \end{subfigure}
    \hfill
    \begin{subfigure}[t]{0.24\textwidth}
        \centering
        \includegraphics[width=\linewidth]{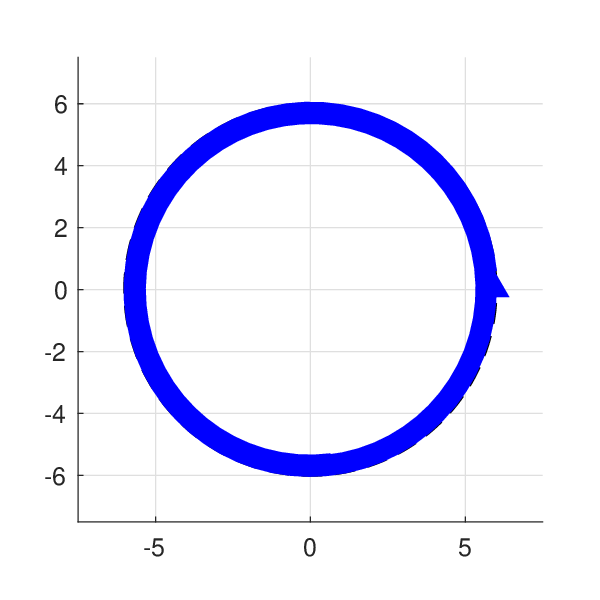}
        \caption{Prediction, \\$M=1$.}
        \label{fig:M1_prediction}
    \end{subfigure}
    \hfill
    \begin{subfigure}[t]{0.24\textwidth}
        \centering
        \includegraphics[width=\linewidth]{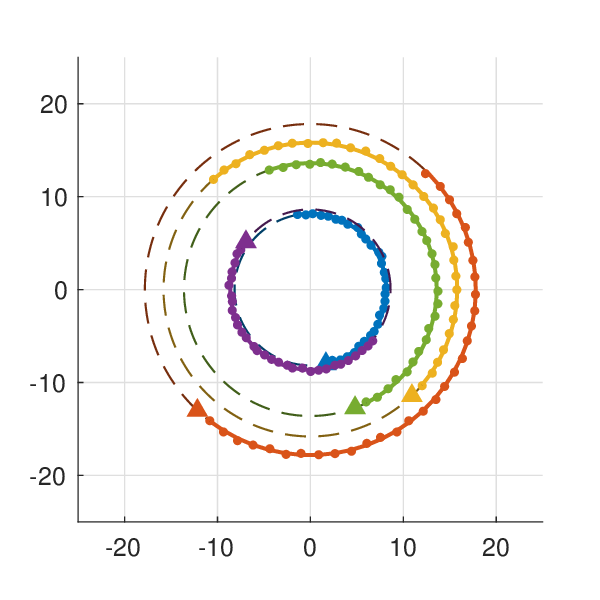}
        \caption{Observations, \\$M=100$ (5 shown).}
        \label{fig:M100_sample}
    \end{subfigure}
    \hfill
    \begin{subfigure}[t]{0.24\textwidth}
        \centering
        \includegraphics[width=\linewidth]{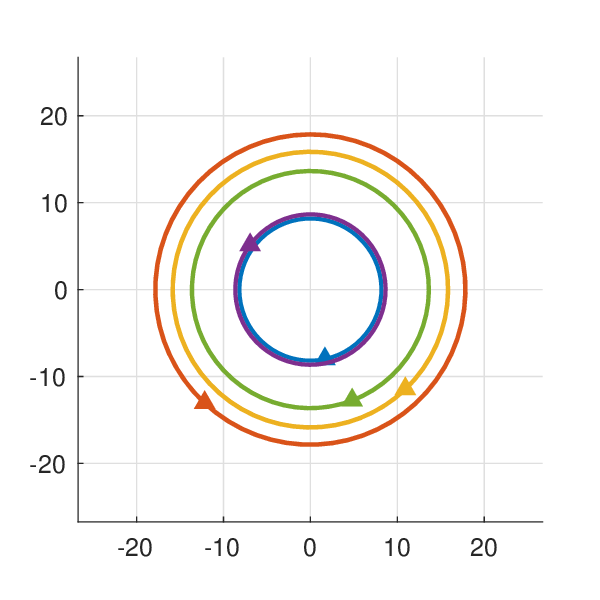}
        \caption{Prediction, \\$M=100$ (5 shown).}
        \label{fig:M100_prediction}
    \end{subfigure}
    \caption{\review{Learning quality under additive noise ($\sigma=0.1$, $h=0.1$).
    The horizontal and vertical axes are the real and imaginary parts of the
    state; the true solution is the unit circle (dashed). Dots are
    noise-corrupted observations. The right two panels correspond to $M=100$
    initial points used for identification; for visual clarity, 5 trajectories
    are randomly selected from the pool and displayed in distinct colors.
    With $M=1$ the predicted trajectory drifts away from the unit circle;
    with $M=100$ it remains on it.}}
    \label{fig:noise_M_comparison}
\end{figure}

\section{Linear systems learning with linear multistep methods}
\review{This section analyzes the learning effect using linear multistep methods.
We first use Lemma~\ref{thm:main} to show that when the true dynamic is conservative, then the learned dynamic $\hat{\lambda}$, under noise-free data, must lie on the boundary of the absolute stability region of the selected method. We then analyze the structure of the characteristic roots of the linear multistep
methods. When the data are clean, the spurious roots with magnitude greater than one
are suppressed and do not contribute to the learned dynamics.
However, when the data are contaminated by noise, these spurious roots can be excited,
causing the learned discrete dynamics to diverge.}
Linear multistep methods (LMMs) for the scalar differential equation, $\dfrac{dz}{dt} = \lambda z,$ take the general form:
\begin{eqnarray}\label{eq:LMMs}
    \sum_{j=0}^k \alpha_j z_{n+j} = \lambda h \sum_{j=0}^k \beta_j z_{n+j},~~~\alpha_k=1,
\end{eqnarray}
 with the given initial conditions
\begin{equation}\label{eq:LMMs-init-data}
z_j = Z_0e^{ j\lambda h},~~~j=0,1,2,\cdots, k-1.    
\end{equation}
For example, the  Leap-Frog scheme:
\begin{eqnarray}\label{eq:Exp-Trapezoidal}
    z_{n+2} = z_{n} + 2\lambda h z_{n+1}.
\end{eqnarray}

Define the characteristic polynomial of a linear multistep method
\begin{equation}
    \pi_{\lambda h}(z):=\rho(z)- \lambda h\kappa(z)=0,
\end{equation}
where
\begin{eqnarray}\label{LMM-characteristic-polynomials}
        \rho(z) = \sum_{j=0}^k \alpha_j z^{k-j},~\text{and}~~\review{\kappa(z)} = \sum_{j=0}^k \beta_j z^{k-j}.
\end{eqnarray}
For the given $\lambda h$, if the roots are unique, 
the solution of the recurrence relation \eqref{eq:LMMs} has the general form
\begin{equation}\label{eq:LMM_general_soln}
    z_n = \sum_{j=1}^k c_j \zeta^n_j.
\end{equation}
If a root $\zeta_j$ has a multiplicity $\mu$, then the general solution will include terms of the form
\[
n^{\nu}\zeta_j^n,~~~\nu=0,1,2,\cdots, \mu-1.
\]

In standard \review{practice}, the coefficients $c_1,\cdots, c_k$ are determined by the given initial data: $z_0, z_1,\cdots, z_{k-1}.$  
We denote $c_j\equiv c_j(\xi; z_0,\cdots, z_{k-1})$ to emphasize the smooth dependence on $\xi$ and the initial \review{conditions}.
Consequently,
\[
z_n \equiv z_n(\xi; z_0,\cdots, z_{k-1}).
\]

For a given $\xi,$ some of the roots may lie outside of the unit disc and some inside.
If the given initial conditions lie in the span of those $\zeta_j \equiv \zeta_j(\xi)$ lying strictly inside the unit disc,  $|z_n|$ is uniformly bounded for $n\ge 0$.
Thus, the stability of $z_n$ is determined by 
$\xi = \lambda h$  as well as the given initial condition. 

 The region of absolute stability for an LMM is defined as the set for $\xi=\lambda h$ in which the roots of $\pi_{\lambda h}(z)$ lie within the unit disc on the complex plane; i.e.
\begin{equation*}\label{abs-region-LMM}
    \mathcal{R}_A:=\{ \xi\in\mathbb{C}: \text{for any}~z\in\mathbb{C}~\text{such that}~\rho(z)-\xi\kappa(z) =0~\implies|z|<1\}.
\end{equation*}

In addition, we introduce the notion of \emph{partial stability}, which refers to the case where the method is stable for properly chosen initial conditions. 
We define the region of partial stability
as
\begin{equation*}\label{par_stable-region-LMM}
    \mathcal{R}_P:=\{ \xi\in\mathbb{C}: \exists~z_j\in\mathbb{C},~\rho(z_j)-\xi\kappa(z_j) =0, j=1,2, \text{and}~ |z_1|\le 1,|z_2|> 1\}.
\end{equation*}

For simulating dissipative problems, it is crucial to select $h$ so that $\lambda h\in \mathcal{R}_A$. This way, the computation will be stable \emph{for all} possible initial data. 
However, the main interest in this paper is in what $\hat\lambda h$ would be for a chosen LMM. In this section, we show that it is possible to identify $\hat\lambda h\in \mathcal{R}_P$ from a given sequence of observations with exponentially decreasing amplitudes. 
A potential consequence is that the identified system can be unstable for many other initial conditions.

We consider the learning problem for $k$-step linear multistep methods as the following constrained optimization problem, in which the data sampling time step $H=mh$ requires $m$ steps ($m>k$) of the integrator using step size $h$:
\begin{equation}\label{LMM-learning-problem}
\left\{
\begin{aligned}
        \min_{\xi\in\mathbb{C}}~
        \dfrac{1}{N}&\sum_{n=1}^N \left|z_{mn}(\xi, c_1,\cdots,c_k)-Z_{mn}\right|^2,\\
    &\sum_{j=0}^k \alpha_j z_{n+j}(\xi) = \xi \sum_{j=0}^k \beta_j z_{n+j}(\xi),~~~\alpha_k=1,~~\text{and}\\
   & z_j(\xi) = Z_j,~~j=0,1,\cdots, k-1,
\end{aligned}
\right.
\end{equation} 
where $Z_{mn}$ for $n=0,1,\cdots, N$, are the observation data, and $z_0, z_1,\cdots, z_{k-1}$ are the initial data for the problem; these initial conditions can be interpolated from the observed data points. 

For the model problem $z^\prime=\lambda z$, it is equivalent to:
\begin{equation}\label{LMM-learning-problem-v2}
\left\{
\begin{aligned}
        \min_{\xi\in\mathbb{C}}~
        \dfrac{1}{N}&\sum_{n=1}^N 
        | \sum_{j=1}^k c_{j} \zeta^{mn}_{j}(\xi) - Z_{mn} |^2,~~~n=1,2,\cdots,N,\\
   & \sum_{j=1}^k c_j \zeta^{n}_j(\xi) = Z_n,~~n=0,1,\cdots, k-1.
\end{aligned}
\right.
\end{equation} 
Note that if $Z_n = \zeta_j^n,j=0,1,2,\cdots, k-1,$ i.e., then $c_{n} = 0 $ for all $n \neq j.$
Conversely, \emph{if $Z_n=e^{n \lambda h}$, there exists a $\hat \xi = \hat\lambda h$ such that $c_{n} = 0 $ for all $n \neq j.$}
\begin{itemize}
    \item If $\operatorname{Re}(\lambda)<0,$
    {$\hat\lambda h$ is \textbf{not} necessarily be in $\mathcal{R}_A$, }
    because the unstable roots cannot contribute to forming the values of $z_n$ to fit the data with exponentially decreasing amplitudes.
    \item  However, when one computes the learned iterations 
    \[
    \sum_{j=0}^k \alpha_j z_{n+j} = \hat\lambda h \sum_{j=0}^k \beta_j z_{n+j},
    \]
    with initial conditions that excite the unstable roots, the resulting simulation becomes unstable.
    \item Even if all the characteristic roots, $\zeta_j(\hat\lambda)$ lie inside the unit disc, $\hat\lambda h$ may not! This possibility is evident for those methods whose $\mathcal{R}_A$ overlaps with the right half of the complex plane.
\end{itemize}

In the remainder of this section, we consider the case when the sampling time step is the same as the step size of the numerical scheme, observations have no noise, and $Z_0 = 1$; i.e., $m=1$ and  $Z_n = e^{n \lambda h}$.

\begin{lemma}\label{lemma:lmm}
    The global minimum of the learning problem \eqref{LMM-learning-problem}, with $m=1$ and $Z_n = e^{n \lambda h}$, $n=0,1,\cdots, N$, is attained at 
    \begin{eqnarray}\label{eq:LMM-condition}
        \hat\xi = \frac{\rho(e^{\lambda h})}{\kappa(e^{\lambda h})},
    \end{eqnarray}
    and  the solution to the recurrence relation matches the analytical solution of the differential equation; i.e.
    $z_n(\hat \xi) = e^{n\lambda h}.$
\end{lemma}

\begin{proof}
    Consider the first $k$ terms of the LMMs, 
    \begin{eqnarray}\label{eq:Lemma3.1_proof}
        \sum_{j=0}^k \alpha_j e^{j\lambda h} = \xi \sum_{j=0}^k \beta_j e^{j\lambda h},~~~\alpha_k=1
    \end{eqnarray}
    and it is straightforward to find $\xi = \dfrac{\sum_{j=0}^k \alpha_j e^{j\lambda h}}{\sum_{j=0}^k \beta_j e^{j\lambda h}} = \dfrac{\rho(e^{\lambda h})}{\kappa(e^{\lambda h})}$.\\
    Next, the numerical solution at $z_{k+1}$ follows the following formula,
    \begin{eqnarray}
        &&\alpha_{k}z_{k+1} + \sum_{j=0}^{k-1} \alpha_j e^{(j+1)\lambda h} = \xi\left(\beta_k z_{k+1}+\sum_{j=0}^{k-1} \beta_j e^{(j+1)\lambda h}\right)\\
         = &&\alpha_{k}z_{k+1} + e^{\lambda h}\sum_{j=0}^{k-1} \alpha_j e^{j\lambda h} = \xi\left(\beta_k z_{k+1}+e^{\lambda h}\sum_{j=0}^{k-1} \beta_j e^{j\lambda h}\right).
    \end{eqnarray}\label{eq:lmm_proof}
    If $\xi = \dfrac{\sum_{j=0}^k \alpha_j e^{j\lambda h}}{\sum_{j=0}^k \beta_j e^{j\lambda h}} = \dfrac{\rho(e^{\lambda h})}{\kappa(e^{\lambda h})}$, 
    \begin{eqnarray*}
        \xi \left(e^{\lambda h}\sum_{j=0}^{k-1} \beta_j e^{j\lambda h}\right) = e^{\lambda h}\left(\sum_{j=0}^k \alpha_j e^{j\lambda h}-\xi\beta_k e^{k\lambda h}\right).
    \end{eqnarray*}
    Thus, we can write the solution at $n = k+1$ as:
    \begin{eqnarray*}
        &&\alpha_{k}z_{k+1} + e^{\lambda h}\sum_{j=0}^{k-1} \alpha_j e^{j\lambda h} = \xi\left(\beta_k z_{k+1}\right)+e^{\lambda h}\left(\sum_{j=0}^k \alpha_j e^{j\lambda h}-\xi\beta_k e^{k\lambda h}\right) \\
        \Rightarrow &&\alpha_{k}\left(z_{k+1}-e^{\lambda h} e^{k\lambda h}\right) = \xi \beta_k\left(z_{k+1}-e^{\lambda h}e^{k\lambda h}\right).
    \end{eqnarray*}
    Since $\alpha_k - \xi \beta_k \neq 0$, we must have
    \begin{eqnarray*}
        z_{k+1}-e^{\lambda h}e^{k\lambda h} = 0 \Rightarrow z_{k+1} = e^{(k+1)\lambda h}.
    \end{eqnarray*}
    If $\xi$ is chosen to be $\dfrac{\rho(e^{\lambda h})}{\kappa(e^{\lambda h})}$, then,  by induction,  $z_{n}(\xi) = e^{n\lambda h}$ for all $n$.
\end{proof}

% \begin{table}[h!]
% \caption{Stability and root behavior for two-step linear multistep methods.}\label{tab:two-step-LMM-roots-suumary}\medskip
% \centering
% \renewcommand{\arraystretch}{1.5}
% \begin{tabular}{@{}llll@{}}
% \toprule
% \textbf{Method} & \textbf{Characteristic Polynomial} & \textbf{Repeated Roots} & \textbf{Mixed Magnitudes?} \\
% \midrule
% Leap-Frog &
% $z^2 - 2\lambda h z - 1$ &
% $\lambda h = \pm i$ &
% Always ($z_1 z_2 = -1$) \\

% BDF2 &
% $(3 - 2\lambda h)z^2 - 4z + 1$ &
% $\lambda h = \frac{5}{2}$ &
%  $\operatorname{Re}(\lambda h)\gg 1$, $\operatorname{Im}(\lambda h)\approx 0$ \\

% AB2 &
% $z^2 - \left(1 + \frac{3}{2}\lambda h\right)z + \frac{1}{2}\lambda h$ &
% $\left(1 + \frac{3}{2}\lambda h\right)^2 = 2\lambda h$ &
% Only if $\operatorname{Im}(\lambda)\neq 0$ \\

% AM2 &
% $z^2 - \left(1 - \frac{1}{2}\lambda h\right)z - \frac{1}{2}\lambda h$ &
% $\left(1 - \frac{1}{2}\lambda h\right)^2 = 2\lambda h$ &
% Only if  $\operatorname{Im}(\lambda)\neq 0$ \\
% \bottomrule
% \end{tabular}
%\end{table}

\begin{theorem}\label{thm:conservative-dynamics-LMMs}
    Assume that $\operatorname{Re}(\lambda)=0$ in \eqref{eq:scalar}.
    If $\hat\lambda$ is a global minimizer of \eqref{LMM-learning-problem}, with $m=1$ and $Z_n = e^{n \lambda h}$, $n=0,1,\cdots, N$,
    %with $S_h^n$ being a Linear Multistep Method, 
    it must lie on the boundary of the numerical method's absolute stability region.
\end{theorem}
\begin{proof}
    Since $\operatorname{Re}(\lambda)=0$, by Lemma~\ref{lemma:lmm},  $|z_n(\hat\lambda h)|\equiv 1.$ This means $\hat\lambda h\in \partial \mathcal{R}_A.$
\end{proof}

See Figure~\ref{fig:hat-lambda-vs-RA} that compares $\hat\lambda h$ to the region of absolute stability for AB2, AM2, and the Leap-Frog scheme.
See Figure~\ref{fig:AB-AM-dissipative} that compares $\hat\lambda$
and the ground truth $\lambda$ for various step sizes.

\subsection{Analysis of four two-step methods}

In this section, we focus on analyzing the two-step methods. For higher-order methods, readers can refer to the arXiv version of this manuscript for further details~\cite{Lu25}.

For two-step methods, we first analyze the co-existence of stable, unstable, and repeated roots by the quadratic formula. We summarize the analysis for four classical two-step methods in Table~\ref{tab:two-step-LMM-roots-summary}.

Next, with the knowledge of the characteristic roots, we study the 
optimal parameter, $\hat\lambda h$, of the learning problem under a simpler setup. 
We consider observation data of \( Z_n = \zeta^n \), where 
\[ \zeta = e^{a + i\theta},~a \le 0,\]
such that $|Z_n|\leq 1$.

The objective is to determine whether the numerical scheme could inadvertently learn an unstable dynamical system ($\operatorname{Re}(\hat{\lambda}) > 0$), potentially also reversing the direction of oscillations ($\operatorname{sign}(\operatorname{Im}(\hat{\lambda}))\neq \operatorname{sign}(\operatorname{Im}(\zeta))$).
 The analysis involves examining the rational functions over two-step methods,
\[
\hat{\lambda} h = \frac{\rho(\zeta)}{\kappa(\zeta)},
\]
which are listed in Table~\ref{tab:explicit_formula}. Then we delve into the analysis of the amplitude and rotation error over these formulas, summarized in Table~\ref{tab:lmm_two_step}.

From the table, we see that AB2 and BDF2 are not generally viable methods for learning dissipative systems with oscillations, because they may lead to the inference of expansive systems.

\begin{table}[h!]
\caption{Explicit formula of the learned parameter $\hat{\lambda}h$ for two-step methods.}
\label{tab:explicit_formula}
\medskip
\centering
\renewcommand{\arraystretch}{1.8}
\begin{tabularx}{\textwidth}{
>{\centering\arraybackslash}X
>{\centering\arraybackslash}X
>{\centering\arraybackslash}X
>{\centering\arraybackslash}X
}
\toprule
\textbf{AB2} & \textbf{AM2} & \textbf{BDF2} & \textbf{Leap-Frog}
 \\
\midrule
$\displaystyle \frac{2\zeta^2 - 2\zeta}{3\zeta - 1}$ &
$\displaystyle \frac{12\zeta^2 - 12\zeta}{5\zeta^2 + 8\zeta - 1}$  &
$\displaystyle \frac{3\zeta^2 - 4\zeta + 1}{2\zeta^2}$  &$\displaystyle \frac{\zeta^2 - 1}{2\zeta}$ \\
\bottomrule
\end{tabularx}
\end{table}

An important consequence is that: \emph{if one solves the learned dynamical system, $dz/dt=\hat\lambda z$, with other conditions, a perturbation of the perfect ones, one may obtain trajectories containing the spurious roots with growing amplitude.}

\subsubsection*{Phase errors in systems learned from Leap-Frog.}
The Leap-Frog method exhibits markedly different behavior in how it encodes the phase of oscillatory dynamics. Suppose $|\theta|$ is small,  \( \operatorname{Im}(\hat\lambda h) \) preserves the sign of \( \sin(\theta)\approx\theta \).
This means that the Leap-Frog method will infer rotation directions that are consistent with the actual system.

In the Leap-Frog method, 
\[
\hat\lambda h = \frac{\zeta^2 - 1}{2\zeta}
\quad \Rightarrow \quad
\operatorname{Im}(\hat\lambda h) = \cosh(a) \sin\theta.
\]
Notably, when $a=0$, the method will infer a system without any phase error, while for \( a < 0 \), the rotational speed is reduced by the dissipation in the actual system.

Thus, Leap-Frog exhibits high phase fidelity but at the cost of stability, as it always includes one unstable root.

\review{Table~\ref{tab:two-step-LMM-roots-summary} summarizes the characteristic roots of the four two-step methods. Together with Figure~\ref{fig:two_step}, this shows that 
among the methods considered, only the Leap-Frog scheme guarantees that the learned 
dynamics remain bounded for any initial data, since both of its characteristic roots 
lie on the unit circle when the data are conservative. For the remaining methods, the co-existence of roots inside and outside the unit disc means that realizations 
starting from initial conditions different from those used in \review{learn}ing may diverge, 
even when the \review{learn}ing loss is small.} 
\begin{table}
\caption{Stable and unstable roots of the characteristic polynomials of two-step methods.}\label{tab:two-step-LMM-roots-summary}\medskip
\centering
\renewcommand{\arraystretch}{1.5}
\begin{tabular}{@{}lll@{}}
\toprule
\textbf{Method} & \textbf{Stable/unstable roots coexist}  & \textbf{Repeated roots} \\
 & ($|\zeta_1|\geq 1$~\text{and}~$|\zeta_2|\leq 1$) & ($\zeta_1=\zeta_2$)\\
\midrule

AB2 &
Dark Gray region in Figure~\ref{fig:AB2_zeta}& 
if $\lambda h = \dfrac{-2\pm 4\sqrt{2}i}{9}$, \ $|\zeta_{1,2}|<1$\\

AM2 &
Dark Gray region in Figure~\ref{fig:AM2_zeta}&
if $\lambda h = \dfrac{-6\pm 4\sqrt{3}i}{7}$, \ $|\zeta_{1,2}|>1$ \\

BDF2 &
Dark Gray region in Figure~\ref{fig:BDF2_zeta}&
if $\lambda h = -\dfrac{1}{2}$, \ $|\zeta_{1,2}|<1$ \\

Leap-Frog &
always, since $\zeta_1 \zeta_2 = -1$ &
if $\lambda h = \pm i$, \ $|\zeta_{1,2}|=1$ \\

\bottomrule
\end{tabular}
\end{table}

\review{Table~\ref{tab:lmm_two_step} and Figures~\ref{fig:two_step_real} 
numerically show that none can reliably preserve the 
qualitative structure of the true dynamics for all data types. The only exception is 
the Leap-Frog scheme, which correctly yields $\operatorname{Re}(\hat{\lambda}) = 0$ 
for conservative data and conditionally for dissipative systems.}
\begin{table}[h!]
\caption{Sign of $\operatorname{Re}(\hat\lambda h)$ for two-step methods, inferred from root $\zeta = e^{a + i\theta}$.}\medskip
\label{tab:lmm_two_step}
\centering
\renewcommand{\arraystretch}{1.5}
\begin{tabularx}{\textwidth}{@{}lll@{}}
\toprule
\textbf{Method} & \textbf{Data with damping ($a < 0$)} & \textbf{Pure oscillatory data ($a = 0$)} \\
\midrule

AB2 &
Gray region in Figure~\ref{fig:AB2_real}&
\correction{$\operatorname{Re}(\hat\lambda h) < 0$ for all $\zeta \in S^1$} \\
\hline

AM2 &
Gray region in Figure~\ref{fig:AM2_real} &
\correction{$\operatorname{Re}(\hat\lambda h) < 0$ for all $\zeta \in S^1$} \\

\hline

BDF2 &
Gray region in Figure~\ref{fig:BDF2_real}&
\correction{$\operatorname{Re}(\hat\lambda h) > 0$ for all $\zeta \in S^1$} \\

\hline
Leap-Frog &
$\operatorname{Re}(\hat\lambda h) > 0$ if $\cos(\theta) < 0$  &
$\operatorname{Re}(\hat\lambda h) = 0$ for all $\zeta \in S^1$ \\
 &
$\operatorname{Im}(\hat\lambda h) = \cosh(a) \sin\theta$  & small $|\theta|\implies$ small phase error
 \\
%\hline

%\hline\hline
\bottomrule
\end{tabularx}
\end{table}
\begin{figure}[!ht]
    \centering
    \begin{subfigure}{0.32\linewidth}
        \includegraphics[width=\linewidth]{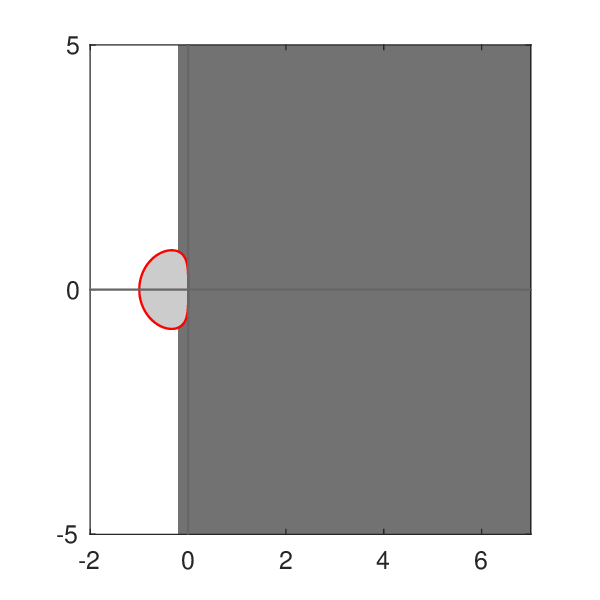}
    \caption{AB2}
    \label{fig:AB2_zeta}
    \end{subfigure}
    \begin{subfigure}{0.32\linewidth}
        \includegraphics[width=\linewidth]{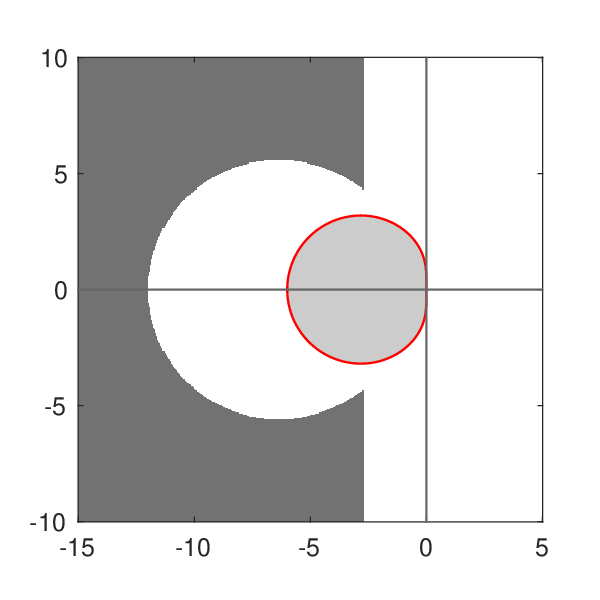}
    \caption{AM2}
    \label{fig:AM2_zeta}
    \end{subfigure}
    \begin{subfigure}{0.32\linewidth}
        \includegraphics[width=\linewidth]{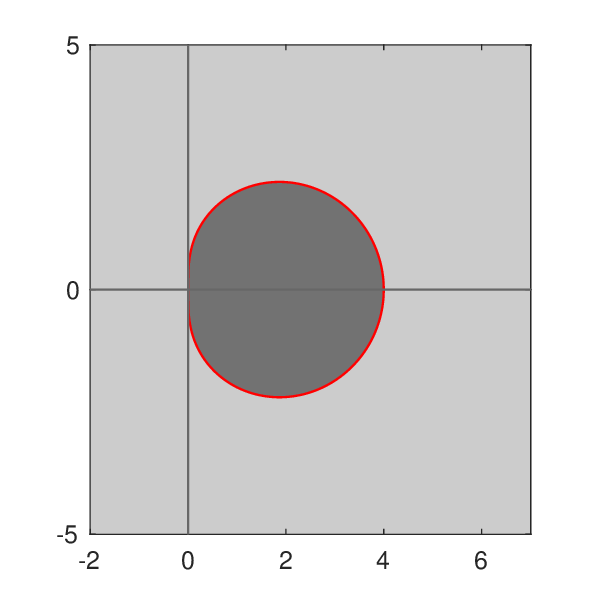}
    \caption{BDF2}
    \label{fig:BDF2_zeta}
    \end{subfigure}
    \caption{The dark gray region in each panel marks the region $\mathcal{R}_P$. The regions in light gray are the region of absolute stability $\mathcal{R}_A$ of each integrator, which is enclosed by the red curve. 
} \label{fig:two_step}\end{figure}

\begin{figure}[!ht]
    \centering
    \begin{subfigure}{0.32\linewidth}
        \includegraphics[width=\linewidth]{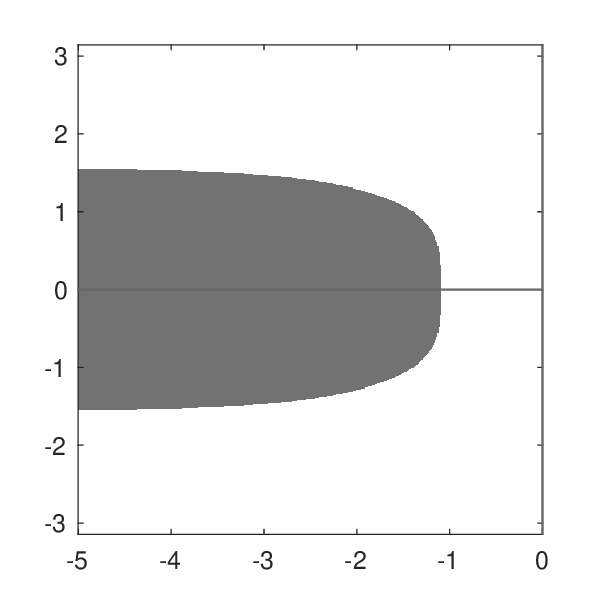}
    \caption{AB2}
    \label{fig:AB2_real}
    \end{subfigure}
    \begin{subfigure}{0.32\linewidth}
        \includegraphics[width=\linewidth]{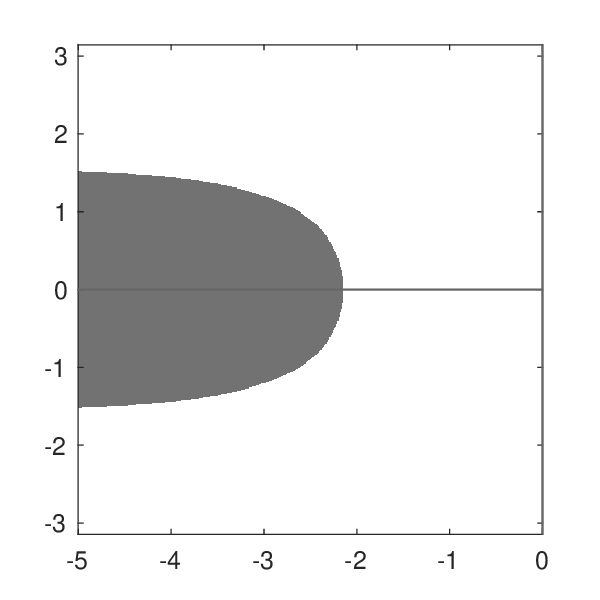}
    \caption{AM2}
    \label{fig:AM2_real}
    \end{subfigure}
    \begin{subfigure}{0.32\linewidth}
        \includegraphics[width=\linewidth]{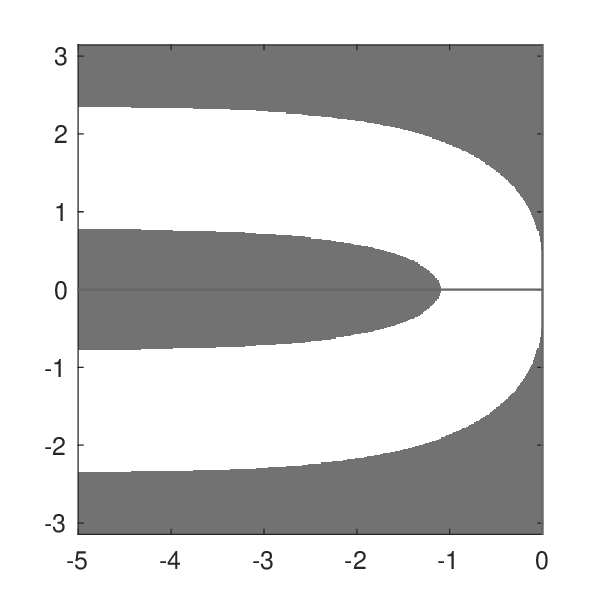}
    \caption{BDF2}
    \label{fig:BDF2_real}
    \end{subfigure}
    \caption{The dark gray region in each panel marks the set of complex numbers $z = a+i\theta$ where $\operatorname{Re}(\hat\lambda h) = \operatorname{Re}({\rho(e^z)}/{\kappa(e^z)})>0$. 
}\label{fig:two_step_real}\end{figure}

\begin{figure}[!ht]
    \centering
    \begin{subfigure}{0.32\linewidth}
        \includegraphics[width=\linewidth]{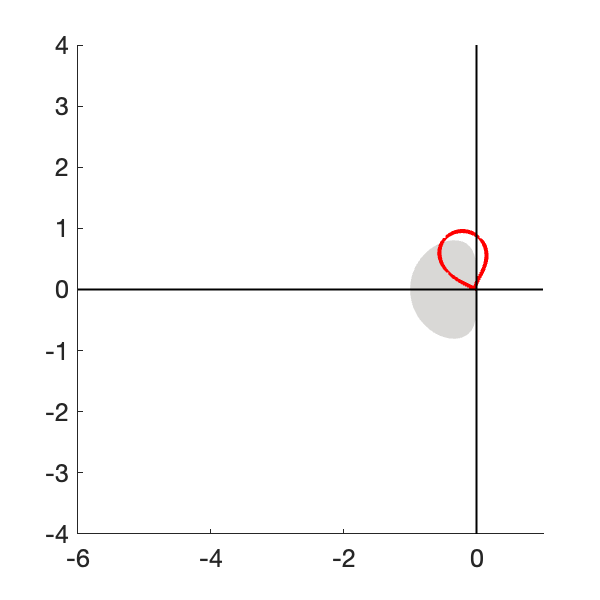}
    \caption{AB2}
    % \label{fig:AB2_obj}
    \end{subfigure}
    \begin{subfigure}{0.32\linewidth}
        \includegraphics[width=\linewidth]{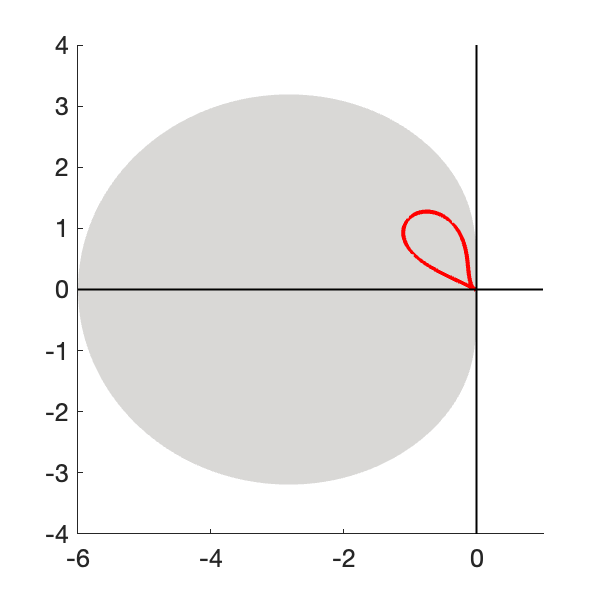}
    \caption{AM2}
    \end{subfigure}
    \begin{subfigure}{0.32\linewidth}
        \includegraphics[width=\linewidth]{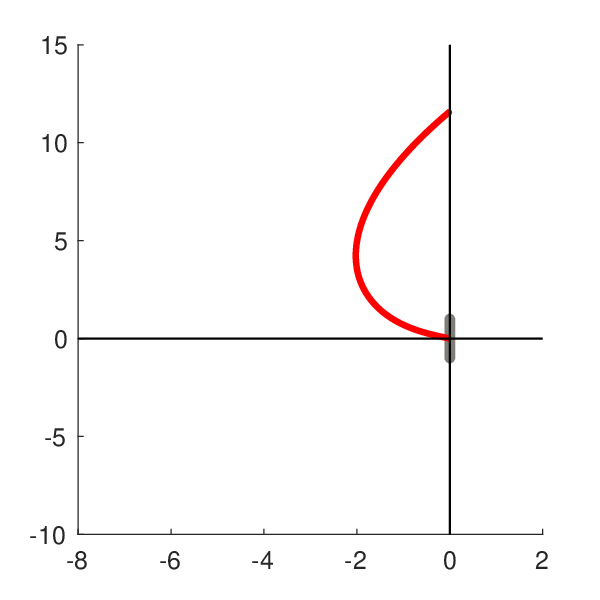}
    \caption{Leap-Frog}
    % \label{fig:AM2_obj}
    \end{subfigure}
    \caption{$\hat\lambda h$ versus the region of absolute stability. The gray region in each plot indicates the stability region of the selected method, while the red curve depicts the learning result, $\hat{\lambda} h = \rho(e^{\lambda h})/\kappa(e^{\lambda h})$ for $h\in(0, \pi/4)$, where $\lambda = -4+2i$.  } \label{fig:hat-lambda-vs-RA}
\end{figure}

\begin{figure} %[!ht]
\centering{
    \begin{subfigure}[t]{.3\textwidth}
         \centering
         \includegraphics[width=\linewidth]{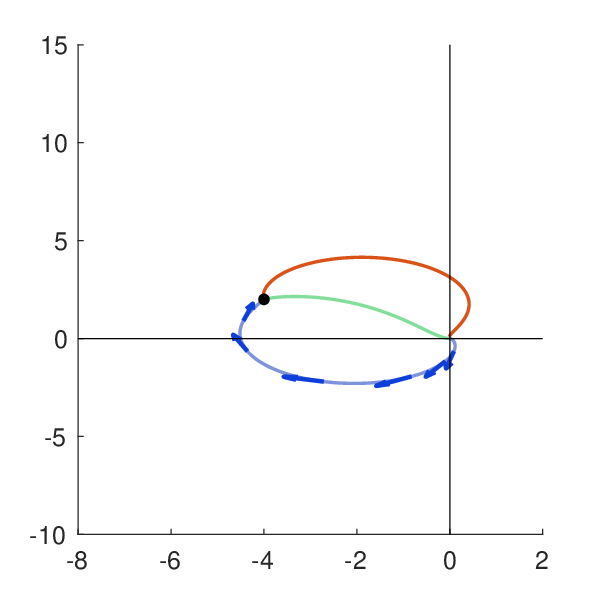}
         \caption{$\hat\lambda$ from ABk.}
         % \label{fig:enter-label}
    \end{subfigure}\hspace{1em}
    \begin{subfigure}[t]{.3\textwidth}
        \centering
        \includegraphics[width=\linewidth]{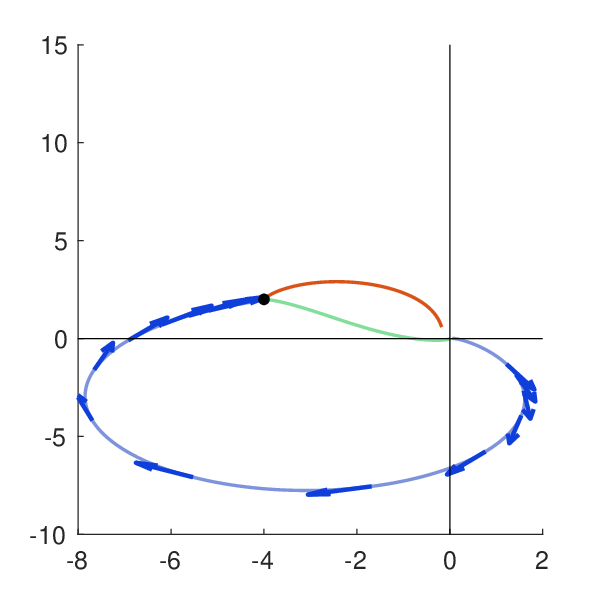}
        \caption{$\hat\lambda$ from AMk.}
        % \label{fig:enter-label}
    \end{subfigure}\hspace{1em}
    \begin{subfigure}[t]{.3\textwidth}
        \centering
        \includegraphics[width=\linewidth]{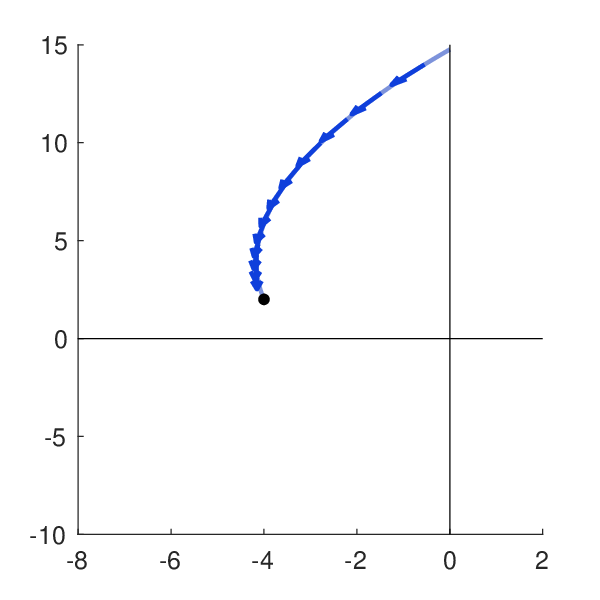}
        \caption{$\hat\lambda$ from Leap-Frog.}
        % \label{fig:enter-label}
    \end{subfigure}
    }
    \caption{  $\hat\lambda\equiv \hat\lambda(h) = h^{-1}\rho(e^{\lambda h})/ \kappa(e^{\lambda h})$
    for $h \in (0, \pi/4)$.
The actual value, $\lambda = -4 + 2i$, is indicated by a triangle in each plot. The arrows on selected curves show the trend of $\hat\lambda$ as  $h\rightarrow 0$ for the respective methods.
 The blue curves represent the ABk or AMk methods for $k = 2$, the green curves for $k = 3$, and the red curves for $k = 4$. The right plot corresponds to the Leap-Frog scheme.}\label{fig:AB-AM-dissipative} 
\end{figure} 

According to Lemma~\ref{lemma:lmm}, the learned dynamic resides on the boundary of each method’s stability region. In particular, the Leap-Frog scheme is absolutely stable in the interval $(-1,1)$ on the imaginary axis. Consequently, the learned dynamical system, $dz/dt = \hat\lambda z,$ will be conservative if
the data points come from a conservative dynamical system, and $0\le |\hat\lambda|< 1/h$, for any step size $h$. 
\emph{This implies that the frequency in the learned dynamical system has to be smaller than the actual frequency.}
% In contrast,  the stability regions of the Adams-family methods intersect the imaginary axis at a finite number of points; some even extend to the right half of the complex plane. 
% Consequently, the learned dynamics will typically exhibit dissipative behavior if the stability region stays on the left half-plane, and may even be expansive when the stability region overlaps the right half-plane. 
% Our analysis suggests that the Leap-Frog scheme is better suited for learning conservative dynamics within the LMM framework.
 
\subsection{Noisy data}
We study the dynamics inferred by the Leap–Frog and noisy observations. The point is that noise in data observed from a dissipative system may cause the learned discrete dynamical system to become unstable. The instability is induced by the existence of spurious roots of the corresponding characteristic polynomial.

Consider using the Leap–Frog to learn $\hat\lambda h$, the solution at $n$ steps can be written as
\[
  c_1\zeta_1^{\,n}+c_2\zeta_2^{\,n},
  \qquad
  \zeta_{1,2}
  :=\hat\lambda h\pm\sqrt{1+(\hat\lambda h)^{2}},
\]
with coefficients $c_1,c_2$ determined by the initial values
$z_0$ and $z_1$.

\begin{theorem}[\correction{Vanishing of the unstable-mode coefficient and the noiseless limit of $\hat\lambda h$}]
  Let $N$ be sufficiently large and let $\hat\lambda h$ be the global minimiser of
  \eqref{LMM-learning-problem-v2} obtained with the Leap–Frog scheme.
  Suppose the observations
  \[
    e^{\lambda n h}+\epsilon_n,\qquad n=1,\dots,N,
  \]
  are uniformly bounded and include zero-mean noise
  $\epsilon_n$ of small variance, where
  $\operatorname{Re}(\lambda)<0$ and $h$ is the sampling step.\\
  Assume that
  \[
      |\zeta_1|>1,\qquad|\zeta_2|\le1 .
  \]
  \correction{Then $c_1=0$. Moreover, in the noiseless limit,
  \[
      \hat\lambda h \;=\; \sinh(\lambda h),
  \]
  so, writing $\lambda h = a+i\theta$ with $a=\operatorname{Re}(\lambda h)<0$,
  \[
      \operatorname{Re}(\hat\lambda h) \;=\; \sinh(a)\cos\theta,
      \qquad
      \operatorname{Im}(\hat\lambda h) \;=\; \cosh(a)\sin\theta,
  \]
  which is negative precisely when $|\theta|<\pi/2$ and can be positive otherwise.}
\end{theorem}

\begin{proof}
    Assume that both coefficients are non-zero,
    $c_1\neq0$ and $c_2\neq0$.
    By the triangle inequality,
    \[
      \bigl|c_1\zeta_1^{\,n}+c_2\zeta_2^{\,n}- e^{\lambda n h}+\epsilon_n \bigr|
       \ge
      \Bigl|\,|c_1|\,|\zeta_1|^{\,n}
            -\bigl|c_2\zeta_2^{\,n}+e^{\lambda n h}\bigr|
            -|\epsilon_n|\,\Bigr|.
    \]
    Because $|c_2\zeta_2^{\,n}+e^{\lambda n h}|$ is uniformly bounded,
    $|\epsilon_n|$ is small, and $|\zeta_1|>1$, the right-hand side
    grows without bound as $n\to\infty$.
    Hence, the residual cannot be minimized globally, contradicting the optimality
    of $\hat\lambda h$.
    Therefore, we must have $c_1=0$.

    \correction{With $c_1=0$, the learned sequence reduces to $c_2\zeta_2^{\,n}$. In the noiseless limit $\epsilon_n\equiv 0$, matching $c_2\zeta_2^{\,n}=e^{\lambda n h}$ for all $n$ forces $\zeta_2=e^{\lambda h}$ and $c_2=1$. Since the Leap-Frog characteristic polynomial $\zeta^2-2(\hat\lambda h)\zeta-1=0$ has $\zeta_1\zeta_2=-1$, this determines $\zeta_1=-e^{-\lambda h}$, and hence
    \[
      \hat\lambda h \;=\; \frac{\zeta_1+\zeta_2}{2}
      \;=\; \frac{e^{\lambda h}-e^{-\lambda h}}{2}
      \;=\; \sinh(\lambda h).
    \]
    Writing $\lambda h=a+i\theta$ and separating real and imaginary parts yields the stated formulas for $\operatorname{Re}(\hat\lambda h)$ and $\operatorname{Im}(\hat\lambda h)$. Since $a<0$ implies $\sinh(a)<0$, the sign of $\operatorname{Re}(\hat\lambda h)$ is opposite to that of $\cos\theta$, so $\operatorname{Re}(\hat\lambda h)<0$ exactly when $|\theta|<\pi/2$.}
\end{proof}

\subsubsection*{An example of "hidden" spurious mode}
Due to the presence of noise in the data, the coefficient $c_1$ for the spurious root does not necessarily vanish. 
To illustrate this fact, we estimate $\hat\lambda h$ from synthetic observations $Z_n = Z_0e^{\lambda n h} + \epsilon_n$ generated from the linear system with $\lambda = -2 + i$ with the sampling step size $h = 0.1$ and the starting point $Z_0$; here $\epsilon_n$ is the noise with zero mean and variance $0.01$, \review{$n = 1\cdots N  = 10, 20, 50$. \\
That is, we get the learned dynamic from minimizing the following objective function:
\begin{equation}
\left\{
\begin{aligned}
        \min_{\xi\in\mathbb{C}}~
        \dfrac{1}{N}&\sum_{n=1}^N 
        | \sum_{j=1}^2 c_{j} \zeta^{n}_{j}(\xi) - Z_{n} |^2,~~~n=1,2,\cdots,N,\\
   & c_1 + c_2  = Z_0,~~~ c_1\zeta_1+c_1\zeta_2 = Z_1.
\end{aligned}
    \right.
\end{equation}
}
In Table ~\ref{tab:spurious_roots} we report the resulting $c_1$ and $c_2$ values for several choices of $N$.

\review{If the coefficient, $c_1$, linked to the spurious root is non-zero, the learned discrete dynamics can diverge from the actual system at later times. We integrate the learned dynamics from $N = 50$ using the same numerical scheme and the same $h$ up to $t = 10$ and show the results in Figure \ref{fig:diverge_imp}. In such a case, the spurious root eventually dominates, and the numerical solution 
eventually diverges.}

\begin{table}[ht]
\caption{Estimated coefficients $c_1$ and $c_2$ for different sample sizes $N$.}
\medskip
  \centering
  \begin{tabular}{crrrr}
    \toprule
    {$N$} & {$c_1$} & {$c_2$} & {$|\zeta_1|$} &{$|\zeta_2|$}\\ \midrule
    10 &  $-6.0912\times 10^{-3}$ & 1.0007 &   1.1602 & 0.8619 \\
    20 &  $-1.6658\times 10^{-3}$ & 0.9962    & 1.1602 &0.8732\\
    50 &  $-7.8910\times 10^{-6}$ & 0.9946  & 1.1602 &0.8732\\ \bottomrule
  \end{tabular}
  
  \label{tab:spurious_roots}
\end{table}

\begin{figure}[!ht]
  \centering
  \begin{subfigure}[t]{0.3\linewidth}
    \includegraphics[width=\linewidth]{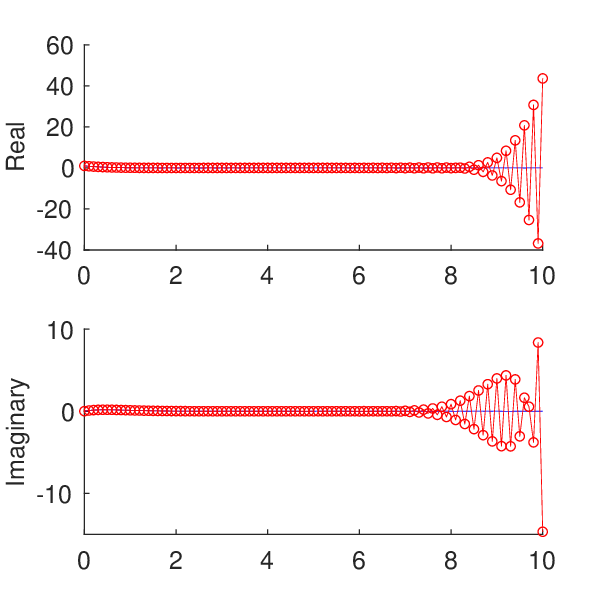}
    \caption{The values generated by the learned discrete process: $c_1\zeta_1^n+c_2\zeta_2^n$.}
    \label{fig:diverge_imp:a}
  \end{subfigure}
  \hspace{2em}
  \begin{subfigure}[t]{0.3\linewidth}
    \includegraphics[width=\linewidth]{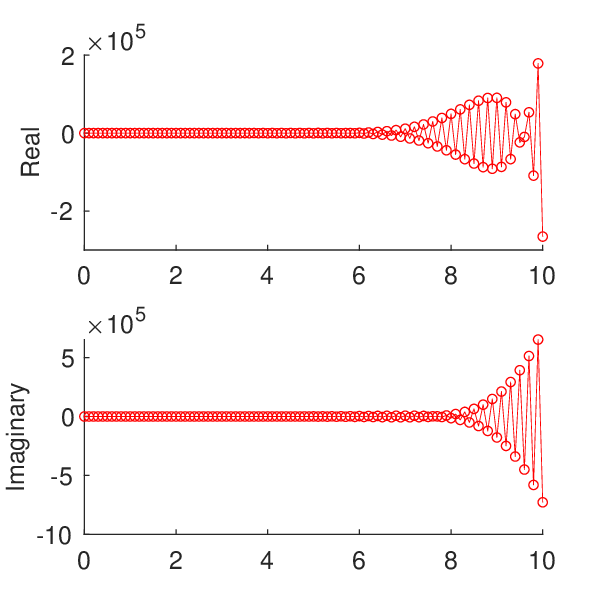}
    \caption{The evolution of $\zeta_1^n$}
    \label{fig:diverge_imp:b}
  \end{subfigure}
  \caption{The $x$-axis represents the physical time. The blue curves represent the observed data. On the left, the red curves depict the numerical solution of the differential equation $z^\prime= \hat\lambda  z(t)$ obtained using the Leap-Frog scheme with step size $h = 0.1$ up to time $t = 10$. On the right, the plots display the evolution of $\zeta_1^n$.
}
  \label{fig:diverge_imp}
\end{figure}

\section{\revmark{From linear to nonlinear dynamics}}\label{sec:nonlinear-learning}

\revmark{Sections~\ref{sec:one-step} and~3 analyze the scalar model $z'=\lambda z$. To connect that analysis to the nonlinear experiments in Section~\ref{sec:numerical-examples}, we adopt the planar damped Hamiltonian system
\begin{equation}\label{eq:damped-ham}
\dot q = p,\qquad \dot p = -\nabla V(q)-\gamma p,
\qquad \gamma\ge 0,
\end{equation}
as a model problem. When $\gamma=0$, the flow is conservative with Hamiltonian $H(q,p)=\tfrac12|p|^2+V(q)$; when $\gamma>0$, oscillatory structure persists but the flow is dissipative. Both regimes covered by the scalar theory are thus captured by the complex eigenvalues of the local linearization, which makes the damped Hamiltonian a natural vehicle for transferring the artifacts identified in Sections~\ref{sec:one-step} and~3 to learning nonlinear dynamical systems from trajectory data.}

\revmark{Suppose now that the learned model is agnostic to Hamiltonian structure and is given by a general planar vector field
\[
\dot x=\mathbf g(x),\qquad x=(u,v),\qquad \mathbf g=(g_1,g_2).
\]
Given an observed trajectory $\{x_n\}_{n=0}^N$, one fits $\mathbf g$ through a chosen one-step map
\[
x_{n+1}=S_h(x_n;\mathbf g).
\]
To understand the local qualitative behavior of the learned dynamics, fix an observed state $x_n$ and linearize the \emph{learned} vector field there:
\[
\dot w=\hat A_n w,\qquad \hat A_n:=D\mathbf g(x_n),
\qquad
\hat\lambda_j:=\text{eigenvalues of }\hat A_n.
\]
The local discrete dynamics is then governed by the Jacobian of the one-step map,
\[
DS_h(x_n;\mathbf g),
\]
which, to leading order as $h\to 0$, is obtained by applying the chosen numerical integrator to the linearized system $\dot w=\hat A_n w$. If $R$ denotes the stability function of the method (the same object denoted $p$ in Table~\ref{tab:characteristic-poly}), then the local eigenvalues of $DS_h(x_n;\mathbf g)$ satisfy the asymptotic identity
\begin{equation}\label{eq:nl-local-eig}
\mu_j=R(\hat\lambda_j h)+\mathcal{O}(h)\qquad\text{as }h\to 0.
\end{equation}
In parallel, the Jacobian of the \emph{true} flow over one step has eigenvalues $e^{\lambda_j h}+\mathcal{O}(h)$ as $h\to 0$, where $\lambda_j$ denote the true local eigenvalues (eigenvalues of $D\mathbf f(x_n)$ for the ground-truth vector field $\mathbf f$). Matching the discrete one-step evolution to the observed data therefore, couples the two spectra through the asymptotic relation
\begin{equation}\label{eq:nl-master}
R(\hat\lambda_j h)=e^{\lambda_j h}+\mathcal{O}(h)\qquad\text{as }h\to 0,
\end{equation}
which is the scalar identity of Section~\ref{sec:one-step} applied pointwise along the trajectory, with the $\mathcal{O}(h)$ correction capturing the nonlinear Taylor terms of $\mathbf g$ and $\mathbf f$ near $x_n$.}

\revmark{For the damped Hamiltonian model \eqref{eq:damped-ham}, the relevant local linearization is
\[
A=
\begin{pmatrix}
0 & 1\\
-\kappa & -\gamma
\end{pmatrix},
\]
where $\kappa$ is the local curvature of the potential and $\gamma$ is the damping coefficient. Its eigenvalues are
\[
\lambda_\pm=\frac{-\gamma\pm\sqrt{\gamma^2-4\kappa}}{2}.
\]
In the underdamped regime $\gamma^2<4\kappa$, this becomes
\begin{equation}\label{eq:nl-underdamped-eig}
\lambda_\pm=-\frac{\gamma}{2}\pm i\omega,
\qquad
\omega=\frac{\sqrt{4\kappa-\gamma^2}}{2}.
\end{equation}
Thus, the correct local behavior is oscillation with decay rate $e^{-\gamma h/2}$. In the conservative limit $\gamma=0$, these eigenvalues reduce to the purely imaginary pair
\[
\lambda_\pm=\pm i\sqrt{\kappa},
\]
so the correct local behavior is neutral oscillation.}

{\color{black}
Consider first the conservative limit $\gamma=0$, for which the true local eigenvalues are purely imaginary, $\lambda_\pm=\pm i\sqrt{\kappa}$, and $|e^{\lambda_\pm h}|=1$. Equation~\eqref{eq:nl-master} then forces $|R(\hat\lambda_\pm h)|=1+\mathcal{O}(h)$ as $h\to 0$, placing the learned $\hat\lambda_\pm h$ on the stability boundary $\{|R(z)|=1\}$ to leading order. Applying Corollary~\ref{cor:unified_cons} (with $\operatorname{Re}\lambda=0$) to each of the three integrators then gives, to leading order in $h$:
\begin{itemize}
\item \textbf{Explicit Euler} (stability boundary $|z+1|=1$, lying in  $\mathbb{C}\cap\mathbb{R}^-$): $\operatorname{Re}\hat\lambda_\pm<0$. The learned local dynamics \emph{contracts} toward $x_n$, artificially damping a truly neutral mode.

\item \textbf{Implicit Euler} (stability boundary $|z-1|=1$, lying in $\mathbb{C}\cap\mathbb{R}^+$): $\operatorname{Re}\hat\lambda_\pm>0$. The learned local dynamics \emph{expands away} from $x_n$; this is the mechanism behind the expansive Lotka--Volterra trajectories reported in Section~\ref{sec:numerical-examples} with implicit Euler.

\item \textbf{Implicit midpoint} (stability boundary coinciding with the imaginary axis): $\operatorname{Re}\hat\lambda_\pm=0$. The learned local dynamics preserves neutral oscillation to leading order, with a phase error of order $h^2$ and no spurious contraction or expansion.
\end{itemize}
}

\revmark{In the underdamped regime $\gamma>0$ with $\lambda_\pm=-\gamma/2\pm i\omega$, the data satisfy $|e^{\lambda_\pm h}|=e^{-\gamma h/2}<1$, so to leading order the learned $\hat\lambda_\pm h$ lies on the sub-level set $\{|R(z)|=e^{-\gamma h/2}\}$ inside the stability region. The signs of $\operatorname{Re}\hat\lambda_\pm$ for each of the three integrators follow from Corollary~\ref{cor:unified_cons} (with $\operatorname{Re}\lambda<0$); to see the magnitude of the bias introduced by each stability-region geometry, we invert \eqref{eq:nl-master} and Taylor-expand in $h$:
\[
\begin{aligned}
\text{Explicit Euler:}\quad
&\operatorname{Re}\hat\lambda_\pm=-\tfrac{\gamma}{2}-\tfrac{h}{2}\correction{\bigl(\omega^{2}-\tfrac{\gamma^{2}}{4}\bigr)}+\mathcal{O}(h^2),\\
\text{Implicit Euler:}\quad
&\operatorname{Re}\hat\lambda_\pm=-\tfrac{\gamma}{2}+\tfrac{h}{2}\correction{\bigl(\omega^{2}-\tfrac{\gamma^{2}}{4}\bigr)}+\mathcal{O}(h^2),\\
\text{Implicit midpoint:}\quad
&\hat\lambda_\pm=\lambda_\pm-\tfrac{h^2}{12}\lambda_\pm^{\,3}+\mathcal{O}(\correction{h^4}).
\end{aligned}
\]
The explicit-Euler learned eigenvalue is always contractive at leading order: the sign of the learned dissipation matches the true sign, but the magnitude is biased toward \emph{over}-damping by the $\mathcal{O}(h\omega^2)$ term coming from the stability-region geometry. The implicit-Euler learned eigenvalue is not guaranteed to have the correct sign: whenever \correction{$h(\omega^{2}-\gamma^{2}/4)>\gamma$}, the leading-order $\operatorname{Re}\hat\lambda_\pm$ becomes positive, producing an expansive learned local dynamics even when the truth is dissipative---making quantitative the ``either expansive or dissipative'' behavior noted in Corollary~\ref{cor:unified_cons}. The implicit-midpoint learned eigenvalue tracks the true eigenvalue to $\mathcal{O}(h^2)$ in magnitude and exactly in sign to leading order, for every $(\gamma,\omega)$.}

\revmark{The prescriptive conclusion of Section~\ref{sec:nonlinear-learning}, then, turns on a geometric property: where each integrator's $|R|=1$ level set sits in the complex plane. Among the three, implicit midpoint is the unique method whose stability boundary coincides with the imaginary axis. In consequence, the learned local spectrum sits on the correct side of the imaginary axis in both the conservative and the weakly dissipative regimes, so a nonlinear learned vector field fit under the implicit midpoint rule inherits the correct qualitative local behavior---neutral oscillation where the truth is conservative, contraction where the truth is dissipative---while explicit Euler biases the learned spectrum into the left half plane and implicit Euler biases it into the right, regardless of the true local dissipation.}

\section{Numerical experiments}\label{sec:numerical-examples}
This section presents some further numerical studies to verify and further illustrate our findings on non-negligible numerical effects that lead to
mistakenly inferring unstable systems, destruction of orbits, or limiting cycles, and oscillatory systems with large nonlinear phase errors.
\review{Although the theoretical analysis in Sections~2 and~3 is developed for the  scalar linear system $z' = \lambda z$, its conclusions extend locally to
nonlinear systems through linearization. Specifically, along any sampled
trajectory $\gamma(t)$, the local behavior of the learned vector field $g_h$
is governed by the Jacobian $J_{g_h}(\gamma)$. By examining the eigenvalues
of this Jacobian, the scalar theory applies locally at each sampled point along the
trajectory, providing the key link between the linear analysis and the
nonlinear examples presented below.} 

Some of the examples below compare the Implicit Midpoint Rule and the Implicit Trapezoidal Rule; these two schemes coincide on the scalar linear problem (where they share the characteristic polynomial recorded in Table~\ref{tab:characteristic-poly}) but differ on nonlinear systems. For a nonlinear autonomous system $y^\prime = f(y)$ their generic forms are
\begin{equation}
  y_{n+1} = y_n + h\,
  f(\frac{y_n + y_{n+1}}{2}),
  \label{eq:implicit_midpoint}
\end{equation}
\begin{equation}
  y_{n+1} = y_n + \frac{h}{2}\bigl(f(y_n) + f(y_{n+1})\bigr),
  \label{eq:implicit_trapezoidal}
\end{equation}
for the Implicit Midpoint Rule and the Implicit Trapezoidal Rule, respectively.

\subsection{Recovering coefficients in a convection-diffusion equation}\label{sec:convection-diffusion-ex}

This example illustrates how numerical artifacts can affect coefficient recovery.  

In this example, we consider the recovery of the coefficients $a$ and $\epsilon$ in the following convection-diffusion equation with the periodic boundary condition,
\begin{equation}\label{eq:convection-diffusion-eqn}
\begin{aligned}
        &&u_t = a u_x+\epsilon u_{xx},~~x\in[0, 1), t\in\mathbb{R}^+,\\
    &&u(x, t=0) = u_0(x),
\end{aligned}
\end{equation}
with $a=2$ and $\epsilon=0.01$.

We consider the spectral method and use data $u_j^n$ on a Cartesian
grid, with $x_j=j\Delta x$, $j=0,1,2,\cdots, 256$ , and $t_n = nh$, for $n=1,2,\cdots, 5\times10^4$. 
The data are computed by the Matlab's \texttt{ODE45} with $\texttt{AbsTol} = 10^{-12}$ and $\texttt{RelTol} = 10^{-12}$.

We employ a spectral discretization is space with grid points $x_j=j\Delta x$, $j=0,\dots,256$, and time steps $t_n=nh$ for $n=1,\dots,5\times10^4$.
In the Fourier domain, 
\eqref{eq:convection-diffusion-eqn} becomes 
\begin{equation*}
    \frac{d}{dt} \hat u(k,t) = \underbrace{(i 2\pi k a- (2\pi k)^2 \epsilon)}_{\lambda_k} \hat u(k,t). 
\end{equation*}
For any candidate system, the time integration is computed by Explicit Euler, Implicit Euler, or the Implicit Midpoint rule with a fixed stepsize.
The reference data is computed by the same spectral discretization in space and MATLAB’s adaptive \texttt{ODE45} solver in time, with tolerances $\texttt{AbsTol}=\texttt{RelTol}=10^{-12}$.

Let $\hat\lambda_k$ be the recovered coefficient, and set $\hat{a} = \text{Im}(\hat\lambda )/2\pi k$ and $\hat{\epsilon}= \text{Re}(\hat\lambda) /(2\pi k)^2$. 
Table~\ref{tab:combined_convection_diffusion} compares the results obtained using the Explicit Euler, Implicit Euler, and Implicit Midpoint Rule for two initial conditions:
\[
u_0(x)=\cos(2\pi kx),\quad k=1,~10.
\]
This experiment supports Theorem~\ref{thm:euler_stab} and Corollary~\ref{cor:unified_cons}.  
In particular, the Implicit Euler scheme, which is stable in the right half-plane, can produce a diffusion coefficient with the incorrect sign.  
Moreover, Table~\ref{tab:combined_convection_diffusion} shows that $\hat{a}<a$ for the Explicit Euler method, while $\hat{a}>a$ for the Implicit Euler method, consistent with the analysis in Section~\ref{sec:one-step}.

%% I finished the paragraph on p.33
\begin{table}[htbp]
\centering
\caption{Learned coefficients $\hat a$ and $\hat\epsilon$ for 
\eqref{eq:convection-diffusion-eqn}, true values \(a=2\), \(\epsilon=0.01\).}
\label{tab:combined_convection_diffusion}
\smallskip
\begin{tabular}{@{}l 
                c 
                S[table-format=1.4] S[table-format=1.4] 
                S[table-format=1.4] S[table-format=1.4]@{}}
\toprule
\multirow{2}{*}{\bfseries Method} 
  & \multirow{2}{*}{\boldmath\(h\)} 
  & \multicolumn{2}{c}{\bfseries \(k=1\)} 
  & \multicolumn{2}{c}{\bfseries \(k=10\)} \\
\cmidrule(lr){3-4} \cmidrule(lr){5-6}
  &  
  & {\(\hat a\)} & {\(\hat\epsilon\)} 
  & {\(\hat a\)} & {\(\hat\epsilon\)} \\
\midrule
\multirow{2}{*}{Explicit Euler} 
  & \(10^{-2}\) & 1.9869 & 0.0299 & 1.0281 & 0.0199 \\
  & \(10^{-3}\) & 1.9992 & 0.0120 & 1.9181 & 0.0117 \\
\midrule
\multirow{2}{*}{Implicit Euler} 
  & \(10^{-2}\) & 2.0026 & {\bfseries -0.0100} & 2.2552 & {\bfseries -0.0133} \\
  & \(10^{-3}\) & 2.0007 & 0.0080         & 2.0748 & 0.0081         \\
\midrule
\multirow{2}{*}{Implicit Midpoint} 
  & \(10^{-2}\) & 2.0026 & 0.0100 & 2.1765 & 0.0146 \\
  & \(10^{-3}\) & 2.0000 & 0.0100 & 2.0018 & 0.0100 \\
\bottomrule
\end{tabular}
\end{table}

\begin{figure} %[!htbp]
    \centering
\begin{subfigure}[t]{0.3\textwidth}
    \centering
    \includegraphics[width=\linewidth]{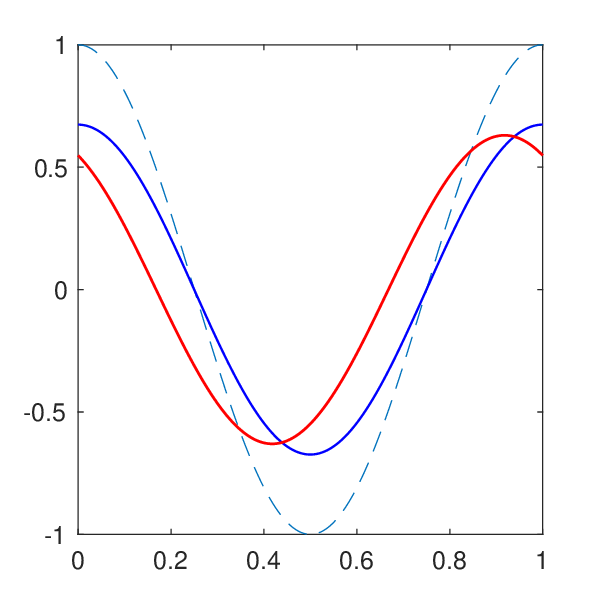}
    \caption{Explicit Euler}
\end{subfigure}
\begin{subfigure}[t]{0.3\textwidth}
    \centering
    \includegraphics[width=\linewidth]{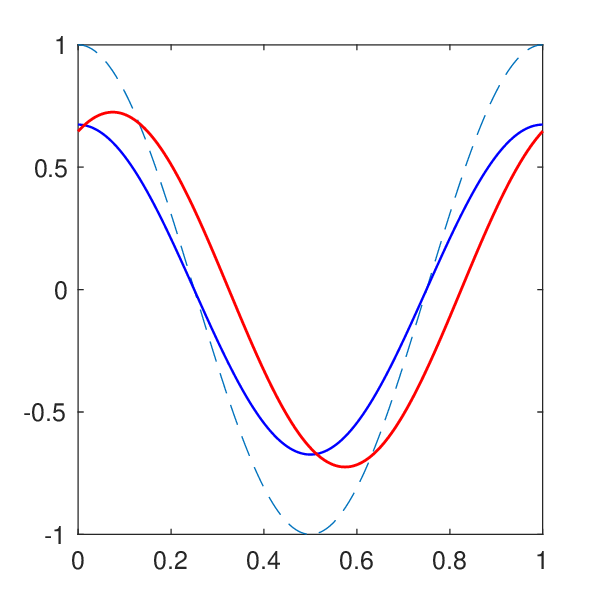}
    \caption{Implicit Euler}
\end{subfigure}
\begin{subfigure}[t]{0.3\textwidth}
    \centering
    \includegraphics[width=\linewidth]{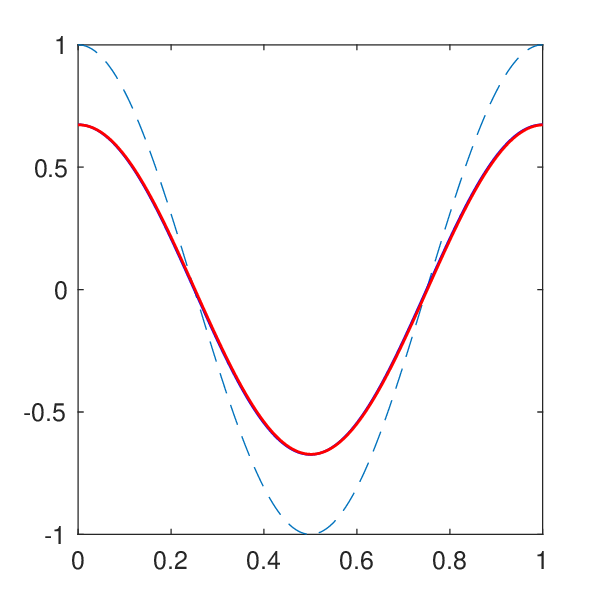}
    \caption{Implicit Midpoint}
\end{subfigure}
\caption{Profiles of the learned convection–diffusion solution at $T=1$. The blue dashed line denotes the initial condition $\cos(2\pi x)$, the solid blue line is the true solution, and the red line represents the learned solution. 
}\label{fig:convection_phase}
\end{figure}

The imaginary part of $\hat\lambda$ is directly associated with the propagation speed $a$ in the equation. 
Figure~\ref{fig:convection_phase} provides a visualization of the discrepancy in the propagation speed between the learned and true dynamics. The plots are the solutions of the equation, defined by the 
coefficients learned with $k = 10$ and $h = 10^{-3}$, computed by different time integrators, starting from 
 the initial condition $u_0 = \cos(2\pi x)$.
Figure~\ref{fig:convection_phase} supports the conclusion in Table~\ref{tab:im-lambdah-one-step-trimmed}, showing that the Explicit Euler method learns a dynamic with a slower propagation speed, whereas the Implicit Euler method, by contrast, learns a dynamic with an advanced propagation speed.

\subsection{Neural ODE models for the Lotka-Volterra system}
We study the learning of orbital trajectories in a nonlinear predator-prey problem, using a common Neural ODE approach.

The Lotka-Volterra dynamic describes the interaction between two species,
\begin{equation*}
    \left\{
    \begin{array}{rcl}
        \dfrac{dx}{dt} &=& x\left(1-\dfrac{y}{\nu}\right),\\
        \\
        \dfrac{dy}{dt} &=& \dfrac{y}{\nu}\left(1-x\right),\\
    \end{array}
    \right.
\end{equation*}
where $\nu>0$. 
We will construct neural networks from data to approximate the vector field of this system (the right-hand side of the above equation). 

{We set $\nu = 1$ in this study, and generate trajectories by integrating the dynamic system using MATLAB's \texttt{ODE45} solver and setting \texttt{AbsTol} and \texttt{RelTol} to $10^{-12}$. We sample trajectory snapshots at $t_n=nh$ with $h = 10^{-1}, 10^{-2}$, and $0\le t_n\le 12.$ }
{That is, dataset $\mathcal{D}_h$ is sampled over the time interval $[0, 12]$ with timesteps $h = 10^{-1}$ and $10^{-2}$. The trajectories of the learned dynamics are generated using the 4th-order Runge-Kutta method over the time interval $[0, 40]$ with $\Delta t = 10^{-4}$.}

\begin{figure} %[!htbp]
    \centering
\begin{subfigure}[t]{0.3\textwidth}
    \centering
    \includegraphics[width=\linewidth]{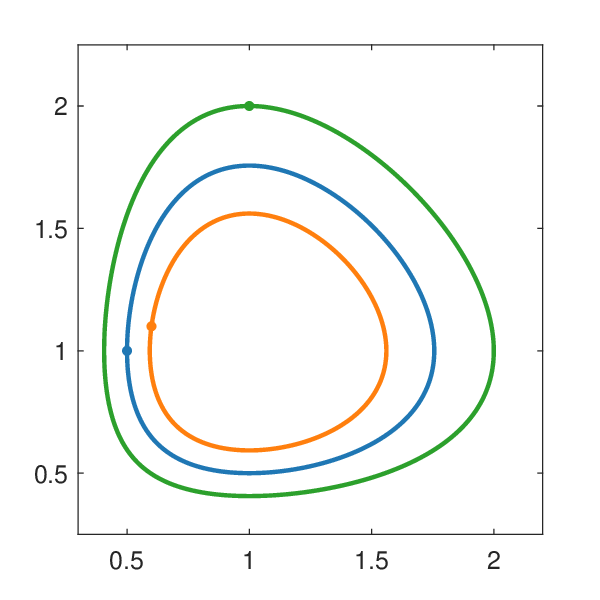}
    % \caption{Three trajectories used for training, Initial points are marked with solid circles.}
\end{subfigure}
\begin{subfigure}[t]{0.3\textwidth}
    \centering
    \includegraphics[width=\linewidth]{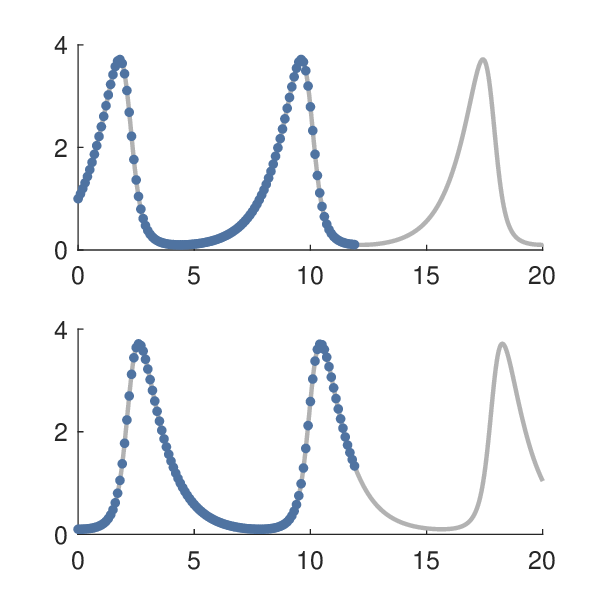}
    \caption*{$h = 0.1$}
\end{subfigure}
\begin{subfigure}[t]{0.3\textwidth}
    \centering
    \includegraphics[width=\linewidth]{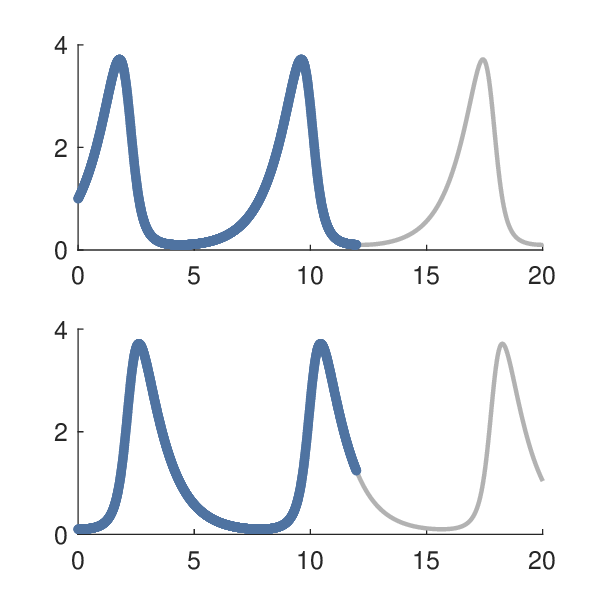}
    \caption*{$h = 0.01$}% \caption{$h = 0.1$; the vertical axis shows elapsed time along the blue trajectory in (a)}
\end{subfigure}
\caption{
Three Lotka–Volterra orbits are displayed. Filled circles mark samples from the blue orbit. In the time-series panels, the horizontal axis is time; the upper panel plots the $x$-information and the lower panel plots the $y$-information of the sampled points.}
    \label{fig:LV1}
\end{figure}

The neural networks are simple multilayer perceptrons with three hidden layers, each with 100 nodes, and use the ELU activation function. We stop the \review{learning} after the mean squared error between the \review{predictions and observations} reaches $10^{-4}$. 
When applied to the system linearized around an orbit, the theory developed in the previous sections can be used to understand the local behavior of the nearby trajectories:
\begin{itemize}
    \item The Implicit Euler may lead to expansive systems, a consequence corresponding to the possibility that $\operatorname{Re}(\hat\lambda)>0$;
    \item The Implicit Trapezoidal Rule and the Implicit Midpoint Rule can better capture the nonlinear conservative nature of the dynamical system; a consequence corresponding to each method's stability region, including the imaginary axis but not the right half plane.
    \item The observed orbits may become a limit cycle for the vector field learned by either the Implicit Trapezoidal Rule or the Implicit Midpoint Rule.
\end{itemize}
See the top row in Figure~\ref{fig:LV_1e_1_results}.
Additionally, the second row shows trajectories originating from various initial data points. Although the Implicit Trapezoidal Rule and the Implicit Midpoint Rule both more effectively capture the conservative property of the given data, the other trajectories, e.g., the orange and green trajectories, do not generate a close orbit. The orbit containing the observation snapshots appears to be a limit cycle of the learned vector field. Observe that both the orange and the green trajectories spiral towards the purple.

Figure~\ref{fig:LV_1e_1_results} presents the learned dynamics using data sampled from $h = 0.1$. Even though the data points "cover" a wider region, the undesirable numerical artifacts from the Implicit Euler method persist.

\begin{figure} %[!htbp]
    \centering
% ── Row 1 : initial data **in** the training set ──────────────────────────
    \begin{subfigure}[t]{\textwidth}
        \centering
        \begin{subfigure}[t]{0.24\textwidth}
            \includegraphics[width=\linewidth]{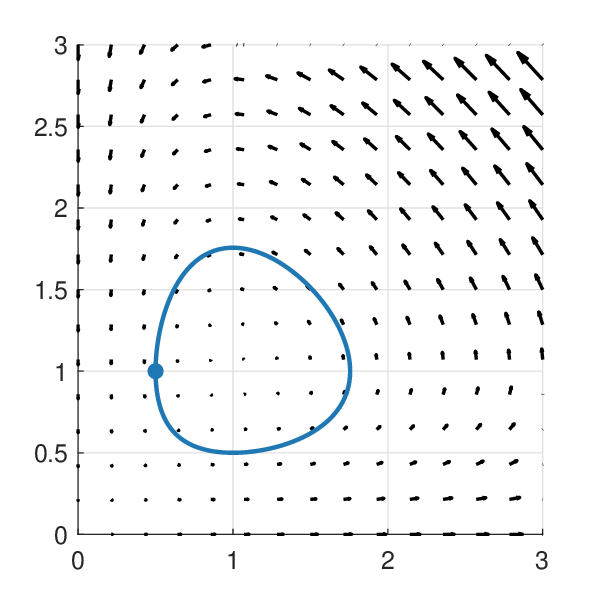}
            \caption{Ground truth}
        \end{subfigure}
        \begin{subfigure}[t]{0.24\textwidth}
            \includegraphics[width=\linewidth]{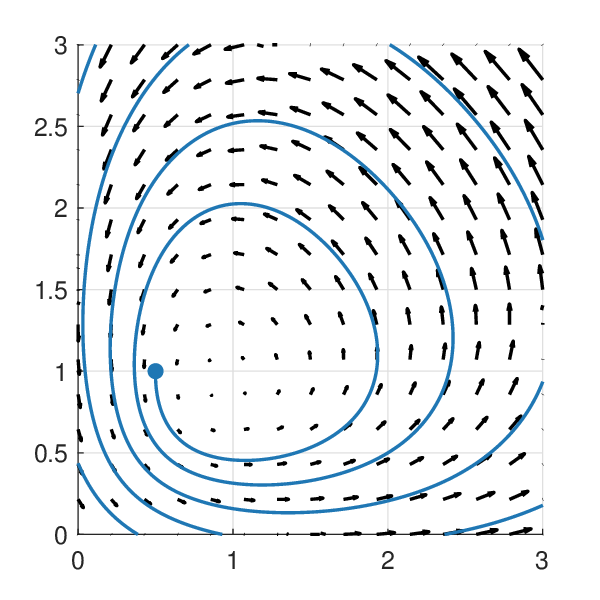}
            \caption{Implicit Euler}
        \end{subfigure}
        \begin{subfigure}[t]{0.24\textwidth}
            \includegraphics[width=\linewidth]{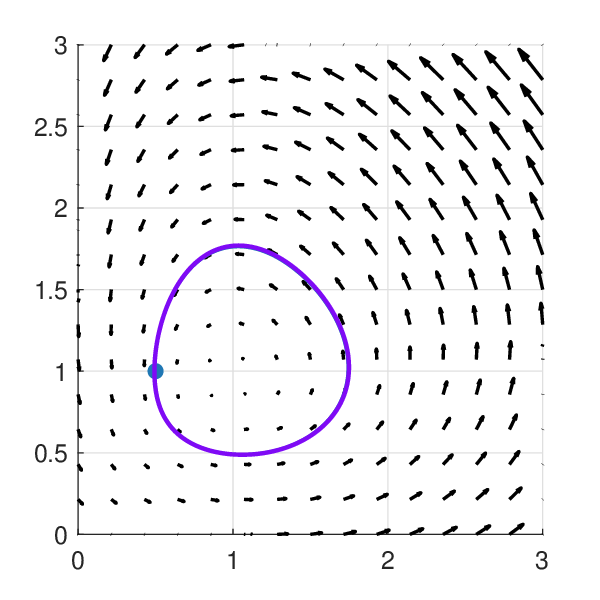}
            \caption{Implicit trapezoidal}
        \end{subfigure}
        \begin{subfigure}[t]{0.24\textwidth}
            \includegraphics[width=\linewidth]{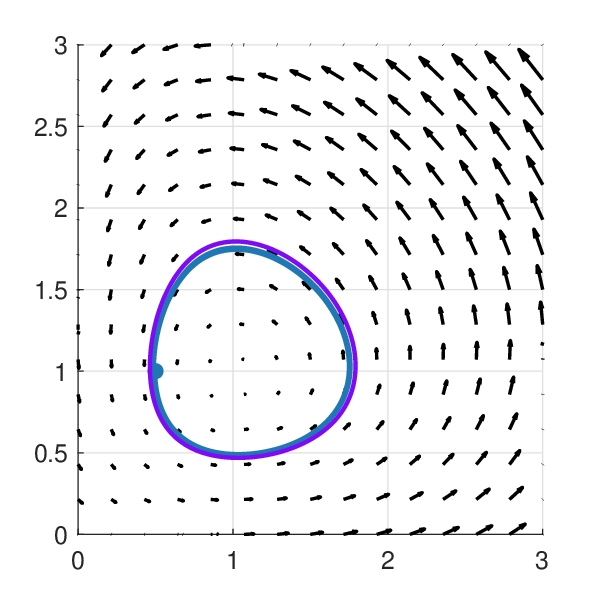}
            \caption{Implicit midpoint}
        \end{subfigure}

    \end{subfigure}

    \vspace{1.5em} % vertical gap between rows

% ── Row 2 : initial data **outside** the training set ─────────────────────
    \begin{subfigure}[t]{\textwidth}
        \centering
        \begin{subfigure}[t]{0.24\textwidth}
            \includegraphics[width=\linewidth]{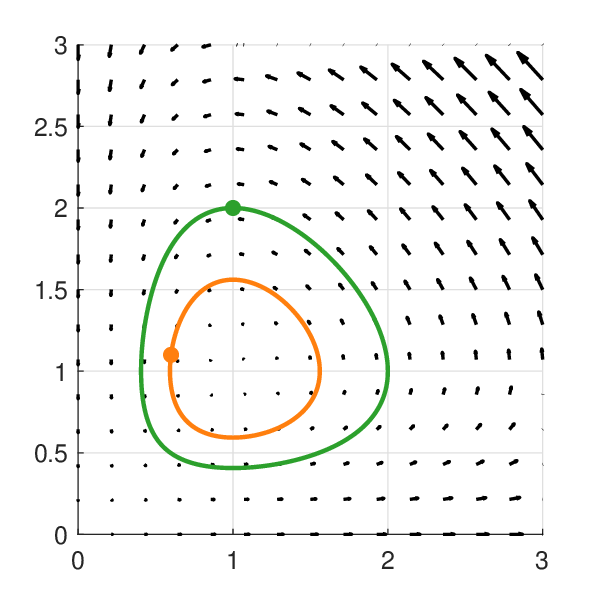}
            \caption{Ground truth}
        \end{subfigure}
        \begin{subfigure}[t]{0.24\textwidth}
            \includegraphics[width=\linewidth]{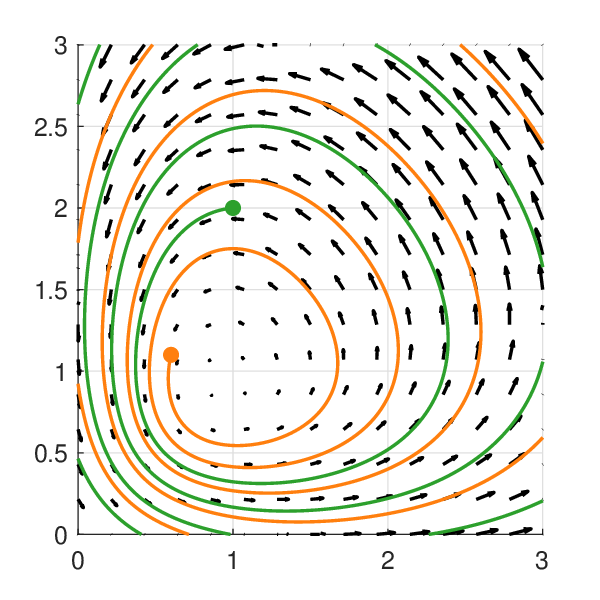}
            \caption{Implicit Euler}
        \end{subfigure}
        \begin{subfigure}[t]{0.24\textwidth}
            \includegraphics[width=\linewidth]{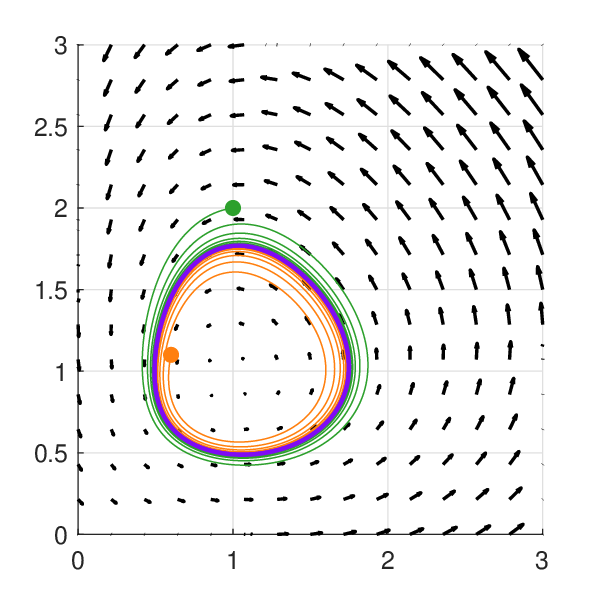}
            \caption{Implicit trapezoidal}
        \end{subfigure}
        \begin{subfigure}[t]{0.24\textwidth}
            \includegraphics[width=\linewidth]{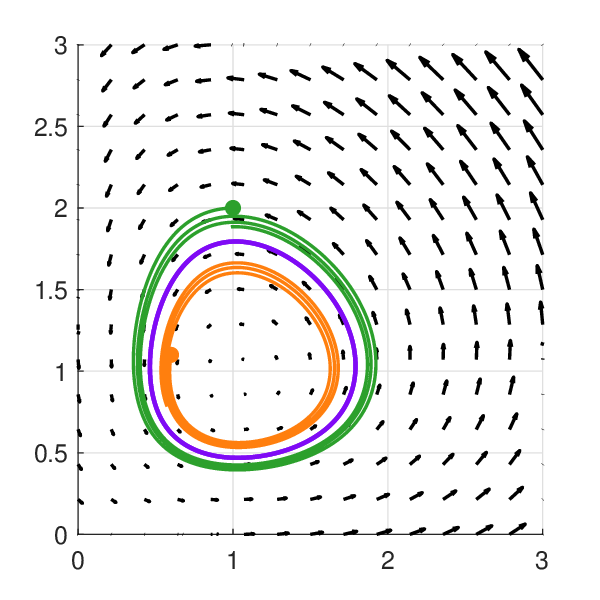}
            \caption{Implicit midpoint}
        \end{subfigure}
    \end{subfigure}

    \caption{Learning results with timestep $h = 0.1$ using the Implicit Euler, implicit trapezoidal, and implicit midpoint methods. The top row presents trajectories generated by the learned dynamics from initial conditions within the given \review{data} set. In contrast, the bottom row shows trajectories from initial conditions outside the \review{data} set.  The purple orbits in (c), (d), (g), and (f) seem to become limiting cycles.}
    \label{fig:LV_1e_1_results}
\end{figure}
Figure~\ref{fig:LV_1e_2_results} shows the learning results using the dataset with $h = 0.01$. Although a finer sampling timestep is used, the learning results using the Implicit Euler method continue to show expansive outcomes.
\begin{figure}
    \centering
% ── Row 1 : initial data **in** the training set ──────────────────────────
    \begin{subfigure}[t]{\textwidth}
        \centering
        \begin{subfigure}[t]{0.24\textwidth}
            \includegraphics[width=\linewidth]{journal_image/LV1/BE/journal_BE_gt.eps}
            \caption{Ground truth}
        \end{subfigure}
        \begin{subfigure}[t]{0.24\textwidth}
            \includegraphics[width=\linewidth]{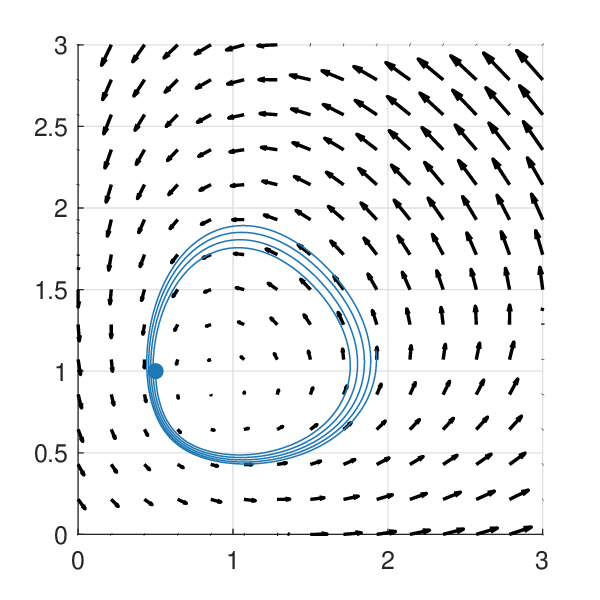}
            \caption{Implicit Euler}
        \end{subfigure}
        \begin{subfigure}[t]{0.24\textwidth}
            \includegraphics[width=\linewidth]{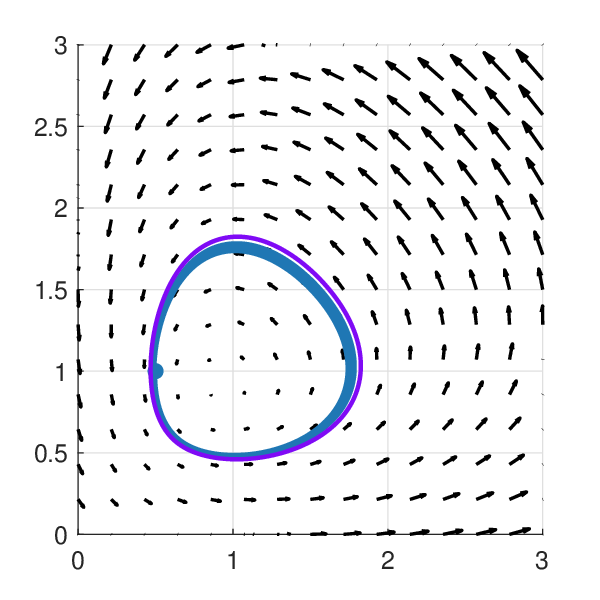}
            \caption{Implicit trapezoidal}
        \end{subfigure}
        \begin{subfigure}[t]{0.24\textwidth}
            \includegraphics[width=\linewidth]{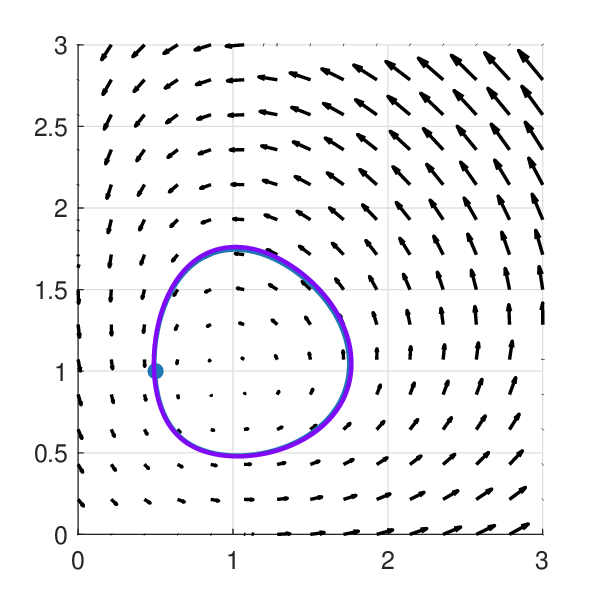}
            \caption{Implicit midpoint}
        \end{subfigure}

    \end{subfigure}

    \vspace{1.5em} % vertical gap between rows
    \begin{subfigure}[t]{\textwidth}
        \centering
        \begin{subfigure}[t]{0.24\textwidth}
            \includegraphics[width=\linewidth]{journal_image/LV1/BE/journal_BE_gt_out.eps}
            \caption{Ground truth}
        \end{subfigure}
        \begin{subfigure}[t]{0.24\textwidth}
            \includegraphics[width=\linewidth]{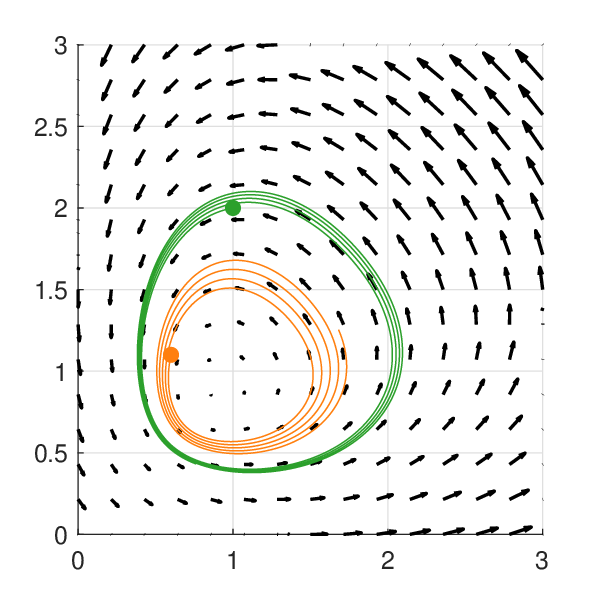}
            \caption{Implicit Euler}
        \end{subfigure}
        \begin{subfigure}[t]{0.24\textwidth}
            \includegraphics[width=\linewidth]{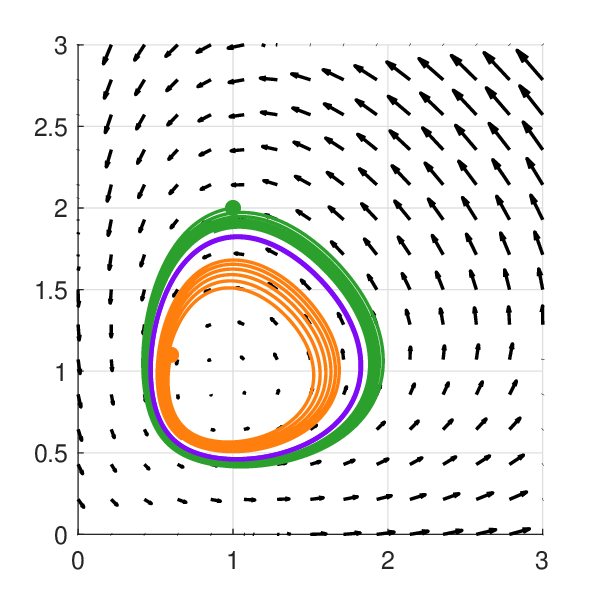}
            \caption{Implicit trapezoidal}
        \end{subfigure}
        \begin{subfigure}[t]{0.24\textwidth}
            \includegraphics[width=\linewidth]{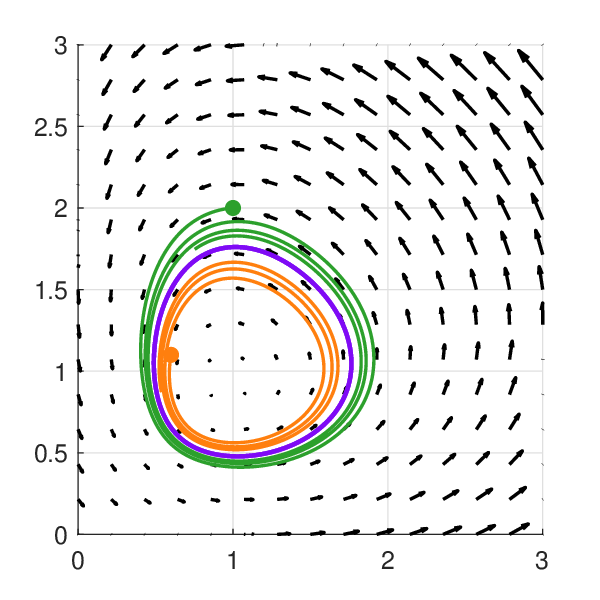}
            \caption{Implicit midpoint}
        \end{subfigure}

    \end{subfigure}

    \caption{Learning results with timestep $h = 0.01$ using the Implicit Euler, implicit trapezoidal, and implicit midpoint methods. The top row presents trajectories generated by the learned dynamics from an initial condition within the \review{given data} set. In contrast, the bottom row shows trajectories from initial conditions outside the \review{given data} set.}
    \label{fig:LV_1e_2_results}
\end{figure}
\subsection{Neural ODE models for a relaxation oscillator}
In this example, we will study the stability properties of learned Neural ODE models constructed with data from a trajectory of the following nonlinear oscillator \review{\cite{Ariel09}} \cite{dahlquist1982numerical}: 
\begin{equation*}
    \left\{
    \begin{array}{rcl}
        \dfrac{dx}{dt} &=& -1-x+8y^3,\\
        \\
        \dfrac{dy}{dt} &=& \dfrac{1}{\nu}(-x+y-y^3).\\
    \end{array}
    \right.
\end{equation*}
For $0<\nu\ll1$, the trajectories of this system will converge to a limiting cycle defined by $x=y-y^3;$ see Figure~\ref{fig:LV_dynamics}. 

We show below that our linear theories can explain why a learned system may not follow the structure (limiting cycle) in the data.

\review{We learned $g_h$, which is parameterized by a neural network}, to approximate the vector field, $f(x,y)$, so that a 
flow line of $g_h(x,y)$ fits the observed trajectory well.
The network architecture consists of three hidden layers, each containing 100 nodes, and uses the \texttt{ELU()} activation function. We terminate the \review{learn}ing when the loss function falls below $10^{-4}$.

We set $\nu = 0.1$ in this study, and generate trajectories by integrating the dynamic system using MATLAB's \texttt{ODE45} solver and setting \texttt{AbsTol} and \texttt{RelTol} to $10^{-12}$. We sample trajectory snapshots at $t_n=nh$ with $h = 10^{-1}, 10^{-2}$, and $0\le t_n\le 5.$
As in other parts of this paper, the subscript $h$ in $g_h$ reflects the step size used during \review{learn}ing by the Explicit Euler method to integrate the vector field $g_h(x,y)$.

 For convenience, we denote the solution of this system by $\gamma(t; \gamma_0) = ( x(t), y(t) )$ with $\gamma_0 := (x(0), y(0) ).$ We use $f(\gamma) \equiv f(x,y)$ to denote the right-hand side of the ODE, and $g_h(x,y)$ to denote a neural network \review{that fits} the given trajectory data.
 Denote $\lambda := \lambda(\gamma)$ and $\hat\lambda := \hat\lambda(\gamma)$ an eigenvalue of the Jacobian, $J_f(\gamma)$, and $J_{g_h}(\gamma)$ at $\gamma$ respectievely.

From this example, our studies show that
\begin{itemize}
    \item As predicted by our theory, the Jacobian eigenvalues of the learned dynamics fall within the stability region of the Explicit Euler method, see Figure~\ref{fig:LV_stab};
    \item Furthermore, due to the presence of non-trivial imaginary part in $\lambda(\gamma),$ and the shape of Explicit Euler's stability region, the learned dynamics could have a larger dissipation; i.e. $\operatorname{Re}(\hat\lambda)< \operatorname{Re}(\lambda)$;  Figure~\ref{fig:LV_eig}; 
    \item Consequently, the trajectories from $g_h$ may not converge to a limiting cycle close to the ground truth. 
    See Figure~\ref{fig:one_traj_vector_field}.
\end{itemize}

\begin{figure}[htbp]
    \centering
    
    \begin{subfigure}[t]{0.3\textwidth}
        \centering
        \includegraphics[width=\linewidth]{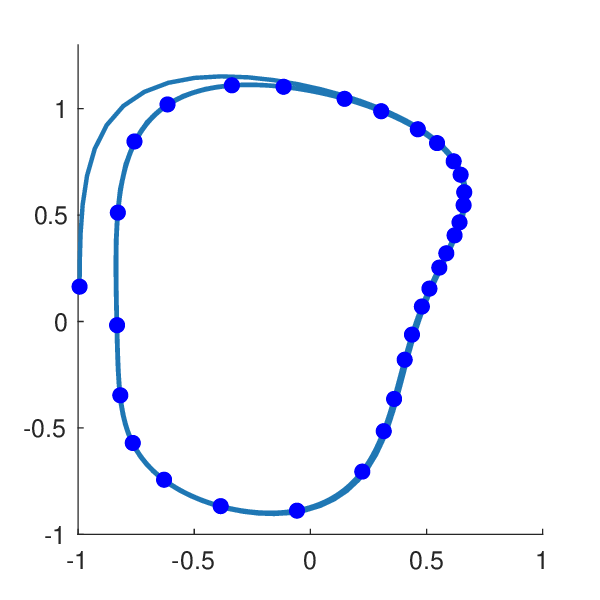}
        \caption{Trajectory with starting point near the limit cycle}
    \end{subfigure}~
    \begin{subfigure}[t]{0.3\textwidth}
        \centering
        \includegraphics[width=\linewidth]{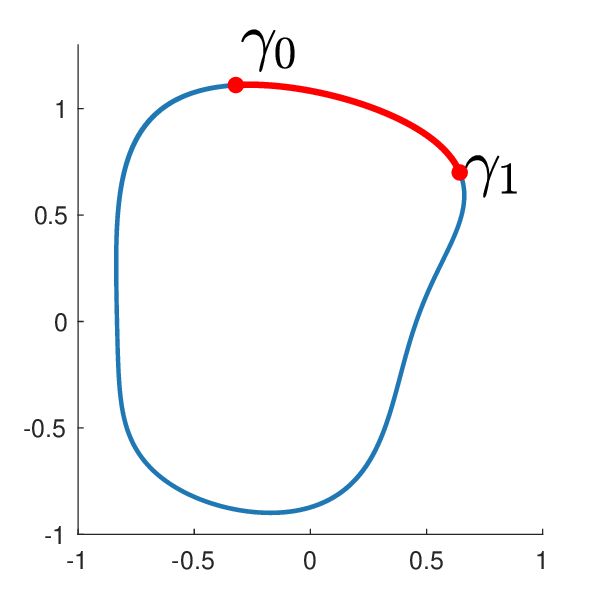}
        \caption{Selected segment of the limit cycle for detailed analysis}
        \label{fig:LV_branch}
    \end{subfigure}

    \caption{The left panel shows the trajectory where data is sampled, which also encompasses a limit cycle. The right panel highlights the segment chosen for detailed analysis.}
    \label{fig:LV_dynamics}
\end{figure}

\begin{figure} %[htbp]
    \centering
    %── left panel ────────────────────────────────────────
    \begin{minipage}[t]{0.42\linewidth}
        \centering
        \textbf{$h = 0.1$}\\[0.4em]   % ← label on top
        \includegraphics[width=\linewidth]{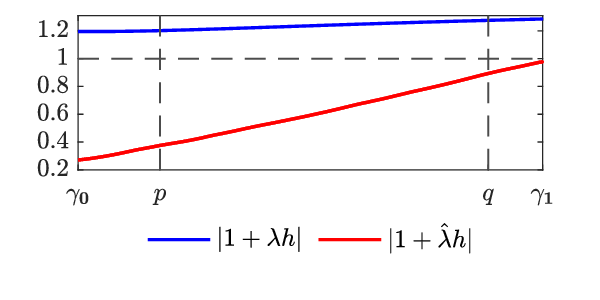}
    \end{minipage}
    \hspace{1cm}
    %── right panel ───────────────────────────────────────
    \begin{minipage}[t]{0.42\linewidth}
        \centering
        \textbf{$h = 0.01$}\\[0.4em]  % ← label on top
        \includegraphics[width=\linewidth]{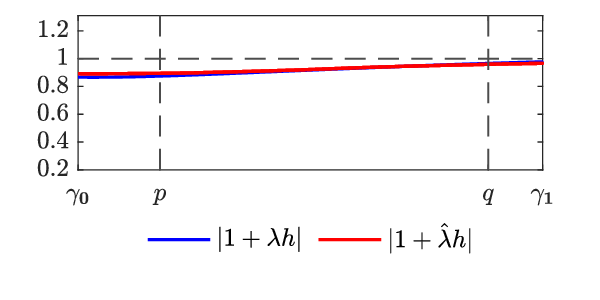}
    \end{minipage}

    \caption{Profiles of $|1+\lambda(x) h|$ and $|1+\hat{\lambda}(x) h|$ along the upper branch of the limit cycle. }
    \label{fig:LV_stab}
\end{figure}

\begin{figure} %[htbp]
    \centering
    
    \begin{subfigure}[t]{0.3\textwidth}
        \centering
        \includegraphics[width=\linewidth]{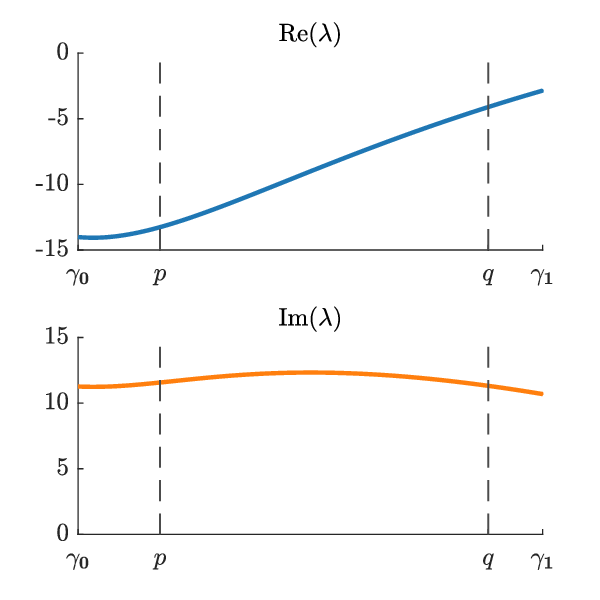}
        \caption{True dynamic}
        \label{fig:true}
    \end{subfigure}
    \begin{subfigure}[t]{0.3\textwidth}
        \centering
        \includegraphics[width=\linewidth]{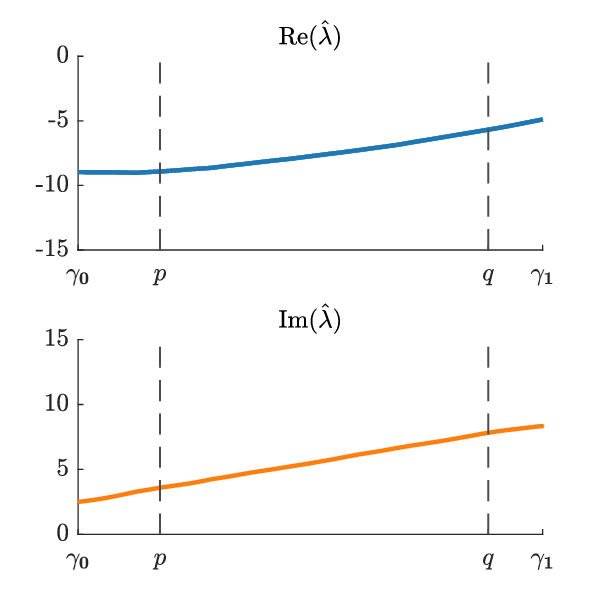}
        \caption{$h = 0.1$}
        \label{fig:eig1e_1}
    \end{subfigure}
    \begin{subfigure}[t]{0.3\textwidth}
        \centering
        \includegraphics[width=\linewidth]{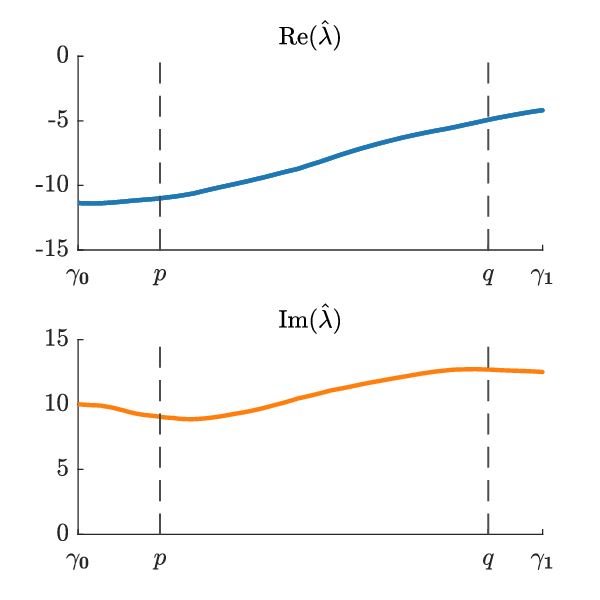} 
        \caption{$h = 0.01$}
        \label{fig:eig1e_2}
    \end{subfigure}
    \caption{Real and imaginary components of the Jacobian eigenvalues for both the learned and true dynamics along the sampled branch; the x-axis ticks match the labels in Figure \ref{fig:LV_dynamics}.}
    \label{fig:LV_eig}
\end{figure}
\begin{figure} %[htbp]
    \centering
    %── left panel ────────────────────────────────────────
    \begin{minipage}[t]{0.3\linewidth}
        \centering
        % \textbf{$h = 0.1$}\\[0.4em]   % ← label on top
        \includegraphics[width=\linewidth]{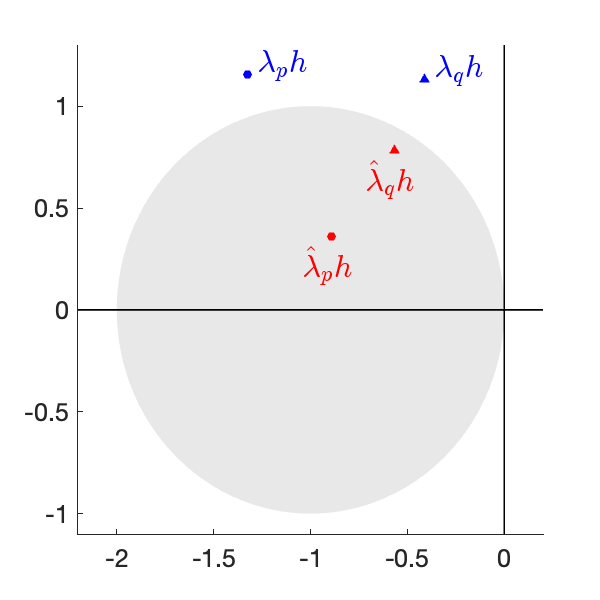}
        \subcaption{$h = 0.1$}
    \end{minipage}
    \hspace{1cm}
    %── right panel ───────────────────────────────────────
    \begin{minipage}[t]{0.3\linewidth}
        \centering
        % ← label on top
        \includegraphics[width=\linewidth]{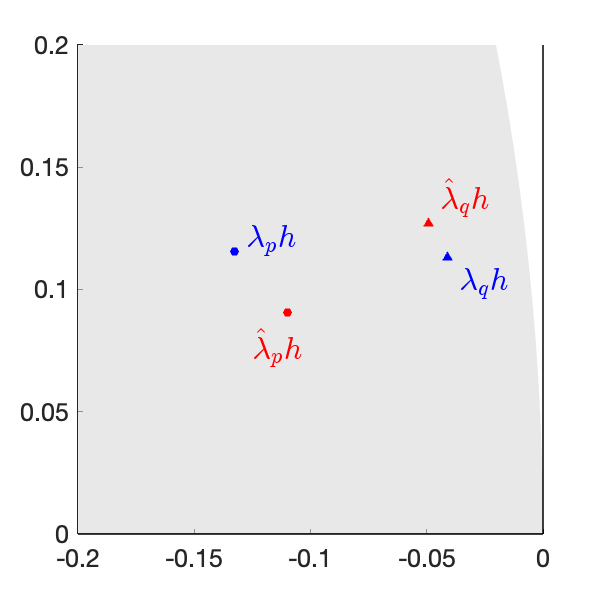}
        \subcaption{$h = 0.01$}
    \end{minipage}

    \caption{Comparison of the eigenvalues $\lambda$ and $\hat\lambda$ of $J_f$ and $J_{g_h}$ at points $p$ and $q$. The eigenvalues are scaled by the step size used by the numerical integrator and plotted on the complex plane, with the gray regions indicating the integrator's stability region.
    }
    \label{fig:LV_FE_stab}
\end{figure}

\begin{figure}%[htbp]
    \centering
    % ---------- first row -------------------------------------------------
 %   \makebox[\textwidth][c]{%
        \begin{subfigure}[t]{0.30\textwidth}
            \centering
            \includegraphics[width=\linewidth]{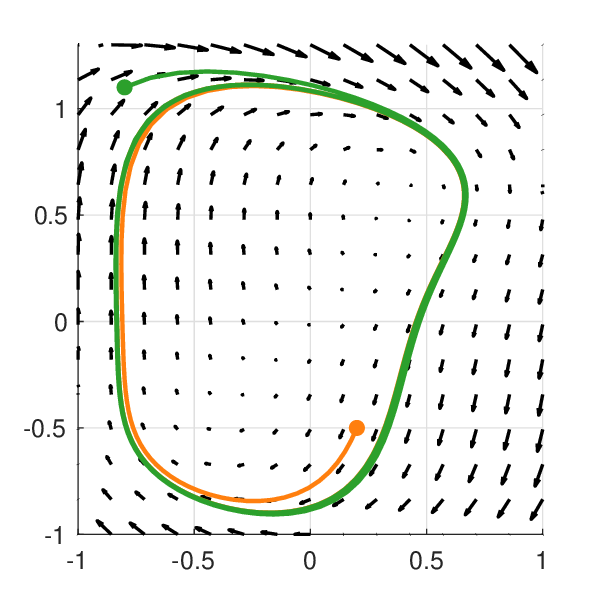}
            \caption{Ground truth}
            \label{fig:true_a}
        \end{subfigure}\hfill
        \begin{subfigure}[t]{0.30\textwidth}
            \centering
            \includegraphics[width=\linewidth]{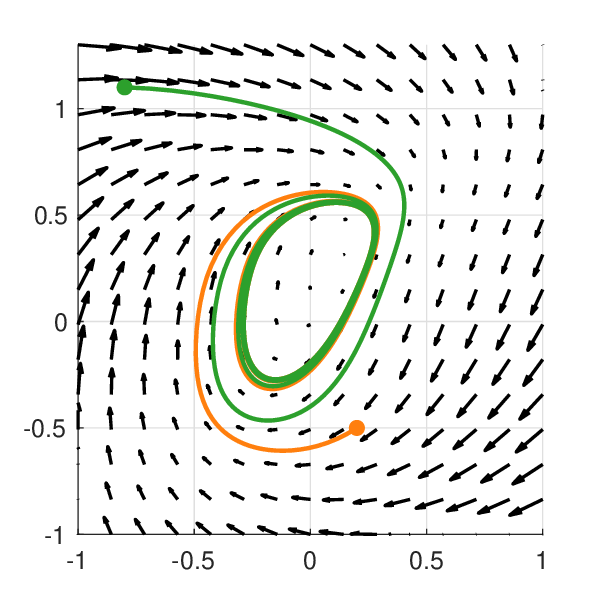}
            \caption{A learned vector field ($h=0.1$)}
            \label{fig:learned_a1}
        \end{subfigure}\hfill
        \begin{subfigure}[t]{0.30\textwidth}
            \centering
            \includegraphics[width=\linewidth]{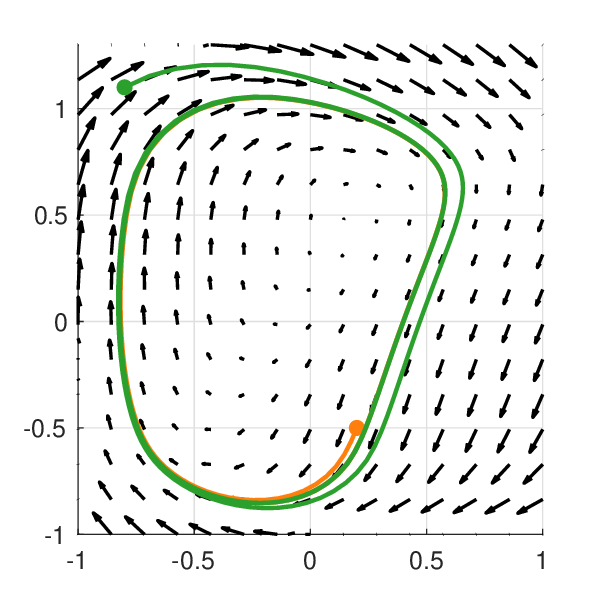}
            \caption{A learned vector field ($h=0.01$)}
            \label{fig:learned_b1}
        \end{subfigure}\hfill        
    \caption{Vector fields using only one trajectory with different sampling timesteps. }
    \label{fig:one_traj_vector_field}
\end{figure}

\subsection{Neural ODE models for a damped nonlinear pendulum}

We return to the study of the nonlinear pendulum mentioned in the Introduction. Even though the sampling rate is significantly higher than the Nyquist rate, we shall observe the striking differences in the \emph{amplitudes and phases of oscillations} in the fully resolved trajectories of the learned neural ODE models.

We use the damped pendulum governed by
\begin{equation*}
    \left\{
    \begin{array}{rcl}
        \dfrac{d\vartheta}{dt} &=& \omega,\\
        \\
        \dfrac{d\omega}{dt} &=& -\gamma \omega -\dfrac{g}{L}\sin \vartheta,\\
    \end{array}
    \right.
\end{equation*}
where we set $\gamma = 0.02$, $\dfrac{g}{L} = 9.81$. We simulate the \review{trajectories} of $\vartheta$ and $\omega$ using MATLAB's \texttt{ODE45} solver and setting \texttt{AbsTol} and \texttt{RelTol} to $10^{-12}$ for $t \in [0, 40].$ Snapshots are taken at $t_n=nh$ with $h = 10^{-1}$, and $0\le t_n\le 10.$

\review{The neural network consists of three hidden layers with 100 neurons each,
equipped with the ELU activation function.}

In Figure~\ref{fig:combined-2x2}, we show the numerical artifacts of four numerical integrators used to learn the damped nonlinear pendulum system. We observe that 
The Implicit Euler method infers a system that supports oscillations of much larger amplitudes; in contrast, the Explicit Euler method infers a system with much larger damping. According to our theory, such discrepancies originate from the differences in the geometries of the stability regions of the integrators. 
We see that the Implicit Midpoint rule yields the most accurate learning, both in terms of amplitude and phase accuracy\revmark{, consistent with its exact preservation of oscillatory modes shown in Section~\ref{sec:nonlinear-learning}}.

\begin{figure}[!htbp]
  \centering
  % Row 1
  \begin{subfigure}[t]{0.48\textwidth}
    \centering
    \includegraphics[width=\linewidth]{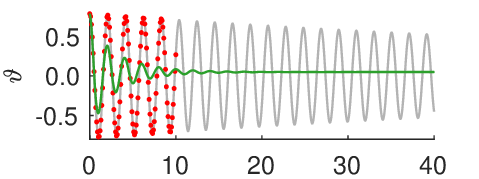}
    \caption{Explicit Euler}
  \end{subfigure}\hfill
  \begin{subfigure}[t]{0.48\textwidth}
    \centering
    \includegraphics[width=\linewidth]{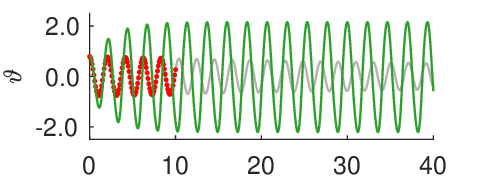}
    \caption{Implicit Euler}
  \end{subfigure}

  \vspace{1em} % small vertical gap

  % Row 2
  \begin{subfigure}[t]{0.48\textwidth}
    \centering
    \includegraphics[width=\linewidth]{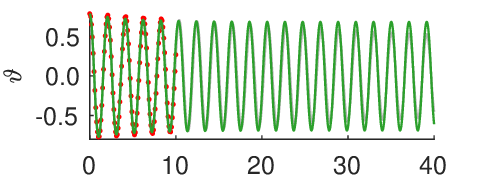}
    \caption{Implicit Midpoint}
  \end{subfigure}\hfill
  \begin{subfigure}[t]{0.48\textwidth}
    \centering
    \includegraphics[width=\linewidth]{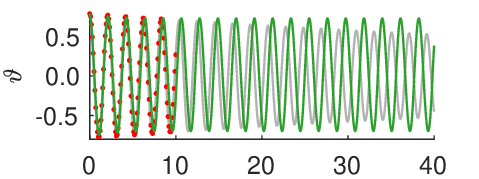}
    \caption{AB2}
  \end{subfigure}

  \caption{Amplitude errors of the learned solution component $\vartheta$ (green curves) inferred from applying four different numerical integrators to fit the data (red dots): (a) Explicit Euler, (b) Implicit Euler, (c) Implicit Midpoint, and (d) AB2. For reference, the true solutions are shown by the gray curves.}
  \label{fig:combined-2x2}
\end{figure}

\section{Conclusion}

This article studies the potential non-negligible Artifacts of Numerical Integration in Learning Dynamical Systems from observed trajectories. 
 The choice of numerical integrator and its step size are critical. 
 Interestingly, the stability of a numerical integrator does not even guarantee the stability of the learned dynamical system. What matters is the geometry of the integrator's stability region on the complex plane.
 The superior stability of the Implicit Euler method actually makes it the worst choice for forward computation in the type of inverse problems we considered in this paper. We observe that in one-step methods, only the implicit midpoint rule accurately preserves both the amplitude and rotation direction, regardless of whether the actual dynamics are conservative or dissipative. 
 We also presented numerical studies demonstrating that our analysis
 can provide a way to analyze the learning problems for more complex nonlinear systems.
 \revmark{Beyond the choice of integrator, the analysis also quantifies how sample size and observational noise propagate into the learned parameter, and identifies a root-selection principle that resolves a non-uniqueness previously reported for linear multistep methods.}
 \review{Overall, our analysis suggests that the implicit midpoint method, whose stability region coincides with the left-hand side of the complex plane, naturally favors the preservation of conservative and dissipative structure from discrete data.}  \review{Therefore, in the absence of prior knowledge about whether the underlying autonomous system is conservative or dissipative, we advocate for the use of the Implicit Midpoint Rule as a principled default choice.}
\section*{Statements and Declarations}

\begin{itemize}
\item \textbf{Funding} Richard Tsai is supported partially by National Science Foundation grants DMS-2208504, DMS-2513857, and 
Army Research Oﬃce Grant W911NF2320240.
Bing-Ze Lu received support from the National Science and Technology Council, Taiwan, through Grant 113-2917-I-564-033 and 114-2115-M-194-007-MY3.
\item \textbf{Conflict of interest/Competing interests (check journal-specific guidelines for which heading to use) } The authors have no competing interests to declare that are relevant to the content of this article.
\item \textbf{Ethics approval and consent to participate} Not applicable
\item \textbf{Consent for publication }Not applicable
\item \textbf{Data availability}  Not applicable
\item \textbf{Materials availability} Not applicable
\item \textbf{Code availability} Not applicable
\item \textbf{Author contribution} Not applicable
\end{itemize}

\section{Appendix}
\newcolumntype{L}{>{\raggedright\arraybackslash}X}
% ---------- Table 1 ----------
\begin{table}[h]
    \caption{Notation: dynamical systems and solutions.}
    \label{tab:notation-dyn}
    \centering
    \renewcommand{\arraystretch}{1.25}
    \begin{tabularx}{\textwidth}{@{} r L @{}}
        \toprule
        \textbf{Symbol} & \textbf{Description} \\
        \midrule
        $f : \mathbb{R}^d \to \mathbb{R}^d$ & True (unknown) vector field of the autonomous system $y' = f(y)$ \\
        $\hat{f} \in \mathcal{F}$          & Candidate vector field drawn from function class $\mathcal{F}$ \\
        $g$                                 & Global minimizer of the discrete learning problem; in general $g \neq f$ \\
        $\varphi_t[f]\,y_0$                 & Flow map of $y' = f(y)$: solution at time $t$ starting from $y_0$ \\
        $y(t;\, y_0)$                       & Solution trajectory, i.e.\ $y(t;\,y_0) \equiv \varphi_t[f]\,y_0$ \\
        $\lambda \in \mathbb{C}$            & Eigenvalue of $f$; governs the scalar model $z' = \lambda z$ \\
        $\hat{\lambda} \in \mathbb{C}$      & Learned eigenvalue, i.e.\ the minimizer of the scalar least-squares problem \\
        \bottomrule
    \end{tabularx}
\end{table}

% ---------- Table 2 ----------
\begin{table}[h]
    \caption{Notation: data and sampling.}
    \label{tab:notation-data}
    \centering
    \renewcommand{\arraystretch}{1.25}
    \begin{tabularx}{\textwidth}{@{} r L @{}}
        \toprule
        \textbf{Symbol} & \textbf{Description} \\
        \midrule
        $H$                                          & Sampling step size (time gap between consecutive observations) \\
        $t_n = nH$                                   & $n$-th observation time \\
        $\mathcal{D}_N = \{(t_n, y(t_n))\}_{n=1}^N$  & Dataset of $N$ trajectory snapshots \\
        $U_M$                                        & Set of $M$ sampled initial conditions \\
        $Z_n = e^{\lambda n h} Z_0$                  & Exact solution of the scalar model at step $n$ \\
        $z_n$                                        & Numerical solution at step $n$ produced by the chosen integrator \\
        \bottomrule
    \end{tabularx}
\end{table}

% ---------- Table 3 ----------
\begin{table}[h]
    \caption{Notation: numerical integrators.}
    \label{tab:notation-integrators}
    \centering
    \renewcommand{\arraystretch}{1.25}
    \begin{tabularx}{\textwidth}{@{} r L @{}}
        \toprule
        \textbf{Symbol} & \textbf{Description} \\
        \midrule
        $h$                                          & Step size used by the numerical integrator \\
        $m$                                          & Number of integrator steps per observation interval ($mh = H$) \\
        $S_h[\hat{f}]$                               & One-step propagator of the integrator applied to $\hat{f}$ with step size $h$ \\
        $S_h^{mn}[\hat{f}]$                          & $mn$-fold composition of $S_h[\hat{f}]$, approximating $\varphi_{mnh}[\hat{f}]$ \\
        $p(\xi)$                                     & Stability function of a one-step method (written $R(\xi)$ in Section~\ref{sec:nonlinear-learning}) \\
        $\xi = \lambda h$                            & Scaled eigenvalue; the fundamental argument of $p(\xi)$ \\
        $\rho(z),\, \kappa(z)$                       & First / second characteristic polynomials of a linear multistep method (LMM) \\
        $\pi_{\lambda h}(z) = \rho(z) - \lambda h\,\kappa(z)$ & Combined characteristic polynomial of an LMM \\
        $\zeta_j$                                    & $j$-th root of $\pi_{\lambda h}(z)$ \\
        \bottomrule
    \end{tabularx}
\end{table}

% ---------- Table 4 ----------
\begin{table}[h]
    \caption{Notation: stability regions and optimization.}
    \label{tab:notation-stability}
    \centering
    \renewcommand{\arraystretch}{1.25}
    \begin{tabularx}{\textwidth}{@{} r L @{}}
        \toprule
        \textbf{Symbol} & \textbf{Description} \\
        \midrule
        $\mathcal{R}_A$          & Region of absolute stability: all $\xi$ for which every root of $\pi_\xi$ lies in the unit disc \\
        $\mathcal{R}_P$          & Region of partial stability: $\xi$ for which roots lie both inside and outside the unit disc \\
        $\mathcal{F}$            & Function class from which the candidate vector field $\hat{f}$ is drawn \\
        $\hat{\xi} = \hat{\lambda}\, h$ & Global minimizer of the scalar least-squares problem \\
        $c_1, \ldots, c_k$       & Coefficients of the general solution of the linear multistep recurrence \\
        \bottomrule
    \end{tabularx}
\end{table}
\subsection{Proofs of Theorems in Section~2}

\subsubsection{Conservative Cases}

Lemma~\ref{thm:main} and Theorem~\ref{thm:euler_stab} indicate that when the underlying dynamics are conservative, the dynamics learned by one-step methods tend to lie on the boundary of the corresponding method’s stability region. The amplitude behavior of the learned dynamics is dissipative for the Explicit Euler and RK2 methods, but expansive for the Implicit Euler method, since the boundaries of their stability regions are contained entirely on either the right-hand or the left-hand side of the complex plane. Therefore, we only need to show that when the step size $h$ is sufficiently small, the dynamics inferred by the RK3 method exhibit expansive behavior. Beyond the amplitude error, we rigorously prove that, once the sampling rate satisfies the Nyquist sampling condition, the learned dynamics correctly preserve the rotation direction. 

\paragraph{Proof of Corollary~\ref{cor:unified_cons} (RK3 Amplitude Error).}
We spot the learned dynamics using the RK3 method that satisfies
\begin{eqnarray}\label{eq:RK3_proof}
    1+\xi+\dfrac{\xi^2}{2}+\dfrac{\xi^3}{6}=e^{\lambda h}, \qquad |\xi|<1.
\end{eqnarray}
If $h>0$ satisfies $6|1-e^{\lambda h}|<1$, then this equation admits a unique solution $\hat{\lambda}h$ within the unit disk.  
Indeed, the curve $|1+\xi+\tfrac{\xi^2}{2}+\tfrac{\xi^3}{6}|=1$ intersects the imaginary axis only at $(0,0)$ and $(0,\pm\sqrt{3})$, indicating that any root with $|\xi|<1$ must have $\operatorname{Re}(\xi)\ge0$. Moreover, if $h$ satisfies $6|1-e^{\lambda h}|<1$, then by Vieta’s relations,
\[
z_1+z_2+z_3=-3,\qquad 
z_1z_2+z_2z_3+z_1z_3=6,\qquad 
z_1z_2z_3=-6(1-e^{\lambda h}),
\]
where $z_1, z_2, z_3$ are the roots of \eqref{eq:RK3_proof}. We find that there must exist one root $|z_1|>1$ from the first relation. By contradiction, if both $|z_2|<1$ and $|z_3|<1$, the stability region would not be contained in $[-3, 1] \times [-3i, 3i]$. Hence, we can find exactly one root inside $|\xi|<1$, which we denote $\hat{\lambda}h$.

\paragraph{Proof of Corollary~\ref{cor:unified_cons} (Rotation Direction).}
Assume the underlying dynamics are conservative or dissipative, and let $h>0$ satisfy the Nyquist condition $\operatorname{Im}(\lambda h)\in(-\pi,\pi)$. 
We show that for each one-step method considered, the learned dynamics $\hat{\lambda}h$ inferred from $p(\hat{\lambda}h)=e^{\lambda h}$ preserve the rotation direction, i.e.,
\[
\operatorname{sign}(\operatorname{Im}(\hat{\lambda}h)) = \operatorname{sign}(\operatorname{Im}(\lambda h)).
\]

\paragraph{Proof of Corollary~\ref{cor:unified_cons} (Explicit Euler).}
From Lemma~\ref{thm:main}, $p(\hat{\lambda} h)=1+\hat{\lambda} h = e^{\lambda h}$, hence
\[
1+\operatorname{Re}(\hat{\lambda} h)=\cos(\operatorname{Im}(\lambda h)),\qquad
\operatorname{Im}(\hat{\lambda} h)=\sin(\operatorname{Im}(\lambda h)),
\]
implying
\[
\operatorname{sign}(\operatorname{Im}(\hat{\lambda} h))=\operatorname{sign}(\operatorname{Im}(\lambda h)).
\]

\paragraph{Proof of Corollary~\ref{cor:unified_cons} (RK2).}
For the second-order Runge--Kutta method, $p(\hat{\lambda} h)=1+\hat{\lambda} h+\tfrac{(\hat{\lambda} h)^2}{2}=e^{\lambda h}$.
Comparing imaginary parts gives
\[
\operatorname{Im}(\hat{\lambda} h)\bigl(1+\operatorname{Re}(\hat{\lambda} h)\bigr)=\sin(\operatorname{Im}(\lambda h)).
\]
Since $\operatorname{Re}(\hat{\lambda} h)>-1$, we obtain
\[
\operatorname{sign}(\operatorname{Im}(\hat{\lambda} h))=\operatorname{sign}(\operatorname{Im}(\lambda h)).
\]

\paragraph{Proof of Corollary~\ref{cor:unified_cons} (RK3).}
For the third-order Runge--Kutta method, 
\[
p(\hat{\lambda} h)=1+\hat{\lambda} h+\tfrac{(\hat{\lambda} h)^2}{2}+\tfrac{(\hat{\lambda} h)^3}{6}=e^{\lambda h}.
\]
Taking imaginary parts yields
\[
\sin(\operatorname{Im}(\lambda h))=\operatorname{Im}(\hat{\lambda} h)\left(1+\operatorname{Re}(\hat{\lambda} h)
+\tfrac{(\operatorname{Re}(\hat{\lambda} h))^2}{2}-\tfrac{(\operatorname{Im}(\hat{\lambda} h))^2}{6}\right).
\]
For $|\hat{\lambda}h|<1$, we have $(\operatorname{Im}(\hat{\lambda}h))^2\in[0,3]$, ensuring positivity of the multiplicative factor, and thus
\[
\operatorname{sign}(\operatorname{Im}(\hat{\lambda} h))=\operatorname{sign}(\operatorname{Im}(\lambda h)).
\]

\paragraph{Proof of Corollary~\ref{cor:unified_cons} (Implicit Euler).}
For the Implicit Euler method, $p(\hat{\lambda} h)=\frac{1}{1-\hat{\lambda} h}=e^{\lambda h}$, giving
\[
\hat{\lambda} h=1-e^{-\lambda h}=1-\cos(\operatorname{Im}(\lambda h))+i\sin(\operatorname{Im}(\lambda h)),
\]
so that
\[
\operatorname{sign}(\operatorname{Im}(\hat{\lambda} h))=\operatorname{sign}(\operatorname{Im}(\lambda h)).
\]

\paragraph{Proof of Corollary~\ref{cor:unified_cons} (Implicit Midpoint Rule).}
For the implicit midpoint rule, 
\[
p(\hat{\lambda} h)=\frac{1+\tfrac{1}{2}\hat{\lambda} h}{1-\tfrac{1}{2}\hat{\lambda} h}=e^{\lambda h},
\]
hence
\[
\hat{\lambda}h=2i\tan\!\left(\frac{\operatorname{Im}(\lambda h)}{2}\right),
\]
and 
\[
\operatorname{sign}(\operatorname{Im}(\hat{\lambda} h))=\operatorname{sign}(\operatorname{Im}(\lambda h)), \quad \text{for } \operatorname{Im}(\lambda h)\in(-\pi,\pi).
\]

\medskip
Therefore, for all considered one-step methods under the Nyquist condition, the learned dynamics $\hat{\lambda}$ preserve the same rotation direction as the true conservative dynamics.

\subsubsection{Dissipative Cases}

\paragraph{Proof of Corollary~\ref{cor:unified_cons} (Explicit Euler).}
From \eqref{eq:one-step-condition}, we have
\[
1 > |e^{\lambda h}| = |1 + \hat{\lambda} h|.
\]
This lies entirely within the Explicit Euler stability region on the left-hand plane, implying $\operatorname{Re}(\hat{\lambda}) < 0$.  
Taking imaginary parts of $e^{\lambda h} = 1 + \hat{\lambda} h$ gives
\[
\sin(\operatorname{Im}(\lambda h)) = \operatorname{Im}(\hat{\lambda} h),
\]
and hence $\operatorname{sign}(\operatorname{Im}(\hat{\lambda})) = \operatorname{sign}(\operatorname{Im}(\lambda))$.

\paragraph{Proof of Corollary~\ref{cor:unified_cons} (RK3).}
We claim that if $|e^{\lambda h}| \le \sqrt{\tfrac{8}{9}}$, then all roots of 
\[
1 + z + \frac{z^2}{2!} + \frac{z^3}{3!} = e^{\lambda h}
\]
have negative real parts.  
Following the tangency and Vieta’s relations arguments in the text, we conclude that if $|e^{\lambda h}|<\tfrac{8}{9}$, all roots satisfy $\operatorname{Re}(z)<0$. Exactly one bounded root $|\xi|<1$ exists, corresponding to $\hat{\lambda}h$.

\paragraph{Proof of Corollary~\ref{cor:unified_cons} (Implicit Euler).}
For Implicit Euler,
\[
e^{\lambda h} = \frac{1}{1 - \hat{\lambda} h}.
\]
Since $1 > |e^{\lambda h}| = \big|\frac{1}{1 - \hat{\lambda} h}\big|$, we have $|1 - \hat{\lambda} h| > 1$.  
The real part of $\hat{\lambda}$, however, cannot be uniquely determined since $\{z: |1-z|=c\}$ crosses both half-planes.

\paragraph{Proof of Corollary~\ref{cor:unified_cons} (Implicit Midpoint Rule).}
From
\[
e^{\lambda h} = \frac{1 + \tfrac{1}{2}\hat{\lambda} h}{1 - \tfrac{1}{2}\hat{\lambda} h},
\]
we obtain
\[
\left|1 + \tfrac{1}{2}\hat{\lambda} h\right| < \left|1 - \tfrac{1}{2}\hat{\lambda} h\right|,
\]
which implies $\operatorname{Re}(\hat{\lambda} h) < 0$.  
The imaginary part analysis follows the same reasoning as in the conservative case.
\bibliographystyle{unsrt}  
\bibliography{reference}

\begin{thebibliography}{10}

\bibitem{Djeumou22}
Franck Djeumou, Cyrus Neary, Eric Goubault, Sylvie Putot, and Ufuk Topcu.
\newblock Neural networks with physics-informed architectures and constraints
  for dynamical systems modeling.
\newblock In {\em Learning for Dynamics and Control Conference}, pages
  263--277. PMLR, 2022.

\bibitem{Djeumou23}
Franck Djeumou, Cyrus Neary, and Ufuk Topcu.
\newblock How to learn and generalize from three minutes of data:
  Physics-constrained and uncertainty-aware neural stochastic differential
  equations.
\newblock {\em arXiv preprint arXiv:2306.06335}, 2023.

\bibitem{Neary23}
Cyrus Neary and Ufuk Topcu.
\newblock Compositional learning of dynamical system models using
  port-hamiltonian neural networks.
\newblock In {\em Learning for Dynamics and Control Conference}, pages
  679--691. PMLR, 2023.

\bibitem{Hersch08}
Micha Hersch, Florent Guenter, Sylvain Calinon, and Aude Billard.
\newblock Dynamical system modulation for robot learning via kinesthetic
  demonstrations.
\newblock {\em IEEE Transactions on Robotics}, 24:1463--1467, 01 2008.

\bibitem{Levinson11}
Jesse Levinson, Jake Askeland, Jan Becker, Jennifer Dolson, David Held, Soeren
  Kammel, J~Zico Kolter, Dirk Langer, Oliver Pink, Vaughan Pratt, et~al.
\newblock Towards fully autonomous driving: Systems and algorithms.
\newblock In {\em 2011 IEEE intelligent vehicles symposium (IV)}, pages
  163--168. IEEE, 2011.

\bibitem{Lucas20}
Houston Lucas and Richard Kelley.
\newblock Generating control policies for autonomous vehicles using neural
  odes.
\newblock In {\em ICLR 2020 Workshop on Integration of Deep Neural Models and
  Differential Equations}, 2019.

\bibitem{Laurie23}
Mark Laurie and James Lu.
\newblock Explainable deep learning for tumor dynamic modeling and overall
  survival prediction using neural-ode.
\newblock {\em npj Systems Biology and Applications}, 9(1):58, 2023.

\bibitem{Lu21}
James Lu, Kaiwen Deng, Xinyuan Zhang, Gengbo Liu, and Yuanfang Guan.
\newblock Neural-ode for pharmacokinetics modeling and its advantage to
  alternative machine learning models in predicting new dosing regimens.
\newblock {\em Iscience}, 24(7), 2021.

\bibitem{Terakawa24}
Shunpei Terakawa and Takaharu Yaguchi.
\newblock Modeling error and nonuniqueness of the continuous-time models
  learned via runge--kutta methods.
\newblock {\em Mathematics}, 12(8):1190, 2024.

\bibitem{Zhu22}
Aiqing Zhu, Pengzhan Jin, Beibei Zhu, and Yifa Tang.
\newblock On numerical integration in neural ordinary differential equations.
\newblock In {\em International Conference on Machine Learning}, pages
  27527--27547. PMLR, 2022.

\bibitem{Calvo94}
MP~Calvo, A~Murua, and JM~Sanz-Serna.
\newblock Modified equations for odes.
\newblock {\em Contemporary Mathematics}, 172:63--63, 1994.

\bibitem{Zhu24}
Aiqing Zhu, Sidi Wu, and Yifa Tang.
\newblock Error analysis based on inverse modified differential equations for
  discovery of dynamics using linear multistep methods and deep learning.
\newblock {\em SIAM Journal on Numerical Analysis}, 62(5):2087--2120, 2024.

\bibitem{Keller21}
Rachael~T Keller and Qiang Du.
\newblock Discovery of dynamics using linear multistep methods.
\newblock {\em SIAM Journal on Numerical Analysis}, 59(1):429--455, 2021.

\bibitem{Du22}
Qiang Du, Yiqi Gu, Haizhao Yang, and Chao Zhou.
\newblock The discovery of dynamics via linear multistep methods and deep
  learning: error estimation.
\newblock {\em SIAM Journal on Numerical Analysis}, 60(4):2014--2045, 2022.

\bibitem{Koopman31}
Bernard~O Koopman.
\newblock Hamiltonian systems and transformation in hilbert space.
\newblock {\em Proceedings of the National Academy of Sciences},
  17(5):315--318, 1931.

\bibitem{mezic04}
Igor Mezi{\'c} and Andrzej Banaszuk.
\newblock Comparison of systems with complex behavior.
\newblock {\em Physica D: Nonlinear Phenomena}, 197(1-2):101--133, 2004.

\bibitem{mezic05}
Igor Mezi{\'c}.
\newblock Spectral properties of dynamical systems, model reduction and
  decompositions.
\newblock {\em Nonlinear Dynamics}, 41:309--325, 2005.

\bibitem{Brunton16}
Steven~L Brunton, Joshua~L Proctor, and J~Nathan Kutz.
\newblock Discovering governing equations from data by sparse identification of
  nonlinear dynamical systems.
\newblock {\em Proceedings of the national academy of sciences},
  113(15):3932--3937, 2016.

\bibitem{Messenger21}
Daniel~A. Messenger and David~M. Bortz.
\newblock Weak {SINDy}: Galerkin-based data-driven model selection.
\newblock {\em Multiscale Modeling \& Simulation}, 19(3):1474--1497, 2021.

\bibitem{Reinbold20}
Patrick A.~K. Reinbold, Daniel~R. Gurevich, and Roman~O. Grigoriev.
\newblock Using noisy or incomplete data to discover models of spatiotemporal
  dynamics.
\newblock {\em Physical Review E}, 101(1):010203, 2020.

\bibitem{Chen18}
Ricky~TQ Chen, Yulia Rubanova, Jesse Bettencourt, and David~K Duvenaud.
\newblock Neural ordinary differential equations.
\newblock {\em Advances in neural information processing systems}, 31, 2018.

\bibitem{Xie19}
Xuping Xie, Guannan Zhang, and Clayton~G Webster.
\newblock Non-intrusive inference reduced order model for fluids using deep
  multistep neural network.
\newblock {\em Mathematics}, 7(8):757, 2019.

\bibitem{Raissi18}
Maziar Raissi, Paris Perdikaris, and George~Em Karniadakis.
\newblock Multistep neural networks for data-driven discovery of nonlinear
  dynamical systems.
\newblock {\em arXiv preprint arXiv:1801.01236}, 2018.

\bibitem{Mardt18}
Andreas Mardt, Luca Pasquali, Hao Wu, and Frank No{\'e}.
\newblock Vampnets for deep learning of molecular kinetics.
\newblock {\em Nature communications}, 9(1):5, 2018.

\bibitem{Vlachas18}
Pantelis~R Vlachas, Wonmin Byeon, Zhong~Y Wan, Themistoklis~P Sapsis, and
  Petros Koumoutsakos.
\newblock Data-driven forecasting of high-dimensional chaotic systems with long
  short-term memory networks.
\newblock {\em Proceedings of the Royal Society A: Mathematical, Physical and
  Engineering Sciences}, 474(2213):20170844, 2018.

\bibitem{Wehmeyer18}
Christoph Wehmeyer and Frank No{\'e}.
\newblock Time-lagged autoencoders: Deep learning of slow collective variables
  for molecular kinetics.
\newblock {\em The Journal of chemical physics}, 148(24), 2018.

\bibitem{Yeung19}
Enoch Yeung, Soumya Kundu, and Nathan Hodas.
\newblock Learning deep neural network representations for koopman operators of
  nonlinear dynamical systems.
\newblock In {\em 2019 American Control Conference (ACC)}, pages 4832--4839.
  IEEE, 2019.

\bibitem{Bertalan19}
Tom Bertalan, Felix Dietrich, Igor Mezi{\'c}, and Ioannis~G Kevrekidis.
\newblock On learning hamiltonian systems from data.
\newblock {\em Chaos: An Interdisciplinary Journal of Nonlinear Science},
  29(12), 2019.

\bibitem{Kolter19}
J~Zico Kolter and Gaurav Manek.
\newblock Learning stable deep dynamics models.
\newblock {\em Advances in neural information processing systems}, 32, 2019.

\bibitem{Qin19}
Tong Qin, Kailiang Wu, and Dongbin Xiu.
\newblock Data driven governing equations approximation using deep neural
  networks.
\newblock {\em Journal of Computational Physics}, 395:620--635, 2019.

\bibitem{Hu22}
Pipi Hu, Wuyue Yang, Yi~Zhu, and Liu Hong.
\newblock Revealing hidden dynamics from time-series data by odenet.
\newblock {\em Journal of Computational Physics}, 461:111203, 2022.

\bibitem{Greydanus19}
Samuel Greydanus, Misko Dzamba, and Jason Yosinski.
\newblock Hamiltonian neural networks.
\newblock {\em Advances in neural information processing systems}, 32, 2019.

\bibitem{Chen19}
Zhengdao Chen, Jianyu Zhang, Martin Arjovsky, and L{\'e}on Bottou.
\newblock Symplectic recurrent neural networks.
\newblock {\em arXiv preprint arXiv:1909.13334}, 2019.

\bibitem{Rico93}
Ramiro Rico-Martinez and Ioannis~G Kevrekidis.
\newblock Continuous time modeling of nonlinear systems: A neural network-based
  approach.
\newblock In {\em IEEE International Conference on Neural Networks}, pages
  1522--1525. IEEE, 1993.

\bibitem{Oussar01}
Yacine Oussar and G{\'e}rard Dreyfus.
\newblock How to be a gray box: dynamic semi-physical modeling.
\newblock {\em Neural networks}, 14(9):1161--1172, 2001.

\bibitem{LeVeque07}
Randall~J LeVeque.
\newblock {\em Finite difference methods for ordinary and partial differential
  equations: steady-state and time-dependent problems}.
\newblock SIAM, 2007.

\bibitem{nyquist1928certain}
Harry Nyquist.
\newblock Certain topics in telegraph transmission theory.
\newblock {\em Transactions of the American Institute of Electrical Engineers},
  47(2):617--644, 1928.

\bibitem{Lu25}
Bing-Ze Lu and Richard Tsai.
\newblock Numerical artifacts in learning dynamical systems.
\newblock {\em arXiv preprint arXiv:2507.14491}, 2025.

\bibitem{Ariel09}
Gil Ariel, Bj{\"o}rn Engquist, Heinz-Otto Kreiss, and Richard Tsai.
\newblock Multiscale computations for highly oscillatory problems.
\newblock In {\em Multiscale modeling and simulation in science}, pages
  237--287. Springer, 2009.

\bibitem{dahlquist1982numerical}
Germund Dahlquist, Lennart Edsberg, Gunilla Sk{\"o}llermo, and Gustaf
  S{\"o}derlind.
\newblock Are the numerical methods and software satisfactory for chemical
  kinetics?
\newblock In {\em Numerical Integration of Differential Equations and Large
  Linear Systems: Proceedings of two Workshops Held at the University of
  Bielefeld Spring 1980}, pages 149--164. Springer, 1982.

\end{thebibliography}
\end{document}